\documentclass[11pt,reqno]{amsart}
\usepackage{amsthm,amssymb,amsmath,euscript}
\usepackage[pagewise]{lineno}
\usepackage{enumitem}
\usepackage{hyperref}
\usepackage{fix-cm}
\newtheorem{theorem}{Theorem}
\newtheorem*{theorem*}{Theorem}
\newtheorem{lemma}{Lemma}

\newtheorem{proposition}{Proposition}

\theoremstyle{definition}

\newtheorem{definition}{\bf Definition}
\newtheorem*{definition*}{\bf Definition}

\newtheorem{example}{\bf Example}

\newtheorem{remark}{\bf Remark}
\newtheorem*{remark*}{\bf Remark}

\newtheorem*{example*}{\bf Example}

\newcommand{\loc}{{\rm loc}}
\newcommand{\supp}{{\rm sprt\,}}

\newtheorem*{theoremA}{Theorem A}

\numberwithin{equation}{section}

\expandafter\def\expandafter\normalsize\expandafter{%
    \normalsize
    \setlength\abovedisplayshortskip{12pt}
    \setlength\belowdisplayshortskip{12pt}
}

\begin{document}

\fontsize{10.5pt}{4.5mm}\selectfont

\title[Particle systems with singular interactions]{On particle systems and critical strengths of general singular interactions}

\author{D.\,Kinzebulatov}

\email{damir.kinzebulatov@mat.ulaval.ca}

\address{Universit\'{e} Laval, D\'{e}partement de math\'{e}matiques et de statistique, Qu\'{e}bec, QC, Canada}



\thanks{The research of the author is supported by NSERC grant (RGPIN-2024-04236).}

\begin{abstract}
For finite interacting particle systems with strong repulsing-attracting or general interactions, we prove global weak well-posedness almost up to the critical threshold of the strengths of attracting interactions (independent of the number of particles), and establish other regularity results, such as a heat kernel bound in the regions where strongly attracting particles are close to each other. 
Our main analytic instruments are a variant of De Giorgi's method in $L^p$ with appropriately chosen large $p$, and an abstract desingularization theorem.

%\medskip
%
%Pour les syst\`{e}mes finis de particules avec des fortes interactions attractives-repulsives ou g\'{e}n\'{e}rales, on montre l'existence globale et l'unicit\'{e} essentiellement jusqu'\`{a} valeur critique de l'intensit\'{e} des interactions attractives (ind\'{e}pendant du nombre des particules) et obtient d'autres r\'{e}sultats sur la r\'{e}gularit\'{e} du syst\`{e}me, y compris une estimation sur le noyau de la chaleur dans les r\'{e}gions o\`{u} les particules sont proches les unes aux autres. Nos instruments analytiques principaux sont la m\'{e}thode de De Giorgi dans $L^p$ avec $p$ suffisamment grand et un th\'{e}or\`{e}me abstrait sur la desingularisation.

\end{abstract}

\maketitle

\tableofcontents

\section{Introduction}

The paper is concerned with well-posedness and other properties of $N$-particle system
\begin{align}
\label{syst2}
dX_i  = & - \frac{1}{N}\sum_{j=1, j \neq i}^N K_{ij}(X_i-X_j)dt + M_i(X_i)dt  + \sqrt{2}dB_i, \qquad X_i(0)=x_i \in \mathbb R^d \\[2mm]
& \text{$\{B_i(t)\}_{t \geq 0}$ are independent  $d$-dimensional Brownian motions}, \notag
\end{align}
under broad assumptions on singular (i.e.\,locally unbounded) interaction kernels and drifts $K_{ij},M_i:\mathbb R^d \rightarrow \mathbb R^d$ ($i=1,\dots,N$) that can have repulsion/attraction structure or can be of general form. Our primary goal is to obtain conditions on $M_i$ and $K_{ij}$ that 

\smallskip

1) reach blow up effects, and

\smallskip

2) withstand the passage to the limit $N \rightarrow \infty$. 

\medskip

Interacting particle systems of type \eqref{syst2} arise in many physical and biological models \cite{AKR,C,CP,CPZ,FJ,FT, HRZ,JL,KR,T,To}. Many of these models require one to deal with the interactions that are not only singular but are so singular that they reach blow up effects: replacing $K_{ij}$ in \eqref{syst2}  by $(1+\varepsilon)K_{ij}$, i.e.\,increasing the strength of interactions by factor $1+\varepsilon$, can lead to a collapse in the well-posedness of \eqref{syst2} even if $\varepsilon>0$ is small. That is, the particles start to collide in finite time with positive probability, and \eqref{syst2} ceases to have a weak solution. 

One of the main questions studied in the present paper is what is the critical threshold value of the strength of general singular interactions that  separates the well-posedness of \eqref{syst2} from a blow up.

Throughout the paper, dimension $d \geq 3$. Important case $d=2$ requires a separate study which we plan to carry out elsewhere.

To illustrate the blow up effect in particle systems, and to formalize the notion of the ``strength of interactions'', consider a particle system with the model singular attracting kernel  \eqref{hardy}:
\begin{equation}
\label{model_syst}
dX_i=-\frac{1}{N}\sum_{j=1, j \neq i}^N \sqrt{\kappa}\frac{d-2}{2}\frac{X_i-X_j}{|X_i-X_j|^2}dt + \sqrt{2}dB_i,
\end{equation}
where $\kappa$ measures the strength of attraction between the particles. (It is convenient for us to include factor $\frac{d-2}{2}$ in the coefficient in \eqref{hardy} because we are going to use Hardy's inequalities, see Example \ref{ex1}(2). The kind of Hardy inequalities that we need are not valid in two dimensions.)

\begin{enumerate}[label=(\alph*)]

\item  In the two-particle case $N=2$, a simple argument shows that for $$\kappa>16\bigg(\frac{d}{d-2}\bigg)^2$$ and $X_1(0)=X_2(0)$ the particle system \eqref{model_syst} does not have a weak solution. Informally, in the struggle between the drift and the diffusion the former starts to have an upper hand.  If fact, already if $\kappa>16$ the particles collide in finite time with positive probability even if $X_1(0)$, $X_2(0)$ are uniformly distributed e.g.\,in a cube, see \cite{BFGM} for detailed proof. 
On the other hand, if $$\kappa<16,$$ then  \eqref{model_syst} has a global in time weak solution for any initial configuration $X_1(0)$, $X_2(0) \in \mathbb R^d$. The latter can be seen from the results of  the present paper.

\medskip

\item 
The weak well-posedness and the blow-up effects for the two-dimensional counterpart of \eqref{model_syst}
\begin{equation}
\label{model2}
dX_i=-\frac{1}{N}\sum_{j=1, j \neq i}^N \sqrt{\kappa}\frac{X_i-X_j}{|X_i-X_j|^2}dt + \sqrt{2}dB_i, \quad N \geq 2,
\end{equation}
 were studied in detail, among other problems connected to the Keller-Segel model of chemotaxis, in \cite{CP,FJ}.

\medskip

\item The density of the formal invariant measure of \eqref{model_syst}
$$
\psi(x)=\prod_{1 \leq i<j \leq N} |x_i-x_j|^{-\sqrt{\kappa}\frac{d-2}{2}\frac{1}{N}}
$$
is locally summable if and only if $\kappa<16(\frac{d}{d-2})^2$. Also, as $\kappa$ reaches and surpasses $\kappa=16$, $\psi$ ceases to be in $W^{2,1}_{\loc}$.

\medskip

Another analytic fact that  suggests that the singularities of the drift in \eqref{model_syst} are critical for any $N \geq 2$, i.e.\,$\kappa$ in general cannot be too large, is the estimates on the constant in the many-particle Hardy inequality \eqref{multi_hardy} due to \cite{HHLT}. We employ their result in the proof of Theorem \ref{thmK1}(\textit{iii}).

\medskip

\item The blow up effects are observed for the Keller-Segel model (here in the parabolic-elliptic form):
$$
\left\{
\begin{array}{l}
\partial_t \rho -\Delta \rho + \sqrt{\kappa}\, {\rm div}\,(\rho \nabla v)=0, \quad \rho(0,\cdot)=\rho_0(\cdot), \\
-\Delta v=\rho,
\end{array}
\right.
$$
where $\rho$ is the population density and $v$ is the chemical density \cite{CP,CPZ,FJ,FT,JL,T}, see also references therein. Of course, $\int_{\mathbb R^d} \rho_0(x)dx=1$ propagates to $\int_{\mathbb R^d} \rho(t,x)dx=1$ for all $t>0$. Solving the elliptic equation, one obtains the expression for the drift:
\begin{equation}
\label{riesz}
\nabla v=-K_1 \ast \rho, \quad K_1(y)=c_d\frac{y}{|y|^{d}}, \quad c_d>0 
\end{equation}
(we can further redefine $\kappa$ to have $c_d=1$).
The resulting  McKean-Vlasov PDE
\begin{equation}
\label{mv2}
\partial_t \rho -\Delta \rho - \sqrt{\kappa}\, {\rm div}\,(\rho (K_1 \ast \rho))=0
\end{equation}
is comparable to \eqref{model_syst} only in dimension $d=2$, where it does indeed arise at the mean field limit of particle system \eqref{model2} as $N \rightarrow \infty$ \cite{CP, FJ, FT,T}. The fact that one needs $d=2$ is at the first regard somewhat disappointing for us, see, however, the end of remark (iv) below. 
It should be added that there is a significant recent progress in understanding the behaviour of \eqref{mv2}, \eqref{model2} around and at the critical threshold $\kappa=16$ both at the level of PDEs and at the SDE level, see \cite{FT,T}. At the PDE level there are other important results on admissible strengths of critical singular interactions and the McKean-Vlasov equation, including recent result in \cite{BJW}, see Section \ref{lit_sect}.

\medskip

The proofs of the results described in (b) and in (d) depend on the special form of the Riesz interaction kernels in \eqref{model_syst} and \eqref{riesz}.
\end{enumerate}

\medskip

Despite the prominent role played in applications by the Riesz interaction kernels, there are many other situations where one needs to handle more general critical-order singular interactions. These are in the focus of the present paper. In this case, one can no longer exploit the special structure of the interaction kernel in \eqref{model_syst}. It turns out that one can still cover large portions (if not most) of the ranges of admissible strengths $\kappa$ of interactions and establish, in particular, global weak well-posedness of particle system \eqref{syst2}, which however requires us to use some deep methods in the theory of elliptic and parabolic PDEs. This is done in Theorems \ref{thmK1_0} and \ref{thmK1}(\textit{i})-(\textit{iii}). For the model singular attracting kernel in \eqref{hardy} we prove a necessarily non-Gaussian upper bound on the heat kernel ($\equiv$ the density of the law) of particle system \eqref{syst2}, see Theorem \ref{thmK1}(\textit{iv}). We believe that this bound is optimal in the regions where the particles are close to each other. 

\medskip

In Section \ref{lit_sect} we comment on the existing literature on particle systems with general singular interactions.

\medskip

We focus on weak solutions and exploit the connection of \eqref{syst2} to the Kolmogorov backward equation
\begin{equation}
\label{kolm_eq}
\big(\partial_t-\Delta_x + \frac{1}{N}\sum_{i=1}^N \sum_{j=1,  j \neq i}^N K_{ij}(x_i-x_j) \cdot \nabla_{x_i}\big)v=0, \quad v(0,\cdot)=f(\cdot),
\end{equation}
i.e.\,$v(t,x_1,\dots,x_N)=\mathbf{E}_{X_1(0)=x_1,\dots,X_N(0)=x_N}[f(X_1(t),\dots,X_N(t))]$.
Our main instruments in this paper are:

\medskip

-- De Giorgi's method, but ran in $L^p$ with $p$ chosen sufficiently large, in order to relax the assumptions on the strength of interactions.

\medskip

-- A ``desingularization theorem'' obtained,  using ideas of Nash, in the paper with Sem\"{e}nov and Szczypkowski \cite{KiSSz}.

\medskip

We impose conditions on the interaction kernels $K_{ij}$ stated in the form of quadratic form inequalities, see \eqref{fbd_} and \eqref{div_fbd}, \eqref{mfbd_}. The reason for this is two-fold. First, as we explain below, such conditions are ultimate in the sense that they provide a minimal PDE theory for the Kolmogorov equation \eqref{kolm_eq}; at the same time, there is a well developed machinery that allows to verify these conditions, see Example \ref{ex1}. Second, these conditions provide a natural setting for controlling the strength of interactions when $N$ is large, see discussion after Theorem A.

\medskip

Our class of general interaction kernels is given by Definition \ref{def1}. We postpone the definition of our class of repulsing-attracting interaction kernels until the next section.

\medskip

Let $L^p=L^p(\mathbb R^d)$ denote the Lebesgue spaces endowed with the norm $\|\cdot\|_p$. Let $W^{1,p}$ be the corresponding Sobolev spaces. Denote by $[L^p]^d$ the space of vector fields $\mathbb R^d \rightarrow \mathbb R^d$ with entries in $L^p$.  

\begin{definition}
\label{def1}
A Borel measurable vector field $K:\mathbb R^d \rightarrow \mathbb R^d$ is said to be form-bounded if $K \in [L^2_{\loc}]^d$ and there exists constant $\kappa$ (``form-bound of $K$'') such that
\begin{equation}
\label{fbd_}
\|K\varphi\|_2^2 \leq \kappa \|\nabla \varphi\|_2^2 + c_\kappa\|\varphi\|_2^2 \quad \forall\,\varphi \in W^{1,2}
\end{equation}
for some $c_\kappa<\infty$. 
\end{definition}

In other words, $$|K|^2 \leq \kappa(-\Delta) + c_\kappa$$
in the sense of quadratic forms in $L^2$, which yields upon applying Cauchy-Schwarz inequaltiy
$$
K \cdot \nabla \leq \sqrt{\kappa}(-\Delta) + \frac{c_\kappa}{2\sqrt{\kappa}}.
$$
We abbreviate \eqref{fbd_} as $K \in \mathbf{F}_\kappa.$ 

\medskip

The class $\mathbf{F}_\kappa$ is a well known in the PDE literature condition on first-order perturbation in elliptic and parabolic operators. Moreover, unlike, say, the optimal Lebesgue class $|K_{ij}| \in L^d$, a larger class $K_{ij} \in \mathbf{F}_\kappa$ is in a sense ultimate from the PDE perspective: assuming $\kappa<(\frac{N}{N-1})^2$, it provides coercivity of the quadratic form of the corresponding Kolmogorov backward operator in $L^2(\mathbb R^{Nd})$:
\begin{equation}
\label{kolm}
\Lambda = - \Delta_x + \frac{1}{N}\sum_{i=1}^N \sum_{j=1,  j \neq i}^N K_{ij}(x_i-x_j) \cdot \nabla_{x_i}, \quad x=(x_1,\dots,x_N) \in \mathbb R^{Nd}.
\end{equation}
Keeping in mind the connection between \eqref{syst2} and the Kolmogorov backward equation, it is natural for us to focus on the assumptions on $K_{ij}$ that provide some ``minimal theory'' of Kolmogorov operator \eqref{kolm}.

\medskip

As the next Example \ref{ex1}(2) shows, the strength of attraction $\kappa$ in model system \eqref{model_syst} is a particular case of the form-bound $\kappa$ in Definition \ref{def1}.

\begin{example}
\label{ex1}
The following are some sufficient conditions for $K \in \mathbf{F}_\kappa$ stated in elementary terms:

\begin{enumerate}[label=\arabic*.]

\item 
\begin{equation}
\label{Lp_incl}
|K| \in L^d \quad \Rightarrow \quad K \in \mathbf{F}_\kappa,
\end{equation}
with $\kappa$ that can be chosen arbitrarily small. 
(As a consequence, we obtain that the blow up effects described above in (a)-(c) can not be observed if we restrict our attention to $|K| \in L^d$. That said, the situation with the Lebesgue class drifts in the Keller-Segel model (d) is different because the regularity of the nonlinear drift $K_1 \ast \rho$ also depends on the regularity of the initial condition $\rho_0$ and is improving as $q$ in $\rho_0 \in L^q$ increases.)

Indeed, for every $\varepsilon>0$ we can represent $K=K_1+K_2$ with $\|K_1\|_d<\varepsilon$ and $\|K_2\|_\infty<\infty$. So, we obtain, using the Sobolev embedding theorem,
\begin{align*}
\|K\varphi\|_2^2 
 & \leq 2\|K_1\|_d^2 \|\varphi\|_{\frac{2d}{d-2}}^2 + 2\|K_2\|_\infty^2 \|\varphi\|_2^2 \\
& \leq C_S 2 \|K_1\|_d^2 \|\nabla \varphi\|_2^2 + 2\|K_2\|_\infty^2 \|\varphi\|_2^2,
\end{align*}
hence $K \in \mathbf{F}_{\kappa}$ with $\kappa=C_S 2\varepsilon$ and $c_\kappa=2\|K_2\|_\infty^2$. 

Of course, above $\kappa$ can be chosen arbitrarily small at expense of increasing $c_\kappa$. Although in some questions the value of constant $c_\kappa$ is important (e.g.\,in the study of long term behaviour of solution of \eqref{syst2}), they are not related to the problem of the blow up versus well-posedness in \eqref{syst2}.

\medskip

\item {(Critical point singularities)\;}
The model singular interaction kernel 
\begin{equation}
\label{hardy}
K(y)= \pm \sqrt{\kappa}\frac{d-2}{2}\frac{ y }{|y|^2}, \quad y \in \mathbb R^d,
\end{equation}
($+$ is the attraction, $-$ is the repulsion)
is  in $\mathbf{F}_\kappa$ with $c_\kappa=0$. This is a re-statement of the well known Hardy inequality:
$$
\frac{(d-2)^2}{4}\big\||y|^{-1}\varphi\big\|_2^2 \leq \|\nabla \varphi\|_2^2, \quad \forall\,\varphi \in W^{1,2}(\mathbb R^d).
$$
This inequality is sharp: $K \not \in \mathbf{F}_{\kappa'}$ for any $\kappa'<\kappa$ regardless of the value of $c_{\kappa'}$.

\medskip 

A finer example is given by the weighted Hardy inequality of \cite{HL}. Fix $0 \leq \Phi \in L^q(S^{d-1})$ for some $q \geq \frac{2(d-2)^2}{2(d-1)}+1,$ where $S^{d-1}$ is the unit sphere in $\mathbb R^d$. If
$$
|K(y)|^2 \leq \kappa \frac{(d-2)^2}{4} c\frac{\Phi(y/|y|)}{|y|^2}, \qquad \text{where $c:=\frac{|S^{d-1}|^{\frac{1}{q}}}{\|\Phi\|_{L^q(S^{d-1})}}$},
$$
then $K \in \mathbf{F}_\kappa$ with $c_\kappa=0$. Using this example, one can e.g.\,cut off a wedge in the model interaction kernel \eqref{hardy} while still controlling the value of the strength of interaction $\kappa$.

\medskip

\item {(Weak $L^d$ class interaction kernels)}\; More generally, vector fields $K$ in $L^{d,\infty}$, i.e.\,such that
\begin{equation}
\label{weak_Ld}
\|K\|_{d,\infty}:=\sup_{s>0}s|\{y \in \mathbb R^d: |K(y)|>s\}|^{1/d}<\infty
\end{equation}
are in $\mathbf{F}_\kappa$ with 
$\sqrt{\kappa}=\|K\|_{d,\infty} |B_1(0)|^{-\frac{1}{d}} \frac{2}{d-2}$, see \cite{KPS}. When applied to \eqref{hardy}, this inclusion gives the constant in Hardy's inequality.

\medskip

\item{(Morrey class interaction kernels)\;} The scaling-invariant Morrey class $M_{2+\varepsilon}$, with $\varepsilon>0$ fixed arbitrarily small, consists of vector fields $K \in [L^{2+\varepsilon}_{\loc}]^d$ such that
\begin{equation}
\label{morrey}
\|K\|_{M_{2+\varepsilon}}:=\sup_{r>0, y \in \mathbb R^d} r\biggl(\frac{1}{|B_r(y)|}\int_{B_r(y)}|K|^{2+\varepsilon}dy \biggr)^{\frac{1}{2+\varepsilon}}<\infty.
\end{equation}
By one of the results in \cite{F}, if
$
K \in M_{2+\varepsilon}$, then $K \in \mathbf{F}_\kappa$ with $\kappa=c\|K\|_{M_{2+\varepsilon}}$ for a constant $c=c(d,\varepsilon)$ that depends on the constants in some classical inequalities of Harmonic Analysis.

This sufficient condition for form-boundedness can be further refined by considering the Chang-Wilson-Wolff class \cite{CWW}: $K \in [L^2_{\loc}]^d$ satisfies
$$
\|K\|_{\xi}:=\sup_{r>0, y \in \mathbb R^d} \biggl(\frac{1}{|B_r(y)|}\int_{B_r(y)} |K|^2\, r^2 \xi\big(|K|^2\,r^2 \big) dy \biggr)^{\frac{1}{2}}<\infty,
$$
where 
$\xi:[0,\infty[ \rightarrow [1,\infty[$ is an increasing function such that
$
\int_1^\infty \frac{ds}{s\xi(s)}<\infty.
$

On the other hand, a simple argument with cutoff functions shows that the class of form-bounded vector fields $\mathbf{F}_\kappa$ (say, $c_\kappa=0$) is contained in the Morrey class $M_2$.

It should be added that the cited results in \cite{F,CWW} appeared as a part of broader efforts to find necessary and sufficient conditions for form-boundedness stated in elementary terms (in the context related to study of Schr\"{o}dinger operators with singular potentials, including self-adjointness, estimates on
the number of bound states, resolvent convergence). 

\medskip

\item {(Hypersurface singularities)\;}Any interaction kernel $K$ satisfying
\begin{equation}
\label{hyp}
|K(y)|^2=C\frac{c(y)\mathbf{1}_{\{\frac{1}{2} \leq |y| \leq \frac{3}{2}\}}}{\big| |y|-1\big|(-\ln\big||y|-1\big|)^\beta}, \quad \beta>1.
\end{equation}
is form-bounded, which can be seen from the previous example by arguing locally. (Note that the components of $K$ are not in $L^{2+\epsilon}_{\loc}$ for any $\epsilon>0$; one can compare this with example (i).)

\end{enumerate}
\end{example}

The class of form-bounded vector fields $\mathbf{F}_\kappa$ is closed with respect to addition and multiplication by functions from $L^\infty$ (up to change of $\kappa$ and $c_\kappa$). So, one can combine the previous examples.

\medskip

 Our main results, stated briefly, are as follows (omitting for now the repulsing-attracting interaction kernels):

\begin{theoremA}
{\rm (\textit{i})} {\it Let $$K_{ij} \in \mathbf{F}_\kappa, \quad \kappa < 4 \biggl(\frac{N}{N-1} \biggr)^2.$$ Then there exists a strong Markov family of martingale solutions of particle system \eqref{syst2} that delivers a unique (in appropriate sense) weak solution to Cauchy problem for the Kolmogorov backward equation \eqref{kolm_eq}. 

\medskip

{\rm (\textit{ii})} If $\kappa$ is smaller than a certain explicit constant $c_{d,N}$, then, moreover, conditional weak uniqueness and strong existence hold for \eqref{syst2}.}

\medskip

{\rm (\textit{iii})} {\it In the special case $$K_{ij}(y)=\sqrt{\kappa}\frac{d-2}{2}\frac{y}{|y|^2}+K_{0,ij}(y), \quad K_{0,ij} \in \mathbf{F}_{\nu},$$ if the strength of attraction satisfies only $\kappa<16$ and $\nu$ is sufficiently small, then the first assertion in {\rm(\textit{i})} still holds. Moreover, in the model attracting case 
$$
K_{ij}(y)=\sqrt{\kappa}\frac{d-2}{2}\frac{y}{|y|^2}, \quad \kappa<16
$$
the heat kernel of \eqref{syst2} satisfies an explicit non-Gaussian upper bound that we believe to be optimal in the regions where the particles are close to each other. (The constant $16$ can be somewhat improved in low dimensions.)}
\end{theoremA}

For detailed statements, see Theorems \ref{thmK1_0} and \ref{thmK1}. 

\medskip

The improvement in the assumptions on $\kappa$ in (\textit{iii}) is due to the use of many-particle Hardy inequality \eqref{multi_hardy}. In assertion (\textit{ii}) constant $c_{d,N} \downarrow 0$ as the number of particles $N \uparrow \infty$, which, of course, is not what we are after in this paper. But we included this assertion anyway for the sake of completeness and to demonstrate that as the strength of interactions becomes smaller the theory of \eqref{syst2} becomes more detailed.

\medskip

We address the problem of well-posedness of stochastic particle system \eqref{syst2}  directly, by rewriting \eqref{syst2} as  SDE
\begin{equation}
\label{Z_eq}
dZ=-b(Z)dt + \sqrt{2}dB, \quad B \text{ is a Brownian motion in $\mathbb R^{Nd}$}
\end{equation}
with $Z=(X_1,\dots,X_N)$ and drift
$$b=(b_1,\dots,b_N):\mathbb R^{Nd} \rightarrow \mathbb R^{Nd},$$ 
\begin{equation}
\label{b_def0_}
\text{where } b_i(x):= \frac{1}{N}\sum_{j=1, j \neq i}^N K_{ij}(x_i-x_j), \quad x=(x_1,\dots,x_N) \in \mathbb R^{Nd}, \quad 1 \leq i \leq N,
\end{equation}
and then applying results on  well-posedness for SDEs with general drifts, in particular, our Theorem \ref{markov_thm} below.
Until recently, the results on general singular SDEs
could not compete, in terms of the admissible point singularities of the drift, with the results on particle systems with singular interactions.
However, in the past few years, there was a substantial progress in proving weak and strong well-posedness of SDE \eqref{Z_eq} with general drift $b$, which now can have critical-order singularities (i.e.\,reach blow up effects), see \cite{KiS_brownian,KiS_sharp,KiM_strong,Ki_Morrey,Kr1,Kr2,Kr3,RZ}.
That said, to apply these results to particle system \eqref{syst2} when the number of particles is large in a way that would allow to control the strength of interactions (measured, in our case, by constant $\kappa$), one needs to keep track of the strength of the singularities of the drift $b$ (in our case, measured by its own form-bound with respect to the Laplacian in $\mathbb R^{Nd}$). In Lemma \ref{p_1} we show
that if $K_{ij} \in \mathbf{F}_\kappa (\mathbb R^d)$, then $b$ satisfies
\begin{equation}
\label{b_fbd}
\left\{
\begin{array}{l}
b \in \mathbf{F}_\delta(\mathbb R^{Nd}) \\
\text{with } \delta=\frac{(N-1)^2}{N^2}\kappa, \quad c_\delta=\frac{(N-1)^2}{N}c_\kappa.
\end{array}
\right.
\end{equation}
(Note that if $c_\kappa=0$, as is the case for \eqref{hardy}, then $c_\delta=0$.)
Thus, we obtain our Theorem \ref{thmK1_0} and Theorem \ref{thmK1}(\textit{i})-(\textit{iii}) for particle system \eqref{syst2} from our  results on the general singular SDE \eqref{Z_eq} in $\mathbb R^{Nd}$, which are Theorems \ref{markov_thm}, \ref{unique_thm}.

\medskip

In this approach, it is crucial that the assumptions on the form-bound $\delta$ of drift $b$ in Theorem \ref{markov_thm} stay dimension-independent, so that when Theorem \ref{markov_thm} is applied to particle system \eqref{syst2} the resulting assumption on $\kappa$ would not depend on the number of particles $N$ (or, rather, would tend to a strictly positive value as the number of particles goes to infinity). This is achieved by means of De Giorgi's method ran in $L^p$ which allows us to ``decouple'' the proof of the tightness estimate  needed to establish the existence of a martingale solution (cf.\,\eqref{tight_i}) from any strong gradient bounds on solutions of the corresponding elliptic or parabolic equations that, generally speaking, introduce a dependence on the dimension in the assumptions on the form-bound of $b$.

\medskip

Running De Giorgi's method in $L^p$ with $p \gg 2$ allows to maximize admissible values of the form bounds/strengths of interactions in particle system \eqref{syst2}. 
At the level of strongly continuous semigroups, the observation that working in $L^p$ with $p$ large allows to relax the assumptions on the form-bound of the drift was made even earlier in \cite{KS}. 

\medskip

By the way, in Remark \ref{Orlicz_rem} we discuss the theory of the backward Kolmogorov equation \eqref{kolm_eq} in the case when the strength of the interactions reaches $\kappa$ the borderline value, which requires us to work in the Orlicz space with ``critical'' gauge function $\cosh -1$ (that is, in some sense, a limit of $L^p$ as $p \uparrow \infty$). 

\medskip

The strong existence in Theorem \ref{thmK1_0}(\textit{iv}) (or in Theorem A(\textit{ii})) follows from the result in \cite{KiM_strong} whose proof, in turn, is a modification of the method of R\"{o}ckner-Zhao \cite{RZ2}.

\medskip

Theorems \ref{markov_thm}, \ref{unique_thm} on the general singular SDE \eqref{Z_eq} are of interest on their own. 

\medskip

Theorem \ref{markov_thm} deals with the existence and uniqueness of a strong Markov family of martingale solutions of SDE \eqref{Z_eq}. In a number of ways, Theorem \ref{markov_thm} continues the paper with Sem\"{e}nov \cite{KiS_sharp}, see further discussion in Section \ref{thm1_rem}. 

\medskip

Theorem \ref{unique_thm} deals with conditional weak uniqueness for \eqref{Z_eq}, i.e.\,the uniqueness among weak solutions satisfying a rather natural condition (Krylov-type bound). The main novelty of Theorem \ref{unique_thm} is related to condition \eqref{C_2} that takes into account the repulsion-attraction structure of the drift. However, in its present form this condition, when applied to particle system \eqref{syst2}, imposes not so natural conditions on the repulsing part of the interactions (i.e.\,admissible strength of repulsion depends on the number of particles, see the last comment before Section \ref{thm1_rem}), so for now we leave this result at the level of general singular SDEs.

\subsection{About the proofs}
\label{about_sect}

The analytic core of the paper are Theorems \ref{thm1}, \ref{thm2} and \ref{thm_grad} from which Theorems \ref{markov_thm} and \ref{unique_thm} for general singular SDEs follow. 

\medskip

In Theorem \ref{thm1} we prove H\"{o}lder continuity of solutions to the elliptic counterpart of the Kolmogorov backward equation \eqref{kolm_eq}, i.e.\,$(\lambda-\Delta + b \cdot \nabla)u=f$, $f \in C_c^\infty$, where $b$ can, in particular, be defined by \eqref{b_def0_}. This is needed to prove the strong Markov property for the martingale solutions in Theorem \ref{markov_thm}. Theorem \ref{thm1} is proved by showing that solution of the elliptic Kolmogorov equation $u$ belongs to appropriate $L^p$ De Giorgi's classes and then following De Giorgi's method. These De Giorgi classes, however,
are somewhat different from the $L^p$ De Giorgi classes found in the literature (cf.\,\cite{G}), i.e.\,they contain the integrals of 
\begin{equation}
\label{nabla_u_p}
|\nabla (u-k)_+^{p/2}|^2, \quad k \in \mathbb R,
\end{equation}
rather than the integrals of $|\nabla (u-k)_+|^p$.

\medskip

Theorem \ref{thm2}, i.e.\,an embedding theorem for a family of
non-homogeneous elliptic Kolmogorov equations that includes 
\begin{equation}
\label{nonhom}
(\lambda-\Delta + b \cdot \nabla)u=|b|f, \quad f \in C_c^\infty, 
\end{equation}
is needed to construct martingale solutions in Theorem \ref{markov_thm}. It also has other uses e.g.\,we apply it in subsequent paper \cite{KiS_feller_4} to construct strongly continuous Feller semigroup with general form-bounded drift with form-bound in the critical range $\delta<4$. 
The proof of Theorem \ref{thm2} also uses De Giorgi's method. Although the assertion of Theorem \ref{thm2} is a global $L^\infty$ estimate on $u$ in terms of a certain $L^p$ norm of the right-hand side, its proof is local. Otherwise we would have to impose an additional global  condition on $b$ that would be difficult to verify for $b$ given by \eqref{b_def0_}. Also, we will need an intermediate result in the proof of Theorem \ref{thm2} in order to establish a ``separation property'', i.e.\,that $u$ is small far away from the support of $f$.

\medskip

The point of departure of De Giorgi's method is the Caccioppoli inequality. To prove it under the repulsing-attracting form-boundedness type condition of Definitions \ref{def3}, \ref{def2}, we extend the iteration procedure (``Caccioppoli's iterations'')  introduced in an earlier paper with Vafadar \cite{KiV} to the non-homogeneous $L^p$ setting, as is needed to handle weak well-posedness of SDEs.

\medskip

Finally, Theorem \ref{thm_grad}, needed to prove Theorem \ref{unique_thm} on conditional weak uniqueness, contains rather strong gradient bounds on solution of \eqref{nonhom}. Its proof uses a quite ingeniously constructed test function of \cite{KS}, see comments after Theorem \ref{thm_grad}.

\subsection{More on the existing results} 
\label{lit_sect}

(i)~Gradient form interaction kernels
\begin{equation}
\label{grad_form}
K=\nabla V:\mathbb R^d \rightarrow \mathbb R^d
\end{equation}
for some potential $V$ on $\mathbb R^d$
play a crucial role in Statistical Physics. We refer to \cite{AKR,KR}, see also references therein. In particular, in \cite{KR} the authors proved strong well-posedness of the particle system  in  $\mathbb R^{Nd} \setminus \cup_{1 \leq i<j \leq N}\{((x_1,\dots,x_N) \in \mathbb R^{Nd} \mid x_i=x_j\}$ for very singular interaction potentials satisfying some fairly general assumptions (however, excluding purely attracting singular interactions such as the ones in \eqref{model_syst}, covered as a special case by Theorems \ref{thmK1_0}, \ref{thmK1}). For instance, the result in \cite[Sect.\,9.2]{KR} yields strong well-posedness of the particle system for potential
$$
V(x)=|x|^{-10}\biggl(2+\sin \frac{1}{|x|}\biggr).
$$
The corresponding interaction kernel $K=-10|x|^{-12}x \bigl(2+\sin(\frac{1}{|x|})\bigr) -|x|^{-13}x\cos(\frac{1}{|x|})$ oscillates between the repulsion and the attraction as $x$ approaches the origin. The repulsion on average dominates the attraction. Still, our results do not cover such interactions. In fact, although in Theorem \ref{thmK1} our condition on the repulsing part of $K$ is much weaker than the condition on the attracting part  of $K$, it is still a global condition: $({\rm div\,}K)_- \in L^1(\mathbb R^d)+L^\infty(\mathbb R^d)$.

\medskip

See also \cite{C, CP} regarding the Dirichlet form approach to the problem of well-posedness of particle systems with gradient form interactions. 

\medskip

We also mention \cite{BJW} where the authors work at the PDE level on the torus, consider interaction kernels of gradient form with the interaction potential $V$ pointwise  comparable to $\sqrt{\kappa}\frac{d-2}{2}\log |x|$ (which thus includes the attracting kernel in \eqref{hardy}) and, importantly, obtain quantitative estimates on the propagation of chaos for the McKean-Vlasov PDE for all $\kappa<16(\frac{d}{d-2})^2$. 

\smallskip

(ii)~The present paper deals with general singular interactions, i.e.\,not having a particular structure such as gradient form. In particular, we refer to \cite{HRZ, T} where the authors prove, as a part of their results on the propagation of chaos, well-posedness of particle system \eqref{syst2} for interaction kernels $K$ in the sub-critical Ladyzhenskaya-Prodi-Serrin class. Applied to \eqref{syst2} (with $K_{ij}=K$), their condition reads as $$|K| \in L^p+L^\infty \;\;(\text{i.e.\,sum of two functions}), \quad p>d.$$ 
See \cite{C} regarding time-homogeneous critical LPS class 
\begin{equation}
\label{lps}
\tag{LPS}
|K| \in L^d+L^\infty.
\end{equation}

\medskip

The class of form-bounded interactions kernels $\mathbf{F}_\kappa$ is larger than \eqref{lps} and, moreover, contains some interaction kernels that are strictly more singular than the ones in \eqref{lps}, such as \eqref{hardy}. However, here we are not comparing our results with papers \cite{C, HRZ, T} since we do not prove the existence of a mean field limit and its uniqueness. 

\medskip

Let us also make the following two comments regarding the relationship between class \eqref{lps} and class $\mathbf{F}_\kappa$:

\begin{enumerate}

\item[--] One advantage of the Lebesgue scale condition \eqref{lps} is that it is easy to verify.
However, it is not necessarily easy to deal with when one considers particle systems of type \eqref{syst2} for $N$ large. Indeed, if, in order to prove well-posedness of \eqref{syst2} we were to consider this particle system as a special case of general SDE \eqref{Z_eq} in $\mathbb R^{Nd}$, then the well-posedness results on the Lebesgue scale for \eqref{Z_eq} would require $|b| \in L^{q}(\mathbb R^{Nd})+L^\infty(\mathbb R^{Nd})$, $q>Nd$  (see \cite{KR}) or $q=Nd$ (see \cite{BFGM}). Clearly, this severely restricts the class of admissible interaction kernels $K_{ij}=K$ in \eqref{b_def0_}. There is a finer  argument due to \cite{HRZ} that still allows to prove strong well-posedness of \eqref{syst2} for $K \in L^p(\mathbb R^d)+L^\infty(\mathbb R^d)$, $p>d$, regardless of the number of particles $N$, but it requires extra work.

\medskip

On the other hand, form-boundedness handles transition from from-bounded $K_{ij}$ on $\mathbb R^d$ to form-bounded $b$ given by \eqref{b_def0_} on $\mathbb R^{Nd}$ rather effortlessly, see \eqref{b_fbd}. Moreover, crucially for particle systems, it allows to keep track of the values of the form-bounds (=\,strengths of interactions) regardless of the number of particles $N$. (To borrow an expression from \cite{C}, the present work can be viewed as a ``propaganda piece'' for form-boundedness and similar conditions in the context of particle systems and singular SDEs.)

\bigskip

\item[--] Consider drift $b:\mathbb R^d \rightarrow \mathbb R^d$.
If $u$ is a weak solution of the elliptic equation $(\lambda -\Delta  + b \cdot \nabla )u=f$, $\lambda>0$, $f \in C_c^\infty$ with $b \in L^d+L^\infty$ and $u \in W^{1,r}$ (e.g.\,using Theorem \ref{thm_grad}) for $r$ large, then, by H\"{o}lder's inequality,
$$
\Delta u \in L^{\frac{rd}{d+r}}_{\loc}. 
$$ 
However, for $b \in \mathbf{F}_\delta$, one can only say that  $$\Delta u \in L^{\frac{2d}{d+2}}_{\loc}$$ (in fact, one can show that $u \in W^{2,2}$). That is, if $b$ is only form-bounded then there are no $W^{2,p}$ estimates on $u$ for $p$ large.

\end{enumerate}

\medskip

Regarding general singular interactions, let us also mention a model of the dynamics of neuroreceptors considered in \cite{L} where the fact that a neurotransmitter, after it gets attached to a fixed neuroreceptor, prevents other neurotransmitters from entering, is modelled by introducing singular repulsing interactions between neurotransmitters in some regions on space. It is 
thus desirable to be able to handle interaction kernels with critical singularities that 
stay admissible after one multiplies them by indicator functions (so that the interactions can be turned on or turned off depending on the positions of the particles relative to each other and in space), as e.g.\,the class of form-bounded interaction kernels considered in the present work.

\medskip

(iii)~De Giorgi's method was used earlier in the context of singular SDEs in \cite{RZ, ZZ, Zh}. There the authors considered singular drifts arising in the study of 3D Navier-Stokes equations.

\medskip

(iv)~In dimensions $d \geq 3$ one does not obtain the Keller-Segel equation \eqref{mv2} as the mean field limit of particle system \eqref{model_syst} since, evidently, there is a gap between the singularity of kernel $K_1(y)=c_d |y|^{-d}y$ in \eqref{mv2} and the singularity of kernel $K(y)=c_d|y|^{-2}y$ in \eqref{model_syst} (that, we know, is already critical). Nevertheless, it is known in the literature on the Keller-Segel equation \cite{JL,CPZ} that requiring extra regularity of the initial distribution $\rho_0 \in L^{d/2}$, one can extend it to $\rho\in L^\infty(\mathbb R_+,L^{d/2})$, in which case, by Young's inequality, $$(K_1 \ast \rho)(t,\cdot) \in L^d,\quad t \geq 0,$$ i.e.\,the drift belongs to still admissible critical time-inhomogeneous Ladyzhenskaya-Prodi-Serrin class. (By the way, repeating the argument in Example \ref{ex1}(1), one sees that drift $(K_1 \ast \rho)(t,\cdot)$ belongs to the class of time-inhomogeneous form-bounded vector fields, i.e.\,for a.e.\,$t \in \mathbb R_+$
$$
\|b(t,\cdot)\varphi\|_2^2 \leq \delta \|\nabla \varphi\|_2^2 + c_\delta\|\varphi\|_2^2 \quad \forall\,\varphi \in W^{1,2},
$$
which, in principle, puts the corresponding Keller-Segel equation within the reach of our methods, at least at the level of a priori Sobolev regularity estimates, see Remark \ref{int_st}.)

\medskip

The observation that to handle the $d$-dimensional Keller-Segel model one can use energy methods in $L^p$ with $p$ large (larger than $\frac{d}{2}$) goes back already to \cite{JL,CPZ}.

\medskip

We also use energy  methods in $L^p$ with $p$ large, but we do it for a different purpose, i.e.\,to relax the assumption of the strength of interactions $\kappa$. 
Furthermore, in the present paper we face another situation where one needs to work in $L^p$ with large $p$. That is, in presence of repulsing-attracting structure in the drift $b$ we can replace the form-boundedness requirement by a more general condition (``multiplicative form-boundedness'', cf.\,Theorem \ref{markov_thm}). Now, to treat the right-hand side of nonhomogeneous equation \eqref{nonhom}, which is the analytic object behind the SDE with drift $b$, we need an
additional condition
$$
|b|^{\frac{1+\alpha}{2}} \in \mathbf{F}_{\chi} \text{ for some } \chi<\infty, \alpha \in ]0,1[,
$$
where $p':=\frac{p}{p-1} \leq 1+\alpha$. This extra condition is least restrictive if $\alpha$ is small, which forces us to consider large $p$. See Remark \ref{A_0_rem} for more details.

\medskip

(v)~We also mention recent results in \cite{CJM} on interacting particle systems and McKean-Vlasov SDEs with distributional interaction kernels in Besov spaces (see also references therein). The assumptions of \cite{CJM} are somewhat orthogonal to the present work and, at least at the moment, do not include the model interaction kernels \eqref{hardy} (while including other quite irregular distributional kernels) or keep track of the strength of interactions $\kappa$.

\medskip

(vi) In what follows, we refer to a well known in the literature on parabolic PDEs and singular SDEs classification of drifts:

 \begin{enumerate}

\item[--] \textit{Sub-critical case} if, upon zooming into small scales, i.e.\,applying parabolic scaling in
$$
(\partial_t-\Delta + b\cdot \nabla)v=0 \text{ in $\mathbb R^d$}
$$
or in
$$
Y_t=y-\int_0^t b(Y_s)ds + \sqrt{2}B_t, \quad y \in \mathbb R^d,
$$ the drift term vanishes. For instance, $b \in [L^q]^d$, $q>d$ is sub-critical.

\medskip

\item[--] \textit{Critical case} if zooming into small scales does not change the ``norm'' of the drift.

\medskip

For instance, parabolic scaling does not change the form-bound of the drift or its $L^d$ norm. So, both $\mathbf{F}_\delta$ and $[L^d]^d$ are critical classes. Note that this classification does not distinguish between critical drifts that reach blow up effects, such as $b \in \mathbf{F}_\delta$, and drifts that do not reach blow up effects, such as $b \in [L^d]^d$. In other words, one can multiply the latter by arbitrarily large constant without affecting well-posedness of the SDE, while form-bounded drifts can in general ``sense'' this multiplication (since it, obviously, changes the form-bound, which cannot be too large, see the beginning of the introduction). In order to distinguish between these two very different cases, we say that the former have critical-order singularities.

\medskip

We also consider in the present paper other \textit{critical} classes of drifts, such as multiplicatively form-bounded drifts (Definition \ref{mfbd_}) and weakly form-bounded drifts (Remark \ref{unique_strong_rem}) that expand the class of form-bounded vector fields $\mathbf{F}_\delta$ rather substantially.

\medskip

\item[--] \textit{Super-critical case} if zooming into small scales actually increases the ``norm'' of the drift. For instance, $b \in L^q$, $q<d$, is super-critical. Let us add that all known results on super-critical drifts $b$ require \textit{critical} positive part of ${\rm div\,}b$. This, of course, includes important case ${\rm div\,}b=0$. 

\end{enumerate}

\smallskip

In Remark \ref{sing_rem} we comment on the existing literature on PDEs and SDEs with super-critical drifts. Briefly, super-criticality of the drift destroys many basic regularity results, but some parts of the theory can be salvaged.

\medskip

(vii) As was indicated above, the proof of Theorem \ref{thmK1}(\textit{iii}) uses the many-particle Hardy inequality of \cite{HHLT}:
for $d \geq 3$, all  $N \geq 2$, 
\begin{equation}
\label{multi_hardy}
C_{d,N} \sum_{1 \leq i<j \leq N}\int_{\mathbb R^{Nd}}\frac{|\varphi(x)|^2}{|x_i-x_j|^2}dx \leq \int_{\mathbb R^{Nd}}|\nabla \varphi(x)|^2 dx, \quad x=(x_1,\dots,x_N),
\end{equation}
for all $\varphi \in W^{1,2}(\mathbb R^{Nd})$,
where
$$
C_{d,N}:=(d-2)^2 \max\bigg\{\frac{1}{N},\frac{1}{1+\sqrt{1+\frac{3(d-2)^2}{2(d-1)^2}(N-1)(N-2)}} \bigg\}.
$$
In the proof we replace constant $C_{d,N}$  with smaller constant $\frac{(d-2)^2}{N}$. However, the maximum for large $N$ and $d \leq 6$ in the definition of $C_{d,N}$ is attained in the second argument. So, the constraint $\kappa<16$ in Theorem \ref{thmK1}(\textit{iii}) (or in Theorem A(\textit{iii})) can be somewhat relaxed for $d \leq 6$. 

\medskip

The authors of \cite{HHLT} also provide, among other results, an upper bound on the constant in \eqref{multi_hardy}. At the moment of writing of this article, to the best of author's knowledge, the optimal constant in \eqref{multi_hardy} is not known. 

\medskip

It is natural to expect that the relationship between Theorem \ref{thmK1}(\textit{iii}) and the many-particle Hardy inequality \eqref{multi_hardy} goes both ways, i.e.\,there is a direct relationship between the optimal constant in many-particle Hardy inequality  \eqref{multi_hardy} and the critical threshold value of $\kappa$  that separates well-posedness of particle system \eqref{model_syst} from a blow up, in which case Monte-Carlo simulations for \eqref{model_syst} should produce the optimal $C_{d,N}$ in \eqref{multi_hardy}; we pursue this in \cite{HKS}.

\subsection{Notations}
\label{notations_sect}
Put 
$$
\langle f\rangle:=\int_{\mathbb R^d} f(y)dy, \quad \langle f,g\rangle:=\langle fg\rangle
$$
(all functions in this paper are real-valued). 
For vector fields $b$, $\mathsf{f}:\mathbb R^d \rightarrow \mathbb R^d$, we put
$$
\langle b,\mathsf{f}\rangle:=\langle b \cdot \mathsf{f}\rangle \qquad \text{($\cdot$ is the inner product in $\mathbb R^d$)}.
$$
Let $\|\cdot\|_{p \rightarrow q}$ denote the $L^p \rightarrow L^q$ operator norm. Let $C_\infty$ denote the space of continuous functions on $\mathbb R^d$ vanishing at infinity, endowed with the $\sup$-norm. Let $B_R(y) \subset \mathbb R^d$ be the open ball of radius $R$ centered at $y \in \mathbb R^d$, $|B_R(x)|$ denotes its volume. Set $B_R:=B_R(0)$. Given a function $f$, we denote its positive and negative parts by 
$$(f)_+:=f \vee 0, \quad (f)_-:=-(f \wedge 0).$$ 
Set
$$
\gamma(x):=\left\{
\begin{array}{ll}
c\exp\left(\frac{1}{|x|^2-1}\right)& \text{ if } |x|<1, \\
0, & \text{ if } |x| \geqslant 1,
\end{array}
\right.
$$
where $c$ is adjusted to $\int_{\mathbb R^d} \gamma(x)dx=1$, and put $\gamma_\varepsilon(x):=\frac{1}{\varepsilon^{d}}\gamma\left(\frac{x}{\varepsilon}\right)$, $\varepsilon>0$, $x\in \mathbb R^d$.
Define the Friedrichs mollifier of a function $h \in L^1_{\loc}$ (or a vector field with entries in $L^1_{\loc}$) by $$E_\varepsilon h:=\gamma_\varepsilon \ast h.$$

\subsection*{Acknowledgements} The author is sincerely grateful to the anonymous referee for making a number of very useful comments.

\bigskip

\section{Particle systems}

For brevity, we will consider first the particle system without the drift terms $M(X_i)$:
\begin{equation}
\label{syst3}
X_i(t)=x_i-\frac{1}{N} \sum_{j=1, j \neq i}^N \int_0^tK_{ij}\big(X_i(s)-X_j(s)\big)ds + \sqrt{2}B_i(t), \quad 1 \leq i \leq N, \quad t \in [0,T],
\end{equation}
 where $x=(x_1,\dots,x_N) \in \mathbb R^{Nd}$, $N \geq 2$. However, we will explain in Remark \ref{ext_rem} below how to put the drifts back there.

\medskip

Let $e_t:C([0,T],\mathbb R^{Nd}) \rightarrow \mathbb R^{Nd}$ be defined by $$e_t(\omega):=\omega_t.$$

\medskip

Recall that a probability measure $\mathbb P_x$ ($x \in \mathbb R^{Nd}$) on the canonical space of continuous trajectories $\omega=(\omega^1,\dots,\omega^N)$ in $\mathbb R^{Nd}$ is called a martingale solution to particle system \eqref{syst3} on $[0,T]$ if 

1) $$\mathbb P_{x,0}=\delta_x,$$ where $\mathbb P_{x,t}:=\mathbb P \circ e_t^{-1}$ (on $\mathbb R^{Nd}$),

2) $$\mathbb E_{x} \sum_{i=1}^N \sum_{j=1, j \neq i}^N \int_0^T |K_{ij}(\omega_t^i-\omega_t^j)|dt<\infty,$$

3) for every $\phi \in C_c^2(\mathbb R^{Nd})$ the process
$$
[0,T] \ni r \mapsto \phi(\omega_r)-\phi(x) + \int_0^r \big(-\Delta_{y} \phi(\omega_t) + \frac{1}{N}\sum_{i=1}^N \sum_{j=1, j \neq i}^N K_{ij}(\omega_t^i-\omega_t^j) \cdot \nabla_{y_i} \phi(\omega_t)\big)dt
$$
is a martingale under $\mathbb P_x$.

\medskip

We will also need the following definition. Let $K$ satisfy \eqref{fbd_}, let $\{K^n\}$ be some sequence of vector fields (in what follows, $K^n$ will be more regular than $K$).

\begin{definition}
\label{def_Kn}
Let us say that $\{K^n\}$ does not increase the form-bounds of $K$ if for every $n \geq 1$
$$
\|K^n \varphi\|_2^2 \leq \kappa \|\nabla \varphi\|_2^2 + c_\kappa\|\varphi\|_2^2  \quad \forall \varphi \in W^{1,2}(\mathbb R^d),
$$
i.e.\,$\{K^n\}$ satisfy \eqref{fbd_} with the same constants as $K$.  
\end{definition}

\subsection{General interaction kernels}

\begin{theorem}[General interactions]
\label{thmK1_0}

Assume that the interaction kernels $K_{ij}$ in particle system \eqref{syst3} satisfy 
\begin{equation}
\label{K1_cond}
K_{ij} \in \mathbf{F}_\kappa \quad \text{with } \kappa<4\bigg(\frac{N}{N-1}\bigg)^2
\end{equation}
(see Definition \ref{def1}).
Then the following are true:

\smallskip

\begin{enumerate}[label=(\roman*)]

\item There exists a strong Markov family of martingale solutions $\{\mathbb P_x\}_{x \in \mathbb R^{Nd}}$ of particle system \eqref{syst3}.

\smallskip

\item  The function
\begin{equation}
\label{u_def}
u(x):=\mathbb E_{\mathbb P_x}\int_0^\infty e^{-\lambda s}f(\omega_s^1,\dots,\omega_s^N) ds, \quad x \in \mathbb R^{Nd}, \quad f \in C_c^\infty(\mathbb R^{Nd}),
\end{equation}
where $\lambda$ is assumed to be sufficiently large,
is a locally H\"{o}lder continuous weak solution  to elliptic Kolmogorov equation
\begin{equation}
\label{eq_el}
\bigg(\lambda - \Delta+ \frac{1}{N}\sum_{i=1}^N \sum_{j=1,  j \neq i}^N K_{ij}(x_i-x_j) \cdot \nabla_{x_i} \bigg)u=f, \quad x=(x_1,\dots,x_N),
\end{equation}
see definitions in Remark \ref{unique_rem} where we also discuss the uniqueness of $u$. 

\medskip

\item Fix $p>\frac{2}{2-\frac{N-1}{N}\sqrt{\kappa}}$. The family of operators $\{P_t\}_{t \geq 0}$ defined by
$$
P_tf(x):=\mathbb E_{\mathbb P_x}[f(\omega_t^1,\dots,\omega_t^N)], \quad f \in C_c^\infty(\mathbb R^{Nd}), 
$$
admits extension by continuity to a strongly continuous quasi contraction Markov semigroup on $L^p$ of integral operators, say $P_t=:e^{-t\Lambda_p}$, such that
\begin{equation}
\label{pq}
\|e^{-t\Lambda_p}\|_{p \rightarrow q} \leq cw^{\omega t}t^{-\frac{Nd}{2}(\frac{1}{p}-\frac{1}{q})}, \quad p \leq q \leq \infty
\end{equation}
for appropriate constants $c$ and $\omega$.  In view of \eqref{pq}, Dunford-Pettis' theorem yields that $e^{-t\Lambda_p}$ is a semigroup of integral operators. 
Their integral kernel $e^{-t\Lambda}(x,z)$ does not depend on $p$ and is defined to be the heat kernel of particle system \eqref{syst3}.

 If $p=2$, then we have
$$
\Lambda_2 \supset - \Delta+ \frac{1}{N}\sum_{i=1}^N \sum_{j=1,  j \neq i}^N K_{ij}(x_i-x_j) \cdot \nabla_{x_i} \upharpoonright C_c^\infty(\mathbb R^{Nd}).
$$

The semigroup $e^{-t\Lambda_p}$ is unique among semigroups that can be constructed via approximation, i.e.\,for any sequence of bounded smooth interaction kernels $\{K^n_{ij}\}$,
$$
K^n_{ij} \rightarrow K_{ij} \quad \text{ in } [L^2_{\loc}(\mathbb R^{d})]^d,
$$  
that do not increase the form-bounds of $K$,
for every $f \in C_c^\infty(\mathbb R^{Nd})$ solutions  $\{v_n\}$ to 
$$
\big(\partial_t-\Delta + \frac{1}{N}\sum_{i=1}^N \sum_{j=1,  j \neq i}^N K^n_{ij}(x_i-x_j) \cdot \nabla_{x_i}\big)v_n=0, \quad v_n(0)=f
$$
converge to the same limit $e^{-t\Lambda_p}f$ in $L^p(\mathbb R^{Nd})$ loc.\,uniformly in $t \geq 0$.

\smallskip

\item If, furthermore,
$$
\kappa<\frac{1}{(N-1)^2d^2},
$$
then for every initial configuration $x=(x_1,\dots,x_N) \in \mathbb R^{Nd}$ martingale solution $\mathbb P_x$ satisfies for a given  $q \in ]Nd,\frac{N}{N-1}\kappa^{-\frac{1}{2}}[$ Krylov-type bounds
\begin{equation}
\label{kr1_u}
\mathbb E_{\mathbb P_x} \int_0^T |h(s,\omega^1_s,\dots,\omega^N_s)|ds \leq c\|h\|_{L^q([0,T] \times \mathbb R^{Nd})}
\end{equation}
and
\begin{equation}
\label{kr1_u2}
\mathbb E_{\mathbb P_x} \int_0^T |b(\omega^1_s,\dots,\omega^N_s)| |h(\tau,\omega^1_s,\dots,\omega^N_s)| ds \leq c\|b|h|^{\frac{q}{2}}\|^{\frac{2}{q}}_{L^2([0,T] \times \mathbb R^{Nd})},
\end{equation}
for all $h \in C_c([0,T] \times \mathbb R^{Nd})$, for some constant $c>0$, where vector field $b=(b_1,\dots,b_N):\mathbb R^{Nd} \rightarrow \mathbb R^{Nd}$ is defined by
\begin{equation*}
b_i(x):= \frac{1}{N}\sum_{j=1, j \neq i}^N K_{ij}(x_i-x_j), \quad x=(x_1,\dots,x_N) \in \mathbb R^{Nd}, \quad 1 \leq i \leq N.
\end{equation*}

\smallskip

Moreover, $\mathbb P_x$ is the only martingale solution to \eqref{syst3} that satisfies \eqref{kr1_u}, \eqref{kr1_u2} (``conditional weak uniqueness'').

\smallskip

\item There exists constant $C<1$ such that if $K_{ij}$ is of the form 
\begin{equation}
\label{K4_cond}
K_{ij}(x_i,x_i-x_j)=\zeta(x_i)K^0_{ij}(x_i-x_j),
\end{equation}
with $\zeta$ having compact support, $\|\zeta\|_\infty \leq 1$,
and $K^0_{ij} \in \mathbf{F}_\kappa$ with
\begin{equation}
\label{K_number2}
\kappa<\frac{C}{(N-1)^2d^2}
\end{equation}
(the previous assertions are valid for such interaction kernels as well),
then for every initial configuration $(x_1,\dots,x_N) \in \mathbb R^{Nd}$ particle system \eqref{syst3} has a strong solution on $[0,T]$ that is unique among all strong solutions defined on the same probability space satisfying \eqref{kr1_u}, \eqref{kr1_u2}. 
\end{enumerate}
\end{theorem}

We recall from the discussion in the introduction that if the strength of interactions $\kappa$ is taken to be too large then a weak solution to the particle system \eqref{syst3} can cease to exist. So, in Theorem \ref{thmK1_0}(\textit{i}) we are dealing with the critical scale of the strength of interactions.

\medskip

Let us emphasize that as the strength of interactions $\kappa$ becomes smaller, the theory of particle system \eqref{syst3} in Theorem \ref{thmK1_0} becomes more detailed.

\medskip

We are rather satisfied with the conditions on the interaction kernels $K_{ij}$ in Theorem \ref{thmK1_0}(\textit{i})-(\textit{iii})  where the assumption on the strength of interactions $\kappa$ ``stabilizes'' to a positive value as the number of particles $N \rightarrow \infty$, so, in principle, this opens up a possibility of studying the existence of the mean field limit (see Remark \ref{mean_field_rem}). However, in assertions (\textit{iv}), (\textit{v}) of Theorem \ref{thmK1_0} the assumption on $\kappa$ degenerates to zero as $N$ goes to infinity, which seems to be a by-product of our method of embedding particle system \eqref{syst3} in the general SDE  \eqref{sde1}. We comment more on this below.

\subsection{Attraction and repulsion}We now turn to the interaction kernels having a repulsion-attraction structure. While the repulsion between the particles, in a sense, contributes towards well-posedness of particle system \eqref{syst3} by preventing collisions, the attraction can lead to the blow up effects (see the discussion in the introduction). We take into account the attraction between the particles by looking at the positive part of the divergence of the interaction kernels $K_{ij}$ in \eqref{syst3}.

\begin{definition}
\label{def3}
 $({\rm div\,}K)_+ \in L^1_{\loc}$ is said to be form-bounded  if there exists constant $\kappa_+$ such that
\begin{equation}
\label{div_fbd}
\langle({\rm div} K)_+ \varphi,\varphi \rangle \leq \kappa_+ \|\nabla\varphi\|_2^2 + c_{\kappa_+}\|\varphi\|_2^2, \quad \forall \varphi \in W^{1,2},
\end{equation}
for some $c_{\kappa_+}$. 
\end{definition}

We abbreviate \eqref{div_fbd}, with a slight abuse of notation, as $$({\rm div\,}K)_+^{\frac{1}{2}} \in \mathbf{F}_{\kappa_+}.$$
For example, the previous condition is satisfied if $({\rm div\,}K)_+ \in L^{\frac{d}{2},\infty}$ (weak $L^\frac{d}{2}$ class). This includes, of course,  $({\rm div\,}K)_+ \in L^{\frac{d}{2}}$, in which case $\kappa_+$ can be chosen arbitrarily small (cf.\,Example \ref{ex1}(1)).

\begin{example} 
Let $K$ be the model singular attracting kernel \eqref{hardy}, i.e.\,$K(y)= \sqrt{\kappa}\frac{d-2}{2}\frac{ y }{|y|^2}$. Then
$${\rm div\,}K=\sqrt{\kappa}\frac{(d-2)^2}{2}|y|^{-2},$$
so, by Hardy's inequality $({\rm div\,}K)_+^\frac{1}{2} \in \mathbf{F}_{\kappa_+}$, $\kappa_+=2\sqrt{\kappa}$, $c_{\kappa_+}=0.$
\end{example}

Regarding the negative part $({\rm div\,}K)_-$, which is responsible for the repulsion between the particles, we will only impose a rather quite mild condition  that $({\rm div\,}K)_-$ can be represented as the sum of a function in $L^1(\mathbb R^d)$ and a bounded function.

Already the hypothesis $({\rm div\,}K_{ij})_+^{\frac{1}{2}} \in \mathbf{F}_{\kappa_+}$ allows one to easily prove, integrating by parts and using Lemma \ref{p_2}, that solution $v$ of the backward Kolmgorov equation for particle system \eqref{syst3}
$$
\biggl(\partial_t - \Delta+ \frac{1}{N}\sum_{i=1}^N \sum_{j=1,  j \neq i}^N K_{ij}(x_i-x_j) \cdot \nabla_{x_i} \bigg)v=0, \quad v(0,\cdot)=f(\cdot) \text{ in } \mathbb R^{Nd}
$$
satisfies a quasi contraction estimate
\begin{equation}
\label{contr_est}
\|v(t)\|_p \leq e^{\omega_p t}\|f\|_p, \quad t>0
\end{equation}
provided $\kappa_+$ is not too large, for appropriate $p$ and $\omega_p$.
However, without any additional assumptions on the interaction kernels $K_{ij}$ themselves, there is no hope of advancing substantially farther than \eqref{contr_est}. In fact, without conditions on $K_{ij}$, even requiring ${\rm div\,}K_{ij}=0$, puts us firmly in the super-critical regime (cf.\,(vi) the introduction), so even the proof of a priori H\"{o}lder continuity of solution $v$ or of solution to the corresponding elliptic equation becomes out of reach. We need a condition on $K_{ij}$ that will put us back in the critical regime.

\begin{definition}
\label{def2}
A vector field $K \in [L^1_{\loc}]^d$
is said to be multiplicatively form-bounded if there exists constant $\kappa_0$ (``multiplicative form-bound'') such that
\begin{equation}
\label{mfbd_}
\langle|K| \varphi,\varphi\rangle \leq \kappa_0\|\nabla \varphi\|_2\|\varphi\|_2 + c_{\kappa_0}\|\varphi\|_2^2, \quad \forall \varphi \in W^{1,2}.
\end{equation}
\end{definition}

We abbreviate \eqref{mfbd_} as $$K \in \mathbf{MF}_{\kappa_0}.$$
It will be clear from the results below that the actual value of $\kappa_0$ is not important for well-posedness of the particle system \eqref{syst3}; it is the value of $\kappa_+$ that matters.

\medskip

Note that the class of form-bounded vector field $\mathbf{F}_{\kappa_0}$ is also a critical class, so we could use it here as well.
However, our ultimate goal is to identify the optimal (least restrictive) assumptions on $K_{ij}$, so we will work with the broader class $\mathbf{MF}_{\kappa_0}$:

\begin{example}

\begin{enumerate}[label=(\roman*)]

\item Every form-bounded vector field is multiplicatively form-bounded, but not vice versa, 
see Remark \ref{mult_class}. 
In particular, all vector fields listed in Example \ref{ex1} are multiplicatively form-bounded.

\smallskip

\item The class $\mathbf{MF}_{\kappa_0}$ contains the largest possible, up to the strict inequality in $\varepsilon>0$, scaling-invariant  Morrey class $M_{1+\varepsilon}$: if
\begin{equation*}
\|K\|_{M_{1+\varepsilon}}:=\sup_{r>0, y \in \mathbb R^d} r\biggl(\frac{1}{|B_r(y)|}\int_{B_r(y)}|K|^{1+\varepsilon}dy \biggr)^{\frac{1}{1+\varepsilon}}<\infty,
\end{equation*}
then $$K \in \mathbf{MF}_{\kappa_0}, \quad \kappa_0=c(d,\varepsilon)\|K\|_{M_{1+\varepsilon}},$$ see details in Remark \ref{mult_class}. Here one can already see the gain in comparison with the class of form-bounded vector field $\mathbf{F}_{\kappa_0}$, which contains only $M_{2+\varepsilon}$ (and itself is contained in $M_{2}$).

\smallskip

\item 
If the following Morrey class condition  is satisfied:
$$
\sup_{r>0, y \in \mathbb R^d} r^2\biggl(\frac{1}{|B_r(y)|}\int_{B_r(y)}|({\rm div\,} K)_+|^{1+\varepsilon}dy \biggr)^{\frac{1}{1+\varepsilon}}<\infty,
$$
then $({\rm div\,} K)_+^{\frac{1}{2}} \in \mathbf{F}_{\kappa_+}$ with appropriate $\kappa_+$.

\end{enumerate}

\end{example}

It was demonstrated in \cite{S} that condition $b \in \mathbf{MF}_\delta$ under additional divergence-free hypothesis ${\rm div\,}b=0$ provides two-sided Gaussian bounds on the heat kernel of operator $-\nabla \cdot a \cdot \nabla + b \cdot \nabla$ with uniformly elliptic measurable matrix $a$. (Of course, having $({\rm div\,}b)^{\frac{1}{2}} \in \mathbf{F}_{\delta_+}$, as in the present paper, destroys both the upper and the lower Gaussian bounds on the heat kernel even of $-\Delta + b \cdot \nabla$.)

\medskip

Put
 $$\mathbf{F}:=\{K \mid K \in \mathbf{F}_{\kappa_0} \text{ for some } \kappa_0<\infty\}$$
and
$$
\mathbf{MF}:=\{K \mid K \in \mathbf{MF}_{\kappa_0} \text{ for some } \kappa_0<\infty\}.
$$

Definition \ref{def_Kn} extends naturally  to $K$ satisfying \eqref{mfbd_}, \eqref{div_fbd} or \eqref{K3_cond} below.
In all these cases, in Section \ref{approx_sect} we show that the vector fields $K^n$ defined by 
\begin{equation}
\label{K_def9}
K^n:=E_{\varepsilon_n}K, \quad \varepsilon_n \downarrow 0, \quad \text{$E_\varepsilon$ is the Friedrichs mollifier},
\end{equation}
are bounded, smooth and
do not increase the corresponding form-bounds of $K$.

\begin{theorem}[Repulsing-attracting interactions]
\label{thmK1}
The following are true:

\begin{enumerate}[label=(\roman*)]

\item 

Assume that the interaction kernels $K_{ij}$ in particle system \eqref{syst3} satisfy 
\begin{equation}
\label{K2_cond}
K_{ij} \in \mathbf{MF}, \qquad
\left\{
\begin{array}{l}
({\rm div\,} K_{ij})_- \in L^1+L^\infty,
\\
({\rm div\,} K_{ij})_+^{\frac{1}{2}} \in \mathbf{F}_{\kappa_+} \text{ with } \kappa_+<4\frac{N}{N-1}
\end{array}
\right.
\qquad |K_{ij}|^{\frac{1+\alpha}{2}} \in \mathbf{F}
\end{equation}
for some $\alpha>0$ fixed arbitrarily close to zero.
The assertions (\textit{i}), (\textit{ii}) of Theorem \ref{thmK1_0} are valid for these interaction kernels as well.

\smallskip

\item Assume that $K_{ij}$ satisfy a more restrictive condition than \eqref{K2_cond} in Theorem \ref{thmK1}:
\begin{equation}
\label{K3_cond}
K_{ij} \in \mathbf{F}, \qquad
\left\{
\begin{array}{l}
({\rm div\,} K_{ij})_- \in L^1+ L^\infty, \\
({\rm div\,} K_{ij})_+^{\frac{1}{2}} \in \mathbf{F}_{\kappa_+} \text{ with } \kappa_+<4\frac{N}{N-1}.
\end{array}
\right.
\end{equation}
Fix $p>\frac{4}{4-\frac{N-1}{N}\kappa_+}$. Then assertion (\textit{iii}) of Theorem \ref{thmK1_0} also remains valid.  

\smallskip

\item
Let
\begin{equation}
\label{model}
K_{ij}(y)=\sqrt{\kappa}\frac{d-2}{2}|y|^{-2}y + K_{0,ij}(y), \quad y \in \mathbb R^d. 
\end{equation}
If the strength of attraction
$$
\kappa<16
$$
and $K_{0,ij}$ satisfy conditions \eqref{K1_cond} or \eqref{K2_cond} with sufficiently small form-bounds,
then assertions (\textit{i})-(\textit{iii}) of Theorem \ref{thmK1_0}  with $p>\frac{4}{4-\sqrt{\kappa}}$ remain valid.

\smallskip 

\item Furthermore, for the model attracting interaction kernel $$K(y)=\sqrt{\kappa}\frac{d-2}{2}|y|^{-2}y, \quad \kappa<16,$$ 
the previous assertions remain valid,
the heat kernel $e^{-t\Lambda}(x,z)$ of particle system \eqref{syst3}
satisfies, up to modification on a measure zero set, the heat kernel bound
\begin{equation*}
e^{-t\Lambda}(x,z)\leq Ct^{-\frac{Nd}{2}}\prod_{1 \leq i<j \leq N} \eta(t^{-\frac{1}{2}}|z_i-z_j|), \quad t \in ]0,T],
\end{equation*}
for some $C=C_T$,
for all $x \in \mathbb R^{Nd}$, $z=(z_1,\dots,z_N) \in \mathbb R^{Nd}$
provided $z_i \neq z_j$ ($i \neq j$),
for a fixed function
$1 \leq \eta \in C^2(]0,\infty[)$ such that $$\eta(r)=\left\{
\begin{array}{ll}
r^{-\sqrt{\kappa}\frac{d-2}{2}\frac{1}{N}} & 0<r<1, \\
2, & r>2.
\end{array}
\right.
$$

\end{enumerate}
\end{theorem}

\begin{remark}
The additional right-most condition on $K_{ij}$ in \eqref{K2_cond} is, generally speaking, much weaker than the left-most condition (informally, the former treats $|K|$ as a potential, while a proper ``potential analogue'' of the drift perturbation $K \cdot \nabla$ would be $|K|^2$). For instance, if we were to state condition \eqref{K2_cond} on the scale of $L^p$ spaces, then it would become 
\begin{align*}
|K| \in L^{d}+L^\infty, \qquad
\left\{
\begin{array}{l}
({\rm div\,} K)_- \in L^1+L^\infty,
\\
({\rm div\,} K)_+ \in L^{\frac{d}{2}}+L^\infty
\end{array}
\right.
\qquad |K| \in L^{\frac{d}{2}(1+\alpha)}+L^\infty,
\end{align*}
where, recall, $\alpha>0$ is fixed arbitrarily small, i.e.\,the right-most condition follows from the left-most one. The same would happen if we were working on the scale of scaling-invariant Morrey spaces (cf.\,Example \ref{ex1}(4)).
\end{remark}

The improvement of the assumptions on $\kappa$ in Theorem \ref{thmK1}(\textit{iii}), compared to Theorem \ref{thmK1_0} and Theorem \ref{thmK1}(\textit{i}),(\textit{ii}), is due to a refinement of Lemma \ref{p_2} by means of the many-particle Hardy inequality \eqref{multi_hardy}
of \cite{HHLT}.

\medskip

The heat kernel bound in Theorem \ref{thmK1}(\textit{iv}) is not unexpected (although we could not find it in the literature). Indeed, an elementary calculation shows that $$
\psi(x):=\prod_{1 \leq i<j \leq N} |x_i-x_j|^{-\sqrt{\kappa}\frac{d-2}{2}\frac{1}{N}}.
$$ is a Lyapunov function of the formal adjoint of $\Lambda=-\Delta_x - \sqrt{\kappa}\frac{d-2}{2}\frac{1}{N}\sum_{i=1}^N\sum_{j=1, j \neq i}^N \frac{x_i-x_j}{|x_i-x_j|^2} \cdot \nabla_{x_i}$, i.e.\,the following identity holds:
\begin{equation}
\label{heat_rem}
-\Delta_x \psi + \sqrt{\kappa}\frac{d-2}{2}\frac{1}{N}\sum_{i=1}^N \nabla_{x_i} \biggl( \sum_{j=1, j \neq i}^N \frac{x_i-x_j}{|x_i-x_j|^2} \psi\biggr)=0.
\end{equation}
One can expect that such Lyapunov function will appear as a multiple in the heat kernel bounds.
That said, the question of how to prove such an estimate is non-trivial due to singularities in the drift. An interesting aspect of Theorem \ref{thmK1}(\textit{iv}) is its proof, which uses an abstract desingularization result from \cite{KiSSz}, see Appendix \ref{app_desing}.

\medskip

In Theorem \ref{thmK1}(\textit{iv}), we expect to have two-sided bound
\begin{equation}
\label{gauss_bd}
C_1t^{-\frac{Nd}{2}}e^{-\frac{|x-y|^2}{c_2 t}}\varphi_t(y) \leq e^{-t\Lambda}(x,y)\leq C_3t^{-\frac{Nd}{2}}e^{-\frac{|x-y|^2}{c_4 t}}\varphi_t(y),
\end{equation}
where $$\varphi_t(y):=\prod_{1 \leq i<j \leq N} \eta(t^{-\frac{1}{2}}|y_i-y_j|),$$ as is suggested by the analogous results for Kolmogorov operator $-\Delta - \sqrt{\kappa}|x|^{-2}x \cdot \nabla$, $0<\kappa<4$ on $\mathbb R^d$, see \cite{MNS}. Moreover, there should be an analogous  to Theorem \ref{thmK1}(\textit{iv}) and \eqref{gauss_bd} result in the case of attracting interactions, see \cite{MNS} and \cite{KiS_RIMS} regarding $-\Delta + \sqrt{\kappa}|x|^{-2}x \cdot \nabla$, $0<\kappa<\infty$. (\cite{KiS_RIMS,KiSSz} deal with the fractional Laplacian $(-\Delta)^{\alpha/2}$ perturbed by the model singular drift term $c|x|^{-\alpha}x \cdot \nabla$, $1<\alpha<2$.)

\medskip

One drawback of assertions (\textit{iv}), (\textit{v}) of Theorem \ref{thmK1_0} is the difficulty with taking into account the repulsion/attraction structure of the interaction kernel $K$ simply by looking at the divergence of $K$, as we do in Theorem \ref{thmK1}.
That said, in what concerns conditional weak uniqueness for the particle system (as in Theorem \ref{thmK1_0}(\textit{iv})), in Theorem \ref{unique_thm} we consider the general SDE \eqref{sde1} and propose another condition on the drift $b$ that provides conditional weak uniqueness for \eqref{sde1} while taking into account the repulsion/attraction. We show in Example \ref{ex_unique} that there is some truth to this condition: it is always satisfied in dimensions $d \geq 4$ for the model repulsing drift $b(x)=-\sqrt{\delta}\frac{d-2}{2}|x|^{-2}x$, regardless of the value for the form-bound $\delta>0$, as one would expect. This requires us to obtain gradient bounds in $L^q$ starting with $q>d-2$, hence the need to work in the elliptic setting. (In the parabolic setting we would need $q>d$.)
Nevertheless, this result, when applied via Lemma \ref{p_2} to drift \eqref{b_def} with repulsing interactions $K_{ij}(y)=-\sqrt{\kappa}\frac{d-2}{2}|y|^{-2}y$, leads to a condition on $\kappa$ that still depends on the number of particles $N$. 
So, there is still work to be done to find a proper analogue of Theorem \ref{unique_thm} for particle system \eqref{syst3}.

\medskip

\subsection{Comments on the proofs of Theorems \ref{thmK1_0} and \ref{thmK1}}
\label{thm1_rem}
It is not difficult to modify the proofs of Theorems \ref{thmK1_0} and \ref{thmK1} to extend them to the sums of the interaction kernels satisfying \eqref{K1_cond} and \eqref{K2_cond}, under properly adjusted assumptions on the form-bounds.

\medskip

We prove Theorems  \ref{thmK1_0} and \ref{thmK1} by embedding particle system \eqref{syst3} in the general SDE \eqref{sde1} considered in $\mathbb R^{Nd}$,
with drift  $b=(b_1,\dots,b_N):\mathbb R^{Nd} \rightarrow \mathbb R^{Nd}$ defined by
\begin{equation}
\label{b_def}
b_i(x):= \frac{1}{N}\sum_{j=1, j \neq i}^N K_{ij}(x_i-x_j), \quad x=(x_1,\dots,x_N) \in \mathbb R^{Nd}, \quad 1 \leq i \leq N.
\end{equation}

\begin{lemma}
\label{p_1} If $K_{ij} \in \mathbf{F}_\kappa (\mathbb R^d)$, then $b$ defined by \eqref{b_def} satisfies
$$\left\{
\begin{array}{l}
b \in \mathbf{F}_\delta(\mathbb R^{Nd}) \\
\text{with } \delta=\frac{(N-1)^2}{N^2}\kappa, \quad c_\delta=\frac{(N-1)^2}{N}c_\kappa.
\end{array}
\right.$$
\end{lemma}

\begin{lemma} 
\label{p_2}
If $K_{ij} \in \mathbf{MF}_\kappa(\mathbb R^d)$, $({\rm div\,} K_{ij})_+^{\frac{1}{2}}  \in \mathbf{F}_{\kappa_+}(\mathbb R^d)$,
$|K_{ij}|^{\frac{1+\alpha}{2}} \in \mathbf{F}_{\sigma} (\mathbb R^d)$, $\alpha \in [0,1]$,
then $b$ defined by \eqref{b_def} satisfies
\begin{equation}
\label{b_est2}
\left\{
\begin{array}{l}
b \in \mathbf{MF}_\delta(\mathbb R^{Nd}) \\
\text{with } \delta=\frac{N-1}{\sqrt{N}}\kappa, \quad c_\delta=(N-1)c_\kappa,
\end{array}
\right.
\end{equation}
\begin{equation}
\label{div_est}
\left\{
\begin{array}{l}
({\rm div\,} b)^{\frac{1}{2}}_+ \in \mathbf{F}_{\delta_+}(\mathbb R^{Nd}), \\
\text{with } \delta_+=\frac{N-1}{N}\kappa_+, \quad c_{\delta_+}=(N-1)c_{\kappa_+},
\end{array}
\right.
\end{equation}
\begin{equation}
\label{gamma_est}
\left\{
\begin{array}{l}
|b|^{\frac{1+\alpha}{2}} \in \mathbf{F}_\chi(\mathbb R^{Nd}), \\
\text{with } \chi=\frac{(N-1)^{1+\alpha}}{N^{1+\alpha}}\sigma , \quad c_{\chi}=\frac{(N-1)^{1+\alpha}}{N^{\alpha}}c_{\sigma}.
\end{array}
\right.
\end{equation}
\end{lemma}

Lemmas \ref{p_1}, \ref{p_2} allow us to obtain the existence of a strong Markov family of martingale solutions to \eqref{syst3} in Theorem \ref{thmK1_0}(\textit{i}), Theorem \ref{thmK1}(\textit{i}) from  Theorem \ref{markov_thm}(\textit{i}) for general SDE \eqref{sde1}.  
Theorem \ref{markov_thm}, and other results in Section \ref{gen_sect} dealing with general singular drifts, are of interest on their own.

\medskip

In Theorem \ref{markov_thm} the family of martingale solutions for \eqref{sde1} is constructed by applying a tightness argument where the central role belongs to the estimate
\begin{equation}
\label{tight_i}
\mathbf E\int_{t_0}^{t_1} |b_\varepsilon(Y_\varepsilon(s))|ds \leq C(t_1-t_0)^{\frac{\gamma}{1+\gamma}}, \quad t_0,t_1 \in [0,T]
\end{equation}
(this is \eqref{tight_ineq}), where $b_\varepsilon$ is a regularization of $b$ that does not increase form-bounds $\delta$, $\delta_+$ (see Definition \ref{def_Kn}) in Lemmas \ref{p_1}, \ref{p_2}, and $Y_\varepsilon$ is the strong solution of \eqref{sde1} with drift $b_\varepsilon$. Constants $C$, $\gamma>0$ are independent of $\varepsilon$.

\medskip

To prove \eqref{tight_i} and, furthermore, to prove the strong Markov property, we establish regularity results for non-homogeneous elliptic PDEs \eqref{eq22} and \eqref{eq7}. These are Theorems \ref{thm1} and \ref{thm2}, obtained via De Giorgi's method ran in $L^p$, where $p$ depends on the values of form-bounds $\delta$ and $\delta_+$. Theorems \ref{thm1} and \ref{thm2} are the main analytic results in the present paper.
We prove Theorem \ref{thm1} by showing that $u$ belongs to $L^p$ De Giorgi's classes and then following the arguments in \cite[Ch.\,7]{G}, that is, applying De Giorgi's method. As was mentioned in the introduction, we deal with $L^p$ De Giorgi classes that
are somewhat different from the $L^p$ De Giorgi classes found in the literature (cf.\,\cite{G}).

\begin{remark}[On the number of particles $N \rightarrow \infty$ ]
\label{mean_field_rem}
Let interaction kernel $K$
satisfy \eqref{K1_cond}. Then in \eqref{tight_i} $\gamma=\sqrt{2}-1$ (see the proof of Theorem \ref{markov_thm}) and, by Lemma \ref{p_1}, $$\delta=\frac{(N-1)^2}{N^2}\kappa,\;\;c_\delta=\frac{(N-1)^2}{N}c_\kappa.$$ Let $c_\kappa=0$ (as is the case for the model singular interactions \eqref{hardy}), then $c_\delta=0$. In turn, as $N \rightarrow \infty$, constant $\delta$ tends to $ \kappa$. Thus, our assumptions on the form-bound withstand the passage to the limit $N \rightarrow \infty$. However, in the tightness estimate \eqref{tight_i} applied to \eqref{b_def} the constant $C$ depends on $N$ (this is because in the proof of \eqref{tight_i} via De Giorgi's method we apply  Sobolev's embedding theorem on $\mathbb R^{Nd}$, which becomes weaker as the dimension of the spaces increases, and hence De Giorgi's iterations converge slower). So, \eqref{tight_i} does not allow to conclude the existence of a mean field limit by arguing as e.g.\,in Fournier-Jourdain \cite{FJ}. This is not surprising since \eqref{tight_i}, as it is proved now, does not take into account the exchangeability hypothesis on \eqref{syst3} even if we were to impose it.
\end{remark}

In \cite{KiS_sharp}, we proved, using De Giorgi's iterations in $L^p$, that the general SDE \eqref{sde1} with $b \in \mathbf{F}_\delta$, $\delta<4$ has a martingale solution for every initial point. This result yields the existence of a martingale solution part of Theorem \ref{markov_thm} under condition \eqref{A_0} on $b$, which we included in Theorem \ref{markov_thm} for the sake of completeness. In what concerns  \eqref{A_0}, in the present paper we make the next step and prove the strong Markov property.

\begin{remark}
\label{A_0_rem}
One of the main observations of the present paper is related to condition \eqref{A_1} of Theorem \ref{markov_thm}. This condition dictates the multiplicative form-boundedness assumption \eqref{K2_cond} on the interaction kernel $K$ when the latter has repulsion-attraction structure. In \eqref{A_1}, we relax the a priori condition $|b| \in L^2_{\loc}$ as in \eqref{A_0} to $|b| \in L^{1+\alpha}_{\loc}$ for $\alpha>0$ fixed arbitrarily small, aiming at stronger hypersurface singularities of $b$ (and thus of $K$). To achieve this, we once again need to work in $L^p$ for $p$ large. In fact, when dealing with the right-hand side of non-homogeneous equation $$(\mu-\Delta + b \cdot \nabla)u=|b|f \quad (f \in C_c^\infty),$$ as is needed to prove weak well-posedness of the general SDE \eqref{sde1}, we need to impose an extra condition 
$$
|b|^{\frac{1+\alpha}{2}} \in \mathbf{F}_{\chi} \text{ for some } \chi<\infty, \alpha \in ]0,1[,$$
which is related to $p$ via inequality
$$
p'=\frac{p}{p-1} \leq 1+\alpha$$ (cf.\,Theorem \ref{thm2}). If we were to consider this non-homogeneous equation in $L^2$, we would have to take $\alpha=1$, and so \eqref{A_1} and \eqref{K2_cond} would force the old form-boundedness assumption on drift $b$, i.e.\,as in \eqref{A_0}. This is another situation where one needs to work in $L^p$ with $p$ large, not related to maximizing admissible values of the form-bounds.

\end{remark}

Another technical novelty of the paper is Theorem \ref{thm2}, i.e.\,the embedding theorem, which has applications beyond this paper, see \cite{KiS_feller_4}.

\medskip

As we already mentioned in the introduction, the proof of Caccioppoli's inequality (Proposition \ref{c_prop2}) under assumption \eqref{K2_cond} uses an extension of the iteration procedure introduced in \cite{KiV}. In \cite{KiV},  the authors worked in $L^2$ and used Moser's method to prove the Harnack inequality for positive solutions of $(-\nabla \cdot a \cdot \nabla + b \cdot \nabla)u=0$ with measurable uniformly elliptic matrix $a$ and $b \in \mathbf{MF}_\delta$, $\delta<\infty$, provided that the form-bounds of the positive and the negative parts of ${\rm div\,}b$ satisfy some sub-critical assumptions.

\medskip

Assertion (\textit{iv}) of Theorem \ref{thmK1_0} follows, after applying Lemma \ref{p_1}, from the result in \cite{KiM_strong} whose proof, in turn, follows the method of R\"{o}ckner-Zhao \cite{RZ2}. In \cite{KiM_strong} we needed a technical hypothesis that $b$ has compact support, hence the condition in (\textit{iv}) on the support of $\zeta$. That said, this hypothesis can be removed in \cite{KiM_strong} by working with weights vanishing at infinity, which we plan to pursue elsewhere.

\bigskip

\section{Other remarks on Theorems \ref{thmK1_0}, \ref{thmK1}}

\begin{remark}[On McKean-Vlasov equation with form-bounded interaction kernel]
\label{int_st}
If $K \in \mathbf{F}_\kappa$, then in the McKean-Vlasov PDE
\begin{equation}
\label{mv}
\partial_t \rho-\Delta \rho - {\rm div\,}( \rho \tilde{K})=0, \quad \tilde{K}(t,\cdot)=K(\cdot) \ast \rho(t,\cdot),
\end{equation}
with the initial condition $\rho_0 \geq 0$, $\langle \rho_0 \rangle =1$, 
the drift $\tilde{K}$ is a time-inhomogeneous form-bounded vector field $[0,\infty[ \times \mathbb R^d \rightarrow \mathbb R^d$ \textit{having the same form-bound as} $K$, i.e.\,for a.e.\,$t \in [0,\infty[$, for all $\varphi \in W^{1,2}$, 
\begin{align*}
\langle |\tilde{K}(t)|^2\varphi^2\rangle & = \big\langle |\langle K(\cdot-z)\rho(t,z)\rangle_z|^2\varphi^2\big\rangle  \\
& (\text{apply Cauchy-Schwartz' inequality  and use $\langle \rho(t,z)\rangle_z=1$}) \\
& \leq \big\langle \langle |K(\cdot-z)|^2\rho(t,z)\rangle_z\varphi^2\big\rangle  = \big\langle \langle |K(\cdot-z)|^2 \varphi^2\rangle \rho(t,z)\big\rangle_z \\
& (\text{apply $K \in \mathbf{F}_\kappa$ and use again $\langle \rho(t,z)\rangle_z=1$}) \\
& \leq \kappa\langle |\nabla \varphi|^2\rangle + c_\kappa \langle \varphi^2 \rangle.
\end{align*}
Thus, in particular, all a priori estimates on solutions of the Kolmogorov forward equation with time-inhomogeneous form-bounded drifts (which can be obtained e.g.\,using the dual version of the method of \cite{Ki_Morrey}) transfer to solutions of McKean-Vlasov equation \eqref{mv}.
\end{remark}

\begin{remark}[Borderline strengths of interactions] 
\label{Orlicz_rem}
Applying the result of \cite{Ki_Orlicz}  for general form-bounded drifts $b \in \mathbf{F}_\delta$, $\delta \leq 4$, one can reach the borderline values of the strengths of interactions 
$$
\kappa=4\bigg(\frac{N}{N-1}\bigg)^2 \quad \text{ if \eqref{K1_cond} holds, or}
$$
$$
\kappa_+ = 4\frac{N}{N-1} \quad \text{ if \eqref{K3_cond} holds}
$$
by considering the corresponding to \eqref{syst3} Kolmogorov backward equation in the Orlicz space  with gauge function $\Phi=\cosh -1$. This space is situated between all $L^p$ and $L^\infty$ (paper \cite{Ki_Orlicz} deals with the dynamics of the torus but, as we show in subsequent paper \cite{KiS_feller_4}, one can also work on $\mathbb R^d$, although at expense of requiring a fast vanishing of the drift at infinity).
 
This result can be viewed as some sort of dual variant of the theory of entropy solutions of the forward Kolmogorov equation (regarding entropy solutions, see \cite{C}). That said, it seems like one can prove more by working with the backward Kolmogorov equation, e.g.\,construct strongly continuous semigroup for the borderline value of the form-bound in addition to the energy inequality and the uniqueness of weak solution, see \cite{Ki_Orlicz}. 
\end{remark}

\begin{remark}[Drifts]
\label{ext_rem}
One can easily extend the proofs of Theorems \ref{thmK1_0}, \ref{thmK1} to include particle system 
\begin{equation*}
dX_i=M_i(X_i)dt - \frac{1}{N}\sum_{j=1, j \neq i}^N K_{ij}(X_i-X_j)dt + \sqrt{2}dB_i, \quad 1 \leq i \leq N,
\end{equation*}
having singular drift terms
$$
M_i \in \mathbf{F}_{\mu}, \quad .
$$
Let us discuss for simplicity the case when $K_{ij}$ satisfy \eqref{K1_cond}. We require
that $\mu$, $\kappa$ satisfy
$$
\bigl(\sqrt{\mu}+\frac{N-1}{N}\sqrt{\kappa} \bigr)^2<4.
$$
We only need to embed this particle system into \eqref{sde1}, i.e.\,prove an analogue of Lemma \ref{p_1} for
vector field $b=b^M+b^K$ with $b^M$, $b^K:\mathbb R^{Nd} \rightarrow \mathbb R^{Nd}$ having components
\begin{equation}
\label{def_M}
b^M_i(x):= M_i(x_i), \quad b^K_i(x):=\frac{1}{N}\sum_{j=1, j \neq i}^N K_{ij}(x_i-x_j), \quad 1 \leq i \leq N,
\end{equation}
and then e.g.\,use Theorem \ref{markov_thm} for the general SDE \eqref{sde1}.
Repeating the proof of Lemma \ref{p_1}, we obtain right away that $$b^K \in \mathbf{F}_{\delta^K}(\mathbb R^{Nd}) \quad \text{ with }\quad \delta^K=\frac{(N-1)^2}{N^2}\kappa, \quad c_{\delta^K}=\frac{(N-1)^2}{N}c_\kappa$$
and
$$
b^M \in \mathbf{F}_{\delta^M}(\mathbb R^{Nd}) \quad \text{ with }\quad \delta^M=\mu, \quad c_{\delta^M}=Nc_\mu,
$$
see Remark \ref{rem_M} in Section \ref{decomp_sect} for the proof. It remains to note that the sum of two form-bounded vector fields is form-bounded, i.e.\,$b=b^M+b^K$ is in $\mathbf{F}_\delta$ with $\sqrt{\delta}=\sqrt{\delta^M}+\sqrt{\delta^K}$, and $\delta$ must be strictly less than $4$, cf.\,Theorem \ref{markov_thm}.

Arguing similarly, one can treat general drifts $M_i(X_1,\dots,X_N)$ ($1 \leq i \leq N$) in the particle system after adjusting the hypothesis on the form-bound, i.e.\,now $\delta^M=N \mu$.

\end{remark}

\begin{remark}[On the uniqueness of weak solution to elliptic Kolmogorov PDE]
\label{unique_rem}
\label{rem_unique} Our most complete uniqueness result for the Kolmogorov elliptic equation in \eqref{eq_el} with the interaction kernels satisfying \eqref{K1_cond} or $\eqref{K3_cond}$ is proved in \cite{Ki_Orlicz} on the torus, see Remark \ref{Orlicz_rem}.
Speaking of $\mathbb R^d$, let us first say a few words about the case of very sub-critical strengths of interactions.

\begin{definition}
\label{def_w1}
If $K$ satisfies \eqref{K1_cond} with $\kappa<(\frac{N}{N-1})^2$,
we say that $u$ is a weak solution of \eqref{eq_el} if $u \in W^{1,2}\cap L^\infty$ and
\begin{equation*}
%\label{w1}
\mu \langle u,\varphi \rangle + \langle \nabla u,\nabla \varphi\rangle +  \frac{1}{N}\bigg\langle 
\sum_{i=1}^N \sum_{j=1,  j \neq i}^N K_{ij}(x_i-x_j) \cdot \nabla_{x_i}u,\varphi\bigg\rangle=\langle f,\varphi\rangle
\end{equation*}
for every $\varphi \in W^{1,2}$. (Recall that in \eqref{eq_el} the initial function is bounded, so the solution is bounded as well.) 
\end{definition}

\begin{definition}
\label{def_w2}
If $K$ satisfies \eqref{K3_cond} with $\kappa_+<2\frac{N}{N-1}$, then $u$ is a weak solution to \eqref{eq_el} if 
$u \in W^{1,2} \cap L^\infty$ and
\begin{align*}
\mu \langle u,\varphi \rangle & + \langle \nabla u,\nabla \varphi\rangle \\
&-  \frac{1}{N} 
\sum_{i=1}^N \sum_{j=1,  j \neq i}^N \biggl[\langle {\rm div\,}K_{ij}(x_i-x_j) u,\varphi\rangle + \langle K_{ij}(x_i-x_j) u,\nabla_{x_i}\varphi\rangle\biggr]=\langle f,\varphi\rangle
\end{align*}
for all $\varphi \in W^{1,2}$.
\end{definition}

In both cases the uniqueness of the weak solution follows upon applying Lemmas \ref{p_1}, \ref{p_2} and the Lax-Milgram theorem in $L^2$, i.e.\,we can take $p=2$ in Theorem \ref{thmK1_0}(\textit{iii}), Theorem \ref{thmK1}(\textit{ii}).

In the general case,  we need to consider \eqref{eq_el} in $L^p$, where $p$ is as in Theorem \ref{thmK1_0}(\textit{iii}) or Theorem \ref{thmK1}(\textit{ii}). 
In this regard, we refer to \cite{S} for the definition of weak solution and results on weak solutions of parabolic equations in $L^p$.

\medskip

If $K$ satisfies \eqref{K2_cond}, then we can prove that $u$ constructed in Theorem \ref{thmK1}(\textit{i}) is a weak solution of \eqref{eq_el} e.g.\,in the following sense.

\begin{definition}
\label{def_w3}
If $K$ satisfies \eqref{K2_cond}, 
then we say that $u$ is a weak solution of \eqref{eq_el} if
$u \in W^{1,2}_{\loc} \cap L^\infty$ and
\begin{align*}
%\label{w3}
\mu \langle u,\varphi \rangle  & + \langle \nabla u,\nabla \varphi\rangle \\
&-  \frac{1}{N} 
\sum_{i=1}^N \sum_{j=1,  j \neq i}^N \biggl[\langle {\rm div\,}K_{ij}(x_i-x_j) u,\varphi\rangle + \langle K_{ij}(x_i-x_j) u,\nabla_{x_i}\varphi\rangle\biggr]=\langle f,\varphi\rangle
\end{align*}
for all $\varphi \in W^{1,2}_{\loc} \cap L^\infty_c$ ($L^\infty_c$ are bounded functions with compact supports).
\end{definition}

The latter is a way to establish a link between function $u$ defined by \eqref{u_def} and the formal elliptic equation \eqref{eq_el}. However, the proof of uniqueness of such a weak solution under condition \eqref{K2_cond} remains elusive.
Still, we can prove that $u$ given by \eqref{u_def} is unique among weak solutions that can be obtained via a reasonable regularization of $K$, see Theorem \ref{markov_thm}(\textit{v}). Alternatively, we can restrict our attention to the subclass of weakly form-bounded vector fields, see Remark \ref{mult_class},
and prove uniqueness via the Lax-Milgram theorem in the triple of Bessel potential spaces $$\mathcal W^{\scriptscriptstyle \frac{1}{2},2} \subset \mathcal W^{-\scriptscriptstyle  \frac{1}{2},2} \subset \mathcal W^{-\scriptscriptstyle  \frac{3}{2},2},$$ 
where $\mathcal W^{p,\alpha}:=(\lambda-\Delta)^{-\frac{\alpha}{2}}L^p$ (rather than the standard $W^{1,2} \subset L^2 \subset W^{-1,2}$), see \cite{KiS_JDE}, although this comes at the cost of requiring that the weak form-bound of $K$ (and therefore its multiplicative form-bound $\kappa_0$, cf.\,\eqref{incl_12}) must be strictly less than $1$.
\end{remark}

\begin{remark}[Sufficient condition for multiplicative form-boundedness]
\label{mult_class}
A Borel measurable vector field $K:\mathbb R^d \rightarrow \mathbb R^d$ is said to belong to the class of weakly form-bounded vector fields $\mathbf{F}_\kappa^{\scriptscriptstyle 1/2}$ if $|K| \in L^1_{\loc}$ and
\begin{equation}
\label{wfbd}
\||K|^{\frac{1}{2}}(\lambda-\Delta)^{-\frac{1}{4}}\|_{2 \rightarrow 2} \leq \sqrt{\kappa} \quad (\text{$L^2 \rightarrow L^2$ operator norm})
\end{equation}
for some $\lambda>0$. We have
\begin{equation}
\label{incl_12}
\mathbf{F}_{\kappa}^{\scriptscriptstyle 1/2} \subset \mathbf{MF}_{\kappa}.
\end{equation}
Indeed, if $K \in \mathbf{F}_{\kappa}^{\scriptscriptstyle 1/2}$, then, arguing as in \cite{S}, we have
\begin{align*}
\langle|K|\varphi,\varphi\rangle & \leq \kappa \langle (\lambda-\Delta)^{\frac{1}{2}}\varphi,\varphi \rangle \leq \kappa \|(\lambda-\Delta)^{\frac{1}{2}}\varphi\|_2\|\varphi\|_2 \\
& = \kappa\sqrt{\|\nabla \varphi\|_2^2 + \lambda\|\varphi\|_2^2}\|\varphi\|_2  \leq \kappa \|\nabla \varphi\|_2\|\varphi\|_2 + \kappa\sqrt{\lambda}\|\varphi\|_2^2,
\end{align*}
i.e.\,$K \in \mathbf{MF}_\kappa$. 

The class $\mathbf{F}_{\kappa}^{\scriptscriptstyle 1/2}$ (and therefore $\mathbf{MF}_{\kappa}$) contains the largest possible up to the strict inequality in $\varepsilon>0$ scaling-invariant  Morrey class $M_{1+\varepsilon}$, i.e.\,if
\begin{equation*}
\|K\|_{M_{1+\varepsilon}}:=\sup_{r>0, x \in \mathbb R^d} r\biggl(\frac{1}{|B_r|}\int_{B_r(x)}|K|^{1+\varepsilon}dx \biggr)^{\frac{1}{1+\varepsilon}}<\infty,
\end{equation*}
then $M_{1+\varepsilon} \subset \mathbf{F}_{\kappa}^{\scriptscriptstyle 1/2}$ with $\kappa=c(d,\varepsilon)\|K\|_{M_{1+\varepsilon}}$ \cite{A}. 

It is easily seen that $M_{1+\varepsilon}$ is larger than $M_{2}$, which, in turn, contains $\mathbf{F}_\kappa$. 
That said, we also need to control the form-bounds. In fact, we have
\begin{equation}
\label{incl_13}
\mathbf{F}_{\kappa} \subset \mathbf{F}_{\sqrt{\kappa}}^{\scriptscriptstyle 1/2}.
\end{equation}
Indeed, rewriting $K \in \mathbf{F}_\kappa$ as
$$
\||K|(\lambda-\Delta)^{-\frac{1}{2}}\|_{2 \rightarrow 2} \leq \sqrt{\kappa}
$$ 
(with $\lambda=c_\kappa/\kappa$), we obtain the required result by applying the Heinz inequality. In \eqref{incl_13} we have a proper inclusion because the class of weakly form-bounded vector fields also contains the Kato class of vector fields $\||K|(\lambda-\Delta)^{-\frac{1}{2}}\|_\infty \leq \sqrt{\kappa}$ while $\mathbf{F}_\kappa$ does not (see \cite{Ki_a_new_approach, KiS_theory}).
\end{remark}

\begin{remark}[Stronger hypersurface singularities]
\label{unique_strong_rem}
We refer to \cite{Ki_Morrey} and \cite{KiS_brownian} for the results on weak well-posedness of general SDE \eqref{sde1} with  drift $b \in \mathbf{F}_\delta^{\scriptscriptstyle 1/2}$ or with  $b$ in the time-inhomogeneous analogue of the Morrey class $M_{1+\varepsilon}$ (it is quite close to $\mathbf{F}_\delta^{\scriptscriptstyle 1/2}$ but also covers critical-order singularities of the drift in time). This allows to treat 
$$
b(x)=\pm \frac{cx}{\big||x|-1\big|^{1-\gamma}}\, \eta(x),
$$
for a fixed $0<\gamma<1$, $c \in \mathbb R$ and $0 \leq \eta \in C_c^\infty$, i.e.\,hypersurface singularities that are essentially twice more singular than \eqref{hyp}. That said, in these results the assumptions on the form-bound $\delta$ are dimension-dependent. So, if we were to apply them to particle system \eqref{syst3} we would arrive at the assumptions on the strength of interaction $\kappa$ that degenerate to zero as $N \uparrow \infty$ (i.e.\,are of the form $\kappa<\frac{C}{(Nd)^2}$), which is not what we are after in the present work.

\end{remark}

\bigskip

\section{SDEs with general singular drifts}
\label{gen_sect}
Theorem \ref{thmK1_0} and Theorem \ref{thmK1} (excluding the heat kernel bound in assertion (\textit{iv})) are proved by embedding particle system \eqref{syst3} in general SDE \eqref{sde1} via \eqref{b_def0_} and then applying Theorem \ref{markov_thm} below. The general SDE, which we consider here, to lighten the notations, in $\mathbb R^d$ instead of $\mathbb R^{Nd}$, is
\begin{equation}
\label{sde1}
Y(t)=y-\int_0^t b\big(Y(s)\big)ds + \sqrt{2}B(t), \quad t \in [0,T], \quad y \in \mathbb R^d,
\end{equation}
where a priori $b \in [L^1_{\loc}]^d$, $\{B(t)\}_{t \geq 0}$ is a Brownian motion in $\mathbb R^d$, 

\medskip

Let us add that once we put the drifts in particle system \eqref{syst3} using Remark \ref{ext_rem}, we can obtain most of Theorem \ref{markov_thm} from Theorems \ref{thmK1_0} and \ref{thmK1} by taking all $K_{ij}=0$. So, thanks to Lemmas \ref{p_1} and \ref{p_2}, our results on particle systems and general SDEs are, to a large extent, equivalent. Of course, the heat kernel bound in Theorem \ref{thmK1}(\textit{iv}) is specific to particle systems. Also, Theorem \ref{unique_thm} does not have at the moment a particle system counterpart, although we believe that it is of interest on its own.

\medskip

Set
\begin{equation}
\label{b_n}
b_n:=E_{\varepsilon_n}b, \quad \varepsilon_n \downarrow 0, \quad \text{$E_\varepsilon$ is the Friedrichs mollifier}, \quad \varepsilon_n \downarrow 0.
\end{equation}

\begin{theorem}
\label{markov_thm} Assume that a Borel measurable vector field $b$ in SDE \eqref{sde1} 
satisfies
one of the following two conditions:

\begin{equation}
\label{A_0}
\tag{$\mathbb{A}_1$}
b \in \mathbf{F}_\delta \quad \text{with } \delta<4
\end{equation}
or
\begin{equation}
\label{A_1}
\tag{$\mathbb{A}_2$}
b \in \mathbf{MF}, \qquad 
\left\{
\begin{array}{l}
({\rm div\,}b)_- \in L^1+L^\infty, \\[2mm]
({\rm div\,}b)_+^{\frac{1}{2}} \in \mathbf{F}_{\delta_+} \text{ with } \delta_+<4, 
\end{array}
\right.
\quad |b|^{\frac{1+\alpha}{2}} \in \mathbf{F} 
\end{equation}
for some $\alpha>0$ fixed arbitrarily small. Then the following are true:

\begin{enumerate}[label=(\roman*)]

\item

There exists a strong Markov family $\{\mathbb P_y\}_{y \in \mathbb R^d}$ of martingale solutions of SDE \eqref{sde1}, i.e. 
$\mathbb P_{y}[\omega_0=y]=1$,
$$
\mathbb E_{y}  \int_0^T |b(\omega_t)|dt<\infty
$$
and for every $\phi \in C_c^2(\mathbb R^{d})$ the process
$$
[0,T] \ni r \mapsto \phi(\omega_r)-\phi(y) + \int_0^r (-\Delta + b \cdot \nabla) \phi(\omega_t)dt
$$
is a martingale under $\mathbb P_y$. 

\smallskip

\item

The function
\begin{equation}
\label{u_def_gen}
u(x):=\mathbb E_{\mathbb P_x}\int_0^\infty e^{-\lambda s}f(\omega_s) ds, \quad x \in \mathbb R^{d}, \quad f \in C_c^\infty(\mathbb R^{d}),
\end{equation}
where $\lambda$ is assumed to be sufficiently large,
is a locally H\"{o}lder continuous weak solution  to elliptic Kolmogorov equation
$\big(\lambda - \Delta+ b \cdot \nabla \big)u=f$
(see Remark \ref{rem_unique} for the definitions). 

\medskip

In assertions (\textit{iii}) and (\textit{iv}) we replace condition \eqref{A_1} with a somewhat more restrictive hypothesis
\begin{equation}
\label{A_3}
\tag{$\mathbb{A}_3$}
\left\{
\begin{array}{l}
b \in \mathbf{F}, \\
 ({\rm div\,}b)^{\frac{1}{2}}_+ \in \mathbf{F}_{\delta_+} \text{ with } \delta_+<4, \quad ({\rm div\,}b)^{\frac{1}{2}}_- \in L^1+L^\infty.
\end{array}
\right.
\end{equation}
If $b$ satisfies \eqref{A_0}, fix $p>\frac{2}{2-\sqrt{\delta}}$. If $b$ satisfies \eqref{A_3}, fix $p>\frac{4}{4-\delta_+}$.

\smallskip

\item (\cite{KS,S}, see also \cite{KiS_theory}) The family of operators $$P_tf(x):=\mathbb E_{\mathbb P_x}[f(\omega_t)], \quad t>0, \quad f \in C_c^\infty$$ admits extension by continuity to a strongly continuous quasi contraction Markov semigroup on $L^p$, say, $P_t=:e^{-t\Lambda_p}$, such that $$
\|e^{-t\Lambda_p}\|_{p \rightarrow q} \leq ce^{\omega t}t^{-\frac{d}{2}(\frac{1}{p}-\frac{1}{q})}, \quad p \leq q \leq \infty
$$
for appropriate constants $c$ and $\omega$.

The following approximation uniqueness result holds:
for any sequence of bounded smooth vector fields $$b_n \rightarrow b \quad \text{ in } [L^2_{\loc}]^d$$
that do not increase the form-bounds on $b$ in \eqref{A_1} or \eqref{A_3}, the classical solutions $v_n$ to 
$$
\big(\partial_t-\Delta + b_n \cdot \nabla\big)v_n=0, \quad v_n(0)=f \in C_c^\infty
$$
converge to the same limit $e^{-t\Lambda_p}f$ in $L^p$ loc.\,uniformly in $t \geq 0$.

\smallskip

\item

 The resolvent $(\mu+\Lambda_p(b))^{-1}$ has Feller property, i.e.\,for each  $\mu$ greater than some $\mu_0>0$ it extends by continuity to a bounded linear operator on $C_\infty$:
$$R_\mu(b):=\bigl[ (\mu+\Lambda_p(b))^{-1} \upharpoonright L^p \cap C_\infty \bigr]^{\rm clos}_{C_\infty \rightarrow C_\infty} \in \mathcal B(C_\infty).
$$
Moreover,
$$
R_\mu(b_n) \rightarrow R_\mu(b) \quad \text{strongly in } C_\infty, \quad \mu \geq \mu_0,
$$
where $R_\mu(b_n)$ coincides with the resolvent of $-\Delta + b_n\cdot \nabla$ on $C_\infty$, $n=1,2,\dots$

\medskip

\item (Approximation uniqueness) If $b$ satisfies 
\begin{equation}
\label{b9}
\left\{
\begin{array}{l}
b \in \mathbf{MF}, \\
({\rm div\,}b)_- \in L^1+L^\infty, \quad
({\rm div\,}b)_+^{\frac{1}{2}} \in \mathbf{F}_{\delta_+} \text{ with } \delta_+<2,
\end{array}
\right.
\end{equation}
then there exist generic constants $\lambda_0>0$ and $\varkappa \in ]0,1[$ such that if, additionally, $|b| \in L^{2-\varkappa}$, then, for any sequence $b_n$ of bounded smooth vector fields satisfying \eqref{b9} with the same constants as $b$ and such that
$
b_n \rightarrow b$ in $L^{2-\varkappa},
$
the sequence of the classical solutions $u_n$ to $$(\lambda-\Delta + b_n\cdot \nabla)u_n=f, \quad f \in C_c^\infty,\quad \lambda \geq \lambda_0$$ converge in $L^2$ to the same limit which, thus, does not depend on a particular choice of $\{b_n\}$.
\end{enumerate}

\end{theorem}

In (\textit{iii}) we can consider $b_n$ defined by \eqref{b_n}. 

\medskip

In subsequent paper \cite{KiS_feller_4} we strengthen assertion (\textit{iv}) by constructing strongly continuous Feller semigroup in $C_\infty$ for $-\Delta + b \cdot \nabla$, $b \in \mathbf{F}_\delta$, for all $\delta<4$.

\medskip

Combining  assertion (\textit{v}) with Theorem \ref{thm1}, one can further show that the limit $u$ is locally H\"{o}lder continuous and is a weak solution of $(\mu-\Delta + b\cdot \nabla)u=f$. Furthermore, one can construct the corresponding strongly continuous semigroup in $L^p$, but we will not pursue this here. We included assertion (\textit{v}) to emphasize that there is, in principle, nothing pathological from the point of view of a posteriori estimates about condition $b \in \mathbf{MF}$ compared to $b \in \mathbf{F}$ used in (\textit{iii}), (\textit{iv}). That said, the proof of assertion (\textit{v}) is somewhat more involved than the proof of the uniqueness of the limit in (\textit{iii}) and uses Gehring's lemma. It is ultimately an $L^2$ argument, hence the need for a more restrictive condition $\delta_+<2$ in \eqref{A_3}, see Remark \ref{gehr_rem} for details.

\medskip

We need assertion (\textit{iv}) of Theorem \ref{markov_thm} in the proof of the following uniqueness result.

\begin{theorem}[Krylov-type estimates and conditional uniqueness]
\label{unique_thm}

Assume that a Borel measurable  vector field $b$ 
satisfies one of the following conditions:
\begin{equation}
\label{C_1}
\tag{$\mathbb{B}_1$}
b \in \mathbf{F}_\delta  \text{ with } \delta<\bigg(\frac{2}{q}\bigg)^2 \wedge 1 \text{ for some } q>(d-2) \vee 2
\end{equation}
or
\begin{equation}
\label{C_2}
\tag{$\mathbb{B}_2$}
\left\{
\begin{array}{l}
b \in \mathbf{F}_\delta \cap [W^{1,1}_{\loc}(\mathbb R^d)]^d \text{ for some finite }\delta, \text{ has symmetric Jacobian } Db:=(\nabla_k b_i)_{k,i=1}^d, \\[2mm]
\text{the normalized eigenvectors $e_j$ and eigenvalues $\lambda_j \geq 0$ of the negative part of $Db -\frac{{\rm div\,}b}{q}I$} \\[2mm]
 \text{for some $q>(d-2) \vee 2$ satisfy } \sqrt{\lambda_j}e_j \in \mathbf{F}_{\nu_j} \text{with } \nu:=\sum_{j=1}^d \nu_j<\frac{4(q-1)}{q^2}.
\end{array}
\right.
\end{equation}

\noindent Then the following are true for the strong Markov family of martingale solutions of SDE \eqref{sde1} 
constructed in Theorem \ref{markov_thm}:

\smallskip

{\rm (\textit{i})} For every $y \in \mathbb R^d$, martingale solution $\mathbb P_y$ satisfies  Krylov-type bound
\begin{equation}
\label{krylov_est}
\mathbb E_{\mathbb P_y} \int_0^\infty e^{-\lambda s}|\mathsf{g}f|(\omega_s)ds \leq C\|\mathsf{g}|f|^{\frac{q}{2}}\|_2^\frac{2}{q}, \quad \forall\,\mathsf{g} \in \mathbf{F}, \quad \forall \,f \in C_c,
\end{equation}
for $q>(d-2) \vee 2$ close to $(d-2) \vee 2$, for all $\lambda$ sufficiently large.

\medskip

{\rm ($i'$)} $\{\mathbb P_y\}_{y \in \mathbb R^d}$ is the only Markov family of martingale solutions to \eqref{sde1} that satisfies Krylov-type bound in {\rm (\textit{i})}.

\medskip

{\rm (\textit{ii})} For every $y \in \mathbb R^d$, $\mathbb P_x$ satisfies Krylov bound: 
\begin{equation}
\label{krylov_2}
\mathbb E_{\mathbb P_y} \int_0^\infty e^{-\lambda s}|f(\omega_s)|ds \leq C\|f\|_{\frac{qd}{d+q-2}}, \quad \forall \,f \in C_c
\end{equation} 
for all $\lambda$ sufficiently large.

\medskip

{\rm ($ii'$)} We make \eqref{krylov_2} more restrictive by selecting $q$ close to $(d-2) \vee 2$, so that in \eqref{krylov_2} $\frac{qd}{d+q-2}=\frac{d}{2-\varepsilon } \wedge \frac{2}{1-\varepsilon}$ for some $\varepsilon>0$ small. Let $\{\mathbb P^2_y\}_{t \in \mathbb R^d}$ be another Markov family of martingale solutions for \eqref{sde1} that satisfies Krylov bound
\begin{equation}
\label{kr_3}
\mathbb E_{\mathbb P^2_y} \int_0^\infty e^{-\lambda s}|f|(\omega_s)ds \leq C\|f\|_{\frac{d}{2-\varepsilon } \wedge \frac{2}{1-\varepsilon}}, \quad \forall \,f \in C_c
\end{equation}
(one such family exists, it is $\{\mathbb P_y\}_{y \in \mathbb R^d}$ from above). Assume additionally that, for some  $\varepsilon_1 \in ]\varepsilon,1[$ we have
$$(1+|x|^{-2})^{-\beta}|b|^{\frac{d}{2-\varepsilon_1} \vee \frac{2}{1-\varepsilon_1}} \in L^1$$
for some $\beta>\frac{d}{2}$ fixed arbitrarily large, and either \eqref{C_1} holds with $\delta<\frac{4}{q_\ast^2} \wedge 1$, where
$$
q_\ast:=\left\{
\begin{array}{ll}
\frac{d-2}{\varepsilon_1-\varepsilon} & \text{ if } d \geq 4, \\
2\bigl(\frac{1}{3(\varepsilon_1-\varepsilon)} \vee 1 \bigr)  & \text{ if } d =3,
\end{array}
\right.
$$
or \eqref{C_2} holds with $q=q_*$ and $\nu<\frac{4(q_*-1)}{q_*^2}$. Then
 $\{\mathbb P^2_y\}_{y \in \mathbb R^d}$ coincides with $\{\mathbb P_y\}_{y \in \mathbb R^d}$ from above.

\end{theorem}

Some remarks are in order.

\begin{enumerate}[label=\arabic*.]
\item
 In the last assertion, the uniqueness class of martingale solutions satisfying Krylov bound \eqref{kr_3}, which depends on our choice of $\varepsilon$, determines the extra conditions that one needs to impose on $b$. Note that if in \eqref{C_1} one has $|b| \in L^d$, or in \eqref{C_2} the eigenvectors have entries in $L^d$, then the form-bounds $\delta$ and $\nu_j$ ($j=1,\dots,d$), respectively, can be chosen arbitrarily small, in which case these extra conditions on $b$ are trivially satisfied.

\smallskip

\item

In \eqref{C_2} we require Jacobian $Db$ to be symmetric, so $b=\nabla V$ for some potential $V$. 
\end{enumerate}

Let us illustrate condition \eqref{C_2} with the following example.

\begin{example}
\label{ex_unique}
Let $d \geq 4$.
Let $$b(x)=-\sqrt{\delta}\frac{d-2}{2}\frac{x}{|x|^2},$$ a drift that pushes solution $Y_t$ of \eqref{sde1} away from the origin. Put for brevity $c:=\sqrt{\delta}\frac{d-2}{2}>0$. We have ${\rm div\,}b=-c(d-2)|x|^{-2}$ and $\nabla_j b_i=c\bigl[-|x|^{-2}\delta_{ij} + 2x_ix_j|x|^{-4}\bigr]$. Therefore, for every $\xi=(\xi_i) \in \mathbb R^d$,
\begin{align*}
\xi^{\scriptsize{\top}}(Db-\frac{{\rm div\,}b}{q}I)\xi & =\sum_{i,j=1}^d \xi_j [(\nabla_j b_i)-\frac{1}{q}({\rm div\,}b)\delta_{ji}] \xi_i=c\biggl(\frac{d-2}{q}-1\biggr)|x|^{-2}|\xi|^2 + 2c|x|^{-4}(x\cdot \xi)^2 \\
& = \xi^{\scriptsize{\top}}(B_+-B_-)\xi,
\end{align*}
where $B_+ \geq 0$ is the matrix with entries $2cx_ix_j|x|^{-4}$, and $B_-:=-c(\frac{d-2}{q}-1)|x|^{-2} I \geq 0$. Thus, constant $\nu$ in condition \eqref{C_2} can be made as small as needed by selecting $q>d-2$ sufficiently close to $d-2$, and so for this $b$ condition \eqref{C_2} can be satisfied for any strength of repulsion from the origin.
\end{example}

In the previous example it is crucial that we can select $q$ as close to $d-2$ as needed. By working in the parabolic setting we could obtain a stronger uniqueness result, i.e.\,for every fixed initial point $x$. However, the parabolic setting requires us to take $q>d$ \cite{KiM_JDE}, and so the previous example becomes invalid: we have to require smallness of $\delta$ even in the case of repulsion.

\begin{remark}[On some other classes of singular vector fields arising in the study of singular SDEs and PDEs]
\label{sing_rem}

1.~A number of important results on the regularity theory of $-\Delta + b \cdot \nabla$   was obtained in \cite{Z_supercritical,Z_supercritical2} which considered supercritical form-boundedness type conditions on $b$ (in the context of the study of 3D Navier-Stokes equations). These are conditions of the type:
there exists $\varepsilon \in ]0,1]$ such that
$|b| \in L^{1+\varepsilon}_{\loc}([0,\infty[ \times \mathbb R^d)$ and 
\begin{align}
 \int_0^\infty \int_{\mathbb R^d} |b(t,\cdot)|^{1+\varepsilon}\xi^2(t,\cdot) dt   \leq \delta \int_0^\infty\|\nabla \xi(t,\cdot)\|_2^2 dt & +\int_0^\infty g(t)\|\xi(t,\cdot)\|_2^2dt \notag \\
& \text{for all } \xi \in C^\infty_c([0,\infty[ \times \mathbb R^d) \label{s}
\end{align}
for some $\delta>0$ and $0 \leq g \in L^1_{\loc}([0,\infty[)$ under, necessarily, some assumptions on ${\rm div\,}b$ which cannot be too singular. 
Here super-criticality/criticality/sub-criticality refer to how the assumptions on $b$ behave under rescaling the equation.
In the super-critical case one has to sacrifice a large portion of the regularity theory 
of $-\Delta + b \cdot \nabla$ including the usual Harnack inequality and 
the H\"{o}lder continuity of solutions to the elliptic and parabolic equations. 
See also counterexample to the uniqueness in law for SDEs with super-critical drifts in \cite{Zh}. 
However, some parts of the theory, such as the local boundedness of weak solutions, can be salvaged, 
see cited papers, see also recent developments in \cite{AD,HZ}. 
Let us also note that if we were to specify \eqref{s} to the critical case 
when the usual regularity theory is still valid, then we would need to take $\varepsilon=1$, 
i.e.\,we would obtain condition \eqref{K3_cond}, but not more general condition \eqref{K2_cond}.

\smallskip

2.~As was noted in \cite{KiS_sharp}, after supplementing \eqref{s} with condition $({\rm div\,}b)^{\frac{1}{2}}_+ \in \mathbf{F}_{\nu}$ for some $\nu<4$, one can still prove the existence of a martingale solution to  SDE \eqref{sde1}. In the present paper we work in the critical setting which allows us to preserve most of the important results in the regularity theory of elliptic equations that do not involve estimates on the second order derivatives of the solutions (which are destroyed by the form-boundedness assumptions),  and thus allows to prove, e.g.\,the strong Markov property, approximation uniqueness or conditional weak uniqueness results for particle system \eqref{syst2} (see, however, \cite{HZ} who constructed a Markov family of weak solutions in a super-critical setting using a selection procedure). 

\medskip

Let us also add that above super-criticality  refers to the assumptions on $b$, but not on $({\rm div\,}b)_+$. In fact, as the counterexample to weak solvability of \eqref{Z_eq} with the model attracting drift shows, one cannot go beyond the form-boundedness assumption (critical) on $({\rm div\,}b)_+$.

\end{remark}

\bigskip

\section{Regularity results for PDEs}

\textbf{1.~}To prove Theorem \ref{markov_thm}, we need the regularity results of Theorems \ref{thm1}, \ref{thm2} for non-homogeneous elliptic equations \eqref{eq22}, \eqref{eq7}, respectively In these results we assume additionally that the coefficients of \eqref{eq22}, \eqref{eq7} are bounded and smooth. However, importantly, the constants in the regularity estimates are \textit{generic}, i.e.\,they depend only on the structure parameters of the equation such as the dimension $d$, constant term $\lambda$ and the form-bounds of the vector fields (but not on the smoothness or boundedness of the coefficients).

\medskip

Theorem \ref{thm1} $\Rightarrow$ strong Markov property in Theorem \ref{markov_thm}.

\medskip

Theorem \ref{thm2} $\Rightarrow$ existence of martingale solutions in Theorem \ref{markov_thm}.

\begin{theorem}[H\"{o}lder continuity of solutions]
\label{thm1} 
Let $b:\mathbb R^d \rightarrow \mathbb R^d$ be a bounded smooth vector field such that either 
\begin{equation}
\tag{this is \eqref{A_0}}
b \in \mathbf{F}_\delta \text{ with } \delta<4
\end{equation}
or
\begin{equation}
\label{A_1prime}
\tag{$\bar{\mathbb{A}}_{2}$}
\left\{
\begin{array}{l}
b \in \mathbf{MF},\\[2mm]
({\rm div\,}b_+)^{\frac{1}{2}} \in \mathbf{F}_{\delta_+} \text{ with } \delta_+<4, 
\end{array}
\right.
\end{equation}
where  ${\rm div\,}b={\rm div\,}b_+ - {\rm div\,}b_-$ for some bounded smooth functions  
${\rm div\,}b_\pm \geq 0$. Let $f \in C_c^\infty$, $\lambda \geq 0$. Then the classical solution $u$ to non-homogeneous equation
\begin{equation}
\label{eq22}
\big(\lambda -\Delta + b \cdot \nabla\big)u=f
\end{equation}
is locally H\"{o}lder continuous with generic constants that also depend on $\|f\|_\infty$.
\end{theorem}
(The difference between \eqref{A_1prime} and  \eqref{A_1} is that in the former case we do not require ${\rm div\,}b_\pm$ to be positive and negative parts of ${\rm div\,}b$, which are continuous but not necessarily smooth.)

The fact that the constants are generic is of course the main point of Theorem \ref{thm1}. 

\medskip

Define weight
\begin{equation}
\label{rho_def}
\rho(y)=(1+k|y|^{2})^{-\frac{d}{2}-1}, \quad y \in \mathbb R^d,
\end{equation}
where constant constant $k>0$ will be chosen sufficiently small. This weight has property
\begin{equation}
\label{est_rho}
|\nabla \rho| \leq \left(\frac{d}{2}+1\right)\sqrt{k} \rho.
\end{equation}

For a fixed $x \in \mathbb R^d$, put $\rho_x(y):=\rho(y-x).$

\begin{theorem}[Embedding theorem]
\label{thm2}
Let $b,\mathsf{h}:\mathbb R^d \rightarrow \mathbb R^d$ be bounded smooth vector fields  such that
\begin{equation}
\label{first_cond}
\text{$b \in \mathbf{F}_\delta$ with $\delta<4$, \quad $\mathsf{h} \in \mathbf{F}_\chi \text{ with } \chi<\infty$}
\end{equation}
or
\begin{equation}
\label{second_cond}
\left\{
\begin{array}{l}
b \in \mathbf{MF}_\delta \text{ for some } \delta<\infty,\\[2mm]
({\rm div\,}b_+)^{\frac{1}{2}} \in \mathbf{F}_{\delta_+} \text{ with } \delta_+<4,
\end{array}
\right.
\quad |\mathsf{h}|^{\frac{1+\gamma}{2}} \in \mathbf{F}_\chi \quad \text{ with } \chi<\infty
\end{equation}
 for some $\gamma>0$ fixed arbitrarily small, where ${\rm div\,}b={\rm div\,}b_+ - {\rm div\,}b_-$ for some 
bounded smooth functions ${\rm div\,}b_\pm \geq 0$.
In the former case, fix $p>\frac{2}{2-\sqrt{\delta}}$, $p \geq 2$, and in the latter case fix $p>\frac{4}{4-\delta_+}$, $p' \leq 1+\gamma$, $p \geq 2$.

Then, for a fixed $1<\theta<\frac{d}{d-2}$, there exist generic constants  $\lambda_0$, $k$ (in $\rho$), $C$ and $\beta \in ]0,1[$ such that the classical solution $u$ to non-homogeneous equation
\begin{equation}
\label{eq7}
(\lambda-\Delta + b \cdot \nabla)u=|\mathsf{h}f|, \quad f \in C_c^\infty
\end{equation}
on $\mathbb R^d$
satisfies for all $\lambda \geq \lambda_0 \vee 1$:
\begin{align}
\|u\|_\infty  \leq C \sup_{x \in \frac{1}{2}\mathbb Z^d}& \biggl((\lambda-\lambda_0)^{-\frac{1}{p\theta}} \langle \big(\mathbf{1}_{|\mathsf{h}|>1} + |\mathsf{h}|^{p\theta}\mathbf{1}_{|\mathsf{h}| \leq 1}\big)|f|^{p\theta}\rho_x\rangle^{\frac{1}{p\theta}}\notag \\
& + \lambda^{-\frac{\beta}{p}}\langle \big(\mathbf{1}_{|\mathsf{h}|>1} + |\mathsf{h}|^{p\theta'}\mathbf{1}_{|\mathsf{h}| \leq 1}\big)|f|^{p\theta'}\rho_x\rangle^{\frac{1}{p\theta'}} \biggr). \label{unif_est_u}
\end{align}
\end{theorem}

\medskip

\textbf{2.~}The following result is needed to prove conditional weak uniqueness in Theorem \ref{unique_thm}.

\begin{theorem}[Gradient bounds]
\label{thm_grad}
Assume that a bounded smooth vector field $b$ satisfies either condition \eqref{C_1} of Theorem \ref{unique_thm} or
\begin{equation}
\label{C_2prime}
\tag{$\bar{\mathbb{B}}_{2}$}
\left\{
\begin{array}{l}
b \in \mathbf{F}_\delta \cap [W^{1,1}_{\loc}(\mathbb R^d)]^d \text{ with finite }\delta \text{ and symmetric Jacobian } Db:=(\nabla_k b_i)_{k,i=1}^d, \\[2mm]
\text{and the negative part $B_-$ of matrix $Db -\frac{{\rm div\,}b}{q}I$ for some $q>(d-2) \vee 2$} \\[2mm]
 \text{satisfies } \langle B_-h,h\rangle \leq \nu \langle |\nabla|h||^2\rangle+c_\nu\langle |h|^2\rangle \text{ for some } \nu<\frac{4(q-1)}{q^2}.
\end{array}
\right.
\end{equation} 
Then the following are true:

\smallskip

{\rm (\textit{i})}  For every $\mathsf{g} \in \mathbf{F}$ there exist generic constants $\mu_0$ and $K$ such that, for every $\mu>\mu_0$, the classical solution $u$ to elliptic equation
$
(\mu- \Delta - b\cdot \nabla)u=|\mathsf{g}|f$,  $f \in C^\infty_c$,
satisfies
\begin{align*}
\|\nabla|\nabla u|^{\frac{q}{2}} \|_{2} & \leq K\|\mathsf{g}|f|^{\frac{q}{2}}\|_2.
\end{align*}

\smallskip

{\rm (\textit{ii})}  There exist generic constants $\mu_0$ and $K$ such that the classical solution $u$ to elliptic equation
$(\mu- \Delta - b\cdot \nabla)u=f$, $f \in C^\infty_c$,
satisfies, for all $\mu>\mu_0$, 
\begin{align*}
\|\nabla|\nabla u|^{\frac{q}{2}} \|_{2} & \leq K\|f\|^{\frac{q}{2}}_{\frac{qd}{d+q-2}}.
\end{align*}
\end{theorem}

In both assertions, by the Sobolev embedding theorem, $u$ is H\"{o}lder continuous, although using these gradient bounds to prove H\"{o}lder continuity would be excessive: in Theorem \ref{thm1} we arrive at the same conclusion directly, using De Giorgi's method, under less restrictive conditions on $b$. 

\medskip

For example, condition \eqref{C_2prime} holds if condition \eqref{C_2} of Theorem \ref{unique_thm} is satisfied, see Lemma \ref{B_lem}.

\medskip

Assuming that $b \in \mathbf{F}_\delta$, $\delta<(\frac{2}{d-2})^2 \wedge 1$, \cite{KS} proved estimate
\begin{equation}
\label{KS_grad}
\|\nabla |\nabla u|^{\frac{q}{2}}\|_2 \leq K\|f\|_q, \quad q \in ](d-2) \vee 2, \frac{2}{\sqrt{\delta}}[
\end{equation}
for solution $u$ to elliptic equation $(\mu- \Delta - b\cdot \nabla)u=f$. This estimate was used in \cite{KS} to construct the corresponding to $-\Delta - b \cdot \nabla$ Feller semigroup via a Moser-type iteration procedure.
The norm $\|f\|_q$ in the right-hand side of \eqref{KS_grad} does not allow to obtain the uniqueness result in Theorem \ref{unique_thm} from \eqref{KS_grad}, unless $b$ satisfies additional assumption $|b| \in L^{(d-2) \vee 2}$. Still, the argument of \cite{KS} can be modified to include a weaker norm of $f$, and this is what we do in the proof of Theorem \ref{thm_grad}. In particular, we use the test function 
\begin{equation}
\label{test_f}
\phi=-\nabla \cdot (\nabla u|\nabla u|^{q-2})
\end{equation}
of \cite{KS}.
In more recent literature one can find other test functions that give gradient bounds on $u$ of the same type as in Theorem \ref{thm_grad} (moreover, these test functions work for larger classes of equations). However, importantly, test function \eqref{test_f} yields the least restrictive assumptions on form-bounds $\delta$ and $\nu$, which are in the focus of the present paper. In fact, one can argue that by multiplying the elliptic equation by test function \eqref{test_f}  and integrating by parts, one differentiates the equation in the optimal direction $\frac{\nabla u}{|\nabla u|}$. We refer to \cite{Ki_survey} for more detailed discussion and references.

\bigskip

\section{Smooth approximation of form-bounded vector fields}

\label{approx_sect}

Let $b \in [L^1_{\loc}(\mathbb R^d)]^d$.  Define
$$
b_\varepsilon:=E_\varepsilon b, \quad \varepsilon>0,
$$
where, recall, $E_\varepsilon h$  denotes the Friedrichs mollifier of function (or vector field) $h$, see Section \ref{notations_sect} for the definition.

\begin{lemma}
\label{fb_approx_lem}
If $b \in \mathbf{F}_\delta$, then
the following is true:

\smallskip

1. $b_\varepsilon\in [L^\infty(\mathbb R^d) \cap C^\infty(\mathbb R^d)]^d$, $b_\varepsilon \rightarrow b$ in $[L^2_{\loc}(\mathbb R^d)]^d$ as $\varepsilon \downarrow 0$.

\smallskip

2.~$b_\varepsilon \in \mathbf{F}_\delta$ with the same constant $c_\delta$ (thus, independent of $\varepsilon$). 
\end{lemma}

\begin{proof}
1.~The smoothness of $b_\varepsilon$ and the convergence follow from the standard properties of Friedrichs mollifiers, so it remains to prove that $|b_\varepsilon| \in L^\infty$. By H\"{o}lder's inequality, $$|b_\varepsilon(x)|\leq \sqrt{E_\varepsilon |b|^2(x)}=\sqrt{\langle\gamma_\varepsilon(x-\cdot)|b(\cdot)|^2\rangle},$$ so
\begin{align*}
|b_\varepsilon(x)| & \leq \big\langle |b(\cdot)|^2\gamma_\varepsilon(x-\cdot)\big\rangle^{\frac{1}{2}} \\
& (\text{we apply the hypothesis $b \in \mathbf{F}_\delta$}) \\
& \leq \big(\delta \big\langle \big|\nabla \sqrt{\gamma_\varepsilon(x-\cdot)}\big|^2\big\rangle + c_\delta\big)^{\frac{1}{2}}=(C\varepsilon^{-2}+c_\delta)^{\frac{1}{2}}.
\end{align*}
Hence $|b_\varepsilon| \in L^\infty$ for each $\varepsilon>0$.

2. Put $\varphi_\varepsilon=\sqrt{E_\varepsilon \varphi^2}$, $\varphi \in W^{1,2}$. 
Then
\begin{align*}
\|b_\varepsilon \varphi\|_2^2 &\leq \langle E_\varepsilon |b|^2,\varphi^2\rangle=\|b \varphi_\varepsilon\|^2_2 \leq \delta\|\nabla \varphi_\varepsilon\|_2^2+c_\delta\|\varphi_\varepsilon\|_2^2, 
\end{align*}
where
\begin{align}
\|\nabla \varphi_\varepsilon\|_2 & =\big\|\frac{E_\varepsilon(\varphi \nabla \varphi)}{\sqrt{E_\varepsilon \varphi^2}}\big\|_2 \notag \\
& (\text{we apply Cauchy-Schwartz' inequality}) \notag \\
&\leq \|\sqrt{E_\varepsilon|\nabla \varphi|^2}\|_2=\|E_\varepsilon|\nabla \varphi|^2\|_1^\frac{1}{2}
 \leq \|\nabla \varphi\|_2 \label{phi_eps}
\end{align}
and, clearly, $\|\varphi_\varepsilon\|_2 \leq \|\varphi\|_2$.
\end{proof}

\begin{lemma} If $b \in \mathbf{MF}_\delta$, then
the following is true:

\smallskip

1. $b_\varepsilon\in [L^\infty(\mathbb R^d) \cap C^\infty(\mathbb R^d)]^d$, $b_\varepsilon \rightarrow b$ in $[L^1_{\loc}(\mathbb R^d)]^d$.

\smallskip

2.~$b_\varepsilon \in \mathbf{MF}_\delta$ with the same $c_\delta$. 
\end{lemma}

\begin{proof} 1.~We only need to prove $|b_\varepsilon| \in L^\infty$. By $b \in \mathbf{MF}_\delta$, for all $x \in \mathbb R^d$,
\begin{align*}
|b_\varepsilon(x)|   \leq \big\langle |b(\cdot)|\gamma_\varepsilon(x-\cdot)\big\rangle 
 \leq \delta \big\langle \big|\nabla \sqrt{\gamma_\varepsilon(x-\cdot)}\big|^2\big\rangle^{\frac{1}{2}} + c_\delta=C\varepsilon^{-1}+c_\delta.
\end{align*}

2.~Let $\varphi_\varepsilon=\sqrt{E_\varepsilon \varphi^2}$, $\varphi \in W^{1,2}$. We have
\begin{align*}
\langle |b_\varepsilon| \varphi,\varphi\rangle & =\langle |b| E_\varepsilon \varphi^2\rangle=\langle |b| \varphi_\varepsilon^2\rangle  \leq \delta\|\nabla \varphi_\varepsilon\|_2\|\varphi_\varepsilon\|_2+c_\delta\|\varphi_\varepsilon\|_2^2,
\end{align*}
where, repeating the previous proof, $\|\nabla \varphi_\varepsilon\|_2 \leq \|\nabla \varphi\|_2$, $\|\varphi_\varepsilon\|_2 \leq \|\varphi\|_2$.
\end{proof}

Assume that ${\rm div\,}b \in L^1_{\loc}$. 
We can represent ${\rm div\,}b_\varepsilon = E_\varepsilon {\rm div\,}b$ as
$$
{\rm div\,}b_\varepsilon  = {\rm div\,}b_{\varepsilon,+} - {\rm div\,}b_{\varepsilon,-},
$$
where
$${\rm div\,}b_{\varepsilon,+}:=E_\varepsilon ({\rm div\,}b)_+, \quad {\rm div\,}b_{\varepsilon,-}:=E_\varepsilon ({\rm div\,}b)_-.$$ 
Note that smooth functions ${\rm div\,}b_{\varepsilon,\pm} \geq 0$ are in general greater than the positive and the negative parts $({\rm div\,}b_\varepsilon)_+:={\rm div}\,b_\varepsilon \vee 0$, $({\rm div\,}b_\varepsilon)_-:=-({\rm div}\,b_\varepsilon \wedge 0)$ of ${\rm div\,}b_\varepsilon$.

\begin{lemma} If $({\rm div\,}b)_+ \in \mathbf{F}_{\delta_+}$, $({\rm div\,}b)_- \in L^1+L^\infty$, then
the following is true:

1.~${\rm div\,}b_{\varepsilon,+} \in L^\infty \cap C^\infty$, ${\rm div\,}b_{\varepsilon,+} \rightarrow ({\rm div\,}b)_+$ in $L^1_{\loc}$ as $\varepsilon \downarrow 0$.

2.~${\rm div\,}b_{\varepsilon,+} \in \mathbf{F}_{\delta_+}$ with the same $c_{\delta_+}$ as the one for $b$. 
\end{lemma}

\begin{proof} The first statement follows from the properties of Friedrichs mollifiers and the following estimate (we use notations from the previous proof): 
for every $x \in \mathbb R^d$,
\begin{align*}
{\rm div\,}b_{\varepsilon,+}(x)   \leq \big\langle ({\rm div\,}b)_+(\cdot)\gamma_\varepsilon(x-\cdot)\big\rangle \leq \delta_+ \big\langle \big|\nabla \sqrt{\gamma_\varepsilon(x-\cdot)}\big|^2\big\rangle + c_{\delta_+}=C\varepsilon^{-2}+c_\delta,
\end{align*}
Let us prove the second statement:
\begin{align*}
\langle {\rm div\,}b_{\varepsilon,+} \varphi, \varphi\rangle = \langle  ({\rm div\,}b)_+ \varphi_\varepsilon^2\rangle \leq \delta_+\|\nabla \varphi_\varepsilon\|_2^2 + c_{\delta_+}\|\varphi_\varepsilon\|_2^2 \leq \delta_+ \|\nabla \varphi\|_2^2 + c_{\delta_+}\|\varphi\|_2^2.
\end{align*}
\end{proof}

Finally, we will need

\begin{lemma} 
\label{cl4}
If $|\mathsf{h}|^{\frac{1+\gamma}{2}} \in \mathbf{F}_\chi$ ($\gamma>0$), then the following is true:

1.~$\mathsf{h}_\varepsilon:=E_\varepsilon \mathsf{h} \in [L^\infty(\mathbb R^d) \cap C^\infty(\mathbb R^d)]^d$, $\mathsf{h}_\varepsilon \rightarrow \mathsf{h}$ in $[L^1_\loc(\mathbb R^d)]^d$ as $\varepsilon \downarrow 0$,

2.~$|\mathsf{h}_\varepsilon|^{\frac{1+\gamma}{2}}\in \mathbf{F}_\chi$ with the same $c_\chi$.
\end{lemma}
\begin{proof}By H\"{o}lder's inequality, $|\mathsf{h}_\varepsilon|^{1+\gamma} \leq E_\varepsilon |\mathsf{h}|^{1+\gamma}$, so
$\langle |\mathsf{h}_\varepsilon|^{1+\gamma} \varphi^2\rangle \leq \langle |\mathsf{h}|^{1+\gamma},\varphi_\varepsilon^2\rangle$, where, recall, $\varphi_\varepsilon=\sqrt{E_\varepsilon \varphi^2}$,  $\varphi \in W^{1,2}$. Now we apply $|\mathsf{h}|^{\frac{1+\gamma}{2}} \in \mathbf{F}_\chi$ and use $\|\nabla \varphi_\varepsilon\|_2 \leq \|\nabla \varphi\|_2$, $\|\varphi_\varepsilon\|_2 \leq \|\varphi\|_2$.
\end{proof}

\section{Proofs of Lemmas \ref{p_1} and \ref{p_2}}
\label{decomp_sect}

Recall: $b=(b_1,\dots,b_N):\mathbb R^{Nd} \rightarrow \mathbb R^{Nd}$ is defined by
\begin{equation*}
b_i(x):= \frac{1}{N}\sum_{j=1, j \neq i}^N K_{ij}(x_i-x_j), \quad x=(x_1,\dots,x_n) \in \mathbb R^{Nd}, \quad 1 \leq i \leq N.
\end{equation*}
Below $|\cdot|$ denotes, depending on the context, the Euclidean norm in $\mathbb R^{Nd}$ or $\mathbb R^d$. 
In this section, $\langle\,,\,\rangle $ is the integration over $\mathbb R^{Nd}$.

\begin{proof}[Proof of Lemma \ref{p_1}]
We have
\begin{align*}
|b(x)|^{2} & \leq \sum_{i=1}^N |b_i(x)|^{2} \leq  \sum_{i=1}^N  \biggl(\frac{1}{N}  \sum_{j=1,j \neq i}^N |K_{ij}(x_i-x_j)|\biggr)^{2} \\ 
&\leq \sum_{i=1}^N \frac{N-1}{N^{2}} \sum_{j=1,j \neq i}^N |K_{ij}(x_i-x_j)|^{2}.
\end{align*}
Therefore,
$
\langle |b|^2\varphi^2\rangle \leq \sum_{i=1}^N \frac{N-1}{N^{2}} \sum_{j=1,j \neq i}^N \langle |K_{ij}(x_i-x_j)|^{2}\varphi^2\rangle,
$
where, denoting by $\bar{x}$ vector $x$ with component $x_i$ removed, we estimate
\begin{align*}
\langle|K_{ij}(x_i-x_j)|^2 \varphi^2 \rangle & = \int_{\mathbb R^{(N-1)d}}\int_{\mathbb R^d} |K_{ij}(x_i-x_j)|^2\varphi^2(x_i,\bar{x})dx_i d\bar{x} \\
& (\text{we use $K_{ij} \in \mathbf{F}_{\kappa}(\mathbb R^d)$ in $x_i$ variable}) \\
& \leq \kappa \int_{\mathbb R^{(N-1)d}}\int_{\mathbb R^d} |\nabla_{x_i}\varphi(x_i,\bar{x})|^2dx_i d\bar{x} + c_{\kappa} \int_{\mathbb R^{Nd}}\varphi^2 dx \\
& = \kappa \langle |\nabla_{x_i}\varphi|^2 \rangle + c_{\kappa} \langle \varphi^2 \rangle.
\end{align*}
Hence
$
\langle |b|^2\varphi^2\rangle \leq \frac{(N-1)^2}{N^2}\kappa \langle |\nabla\varphi|^2 \rangle + \frac{(N-1)^2}{N}c_{\kappa} \langle \varphi^2 \rangle,
$
as claimed.
\end{proof}

\begin{proof}[Proof of Lemma \ref{p_2}] Let us first prove \eqref{b_est2}.
We have
\begin{align}
\langle |b|\varphi^2 \rangle & \leq \sum_{i=1}^N \langle |b_i| \varphi^2 \rangle \leq \sum_{i=1}^N \frac{1}{N}\sum_{j=1, j \neq i}^N \langle |K_{ij}(x_i-x_j)|\varphi^2 \rangle. \label{eta_psi2}
\end{align}
Denoting by $\bar{x}$ the variable $x$ with component $x_i$ removed, we estimate
\begin{align*}
& \langle |K_{ij}(x_i-x_j)|\varphi^2 \rangle  = \int_{\mathbb R^{(N-1)d}}\int_{\mathbb R^d} |K_{ij}(x_i-x_j)|\varphi^2(x_i,\bar{x}) dx_i d\bar{x} \\
& (\text{apply $K_{ij} \in \mathbf{MF}_\kappa(\mathbb R^d)$ in $x_i$ variable}) \\
& \leq \int_{\mathbb R^{(N-1)d}} \biggl[\biggl(\int_{\mathbb R^d} |\nabla_{x_i}\varphi(x_i,\bar{x})|^2dx_i\biggr)^{\frac{1}{2}} \biggl(\int_{\mathbb R^d} \varphi^2(x_i,\bar{x})dx_i \biggr)^{\frac{1}{2}}+c_\kappa \int_{\mathbb R^d} \varphi^2(x_i,\bar{x})dx_i\biggr] d\bar{x} \\
& \leq \kappa\langle|\nabla_{x_i}\varphi|^2 \rangle^{\frac{1}{2}} \langle \varphi^2 \rangle^{\frac{1}{2}} + c_\kappa\langle \varphi^2 \rangle.
\end{align*}
Therefore,
\begin{align*}
\sum_{i=1}^N \frac{1}{N}\sum_{j=1, j \neq i}^N \langle|K_{ij}(x_i-x_j)|\varphi^2 \rangle 
& \leq 
\sum_{i=1}^N \frac{N-1}{N} \biggl[\kappa\langle|\nabla_{x_i}\varphi|^2 \rangle^{\frac{1}{2}}\langle \varphi^2 \rangle^{\frac{1}{2}} + c_\kappa\langle\varphi^2 \rangle \biggr] \\
& \leq \frac{N-1}{N}\sqrt{N}\kappa\langle |\nabla\varphi|^2 \rangle^{\frac{1}{2}}\langle \varphi^2 \rangle^{\frac{1}{2}} + (N-1)c_\kappa\langle \varphi^2 \rangle.
\end{align*}
Applying these estimates in \eqref{eta_psi2}, we obtain \eqref{b_est2}.

Next, we prove \eqref{div_est}.
We have ${\rm div\,}b(x)=\sum_{i=1}^N \frac{1}{N}\sum_{j=1,j \neq i}^N ({\rm div}\,K_{ij})(x_i-x_j).$ 
So,
$$
({\rm div\,}b)_+ =
\sum_{i=1}^N \frac{1}{N}\sum_{j=1,j \neq i}^N ({\rm div\,} K_{ij})_+(x_i-x_j).
$$
Hence, by 
$({\rm div\,}K_{ij})_+^{\frac{1}{2}} \in \mathbf{F}_{\kappa_+}(\mathbb R^d)$ (note that this condition is linear in $({\rm div}K_{ij})_+$),
\begin{align*}
\langle ({\rm div\,}b)_+, \varphi^2 \rangle \leq \frac{N-1}{N}\kappa_+\langle |\nabla \varphi|^2 \rangle + (N-1)c_{\kappa_+}\langle\varphi^2 \rangle,
\end{align*}
i.e.\,we have proved \eqref{div_est} for $({\rm div\,}b)_+$.

Now, we prove \eqref{gamma_est}. Recall that $\alpha \in [0,1]$. We have, using Jensen's inequality,
\begin{align*}
|b|^{1+\alpha} & \leq \sum_{i=1}^N |b_i(x)|^{1+\alpha} \leq  \sum_{i=1}^N  \biggl(\frac{1}{N}  \sum_{j=1,j \neq i}^N |K_{ij}(x_i-x_j)|\biggr)^{1+\alpha} \\ 
&\leq  \sum_{i=1}^N \frac{(N-1)^\alpha}{N^{1+\alpha}} \sum_{j=1,j \neq i}^N |K_{ij}(x_i-x_j)|^{1+\alpha}.
\end{align*}
Therefore, applying   
$|K_{ij}|^{\frac{1+\alpha}{2}} \in \mathbf{F}_{\sigma} (\mathbb R^d)$, we obtain
\begin{align*}
\langle |b|^{1+\alpha}\varphi^2 \rangle  \leq \frac{(N-1)^{1+\alpha}}{N^{1+\alpha}}\sigma \langle|\nabla \varphi|^2 \rangle  + N\frac{(N-1)^{1+\alpha}}{N^{1+\alpha}}c_{\sigma}  \langle \varphi^2 \rangle,
\end{align*}
which gives us \eqref{gamma_est}.
\end{proof}

\begin{remark}
\label{rem_M} In Remark \ref{ext_rem} we promised to prove that vector field $b^M:\mathbb R^{Nd} \rightarrow \mathbb R^{Nd}$ defined by \eqref{def_M} is in $\mathbf{F}_{\delta^M}(\mathbb R^{Nd})$ with $\delta^M=\mu$, $c_{\delta^M}=Nc_\mu$. Here is the proof:
$$
|b^M(x)|^{2} = \sum_{i=1}^N |M_i(x_i)|^{2},
$$
where (recall that $\langle\,,\,\rangle$ is the integration over $\mathbb R^{Nd}$, $\bar{x}$ is vector $x \in \mathbb R^{Nd}$ with component $x_i \in \mathbb R^d$)
\begin{align*}
& \langle |M_i(x_i)|\varphi^2 \rangle  = \int_{\mathbb R^{(N-1)d}}\int_{\mathbb R^d} |M_i(x_i)|\varphi^2(x_i,\bar{x}) dx_i d\bar{x} \\
& (\text{we use $M_i \in \mathbf{F}_{\mu}(\mathbb R^d)$ in $x_i$ variable}) \\
& \leq \mu \int_{\mathbb R^{(N-1)d}}\int_{\mathbb R^d} |\nabla_{x_i}\varphi(x_i,\bar{x})|^2dx_i d\bar{x} + c_{\mu} \int_{\mathbb R^{Nd}}\varphi^2 dx  = \mu \langle |\nabla_{x_i}\varphi|^2 \rangle + c_{\mu} \langle \varphi^2 \rangle.
\end{align*}
So,
$$
\langle |b^M(x)|^{2}\varphi^2\rangle  = \sum_{i=1}^N\langle |M_i(x_i)|^{2}\varphi^2\rangle \leq \mu \langle |\nabla\varphi|^2 \rangle + Nc_{\mu} \langle \varphi^2 \rangle,
$$
as claimed.
\end{remark}

\bigskip

\section{Proof of Theorem \ref{thm1}}

\label{thm1_proof_sect}

With the exception of our proof of Proposition \ref{c_prop} which is, modulo its homogeneous $L^2$ version in \cite{KiV}, is new, we follow closely De Giorgi's method as it is presented in \cite[Ch.\,7]{G} with appropriate modifications to account for our somewhat different definition of $L^p$ De Giorgi's classes ($\equiv$\,functions satisfying the inequality in Proposition \ref{c_prop}, see discussion in Section \ref{thm1_rem}). If we were to take $p=2$ (obviously, at the cost of imposing very sub-optimal constrains on the form-bounds), then,  once Proposition \ref{c_prop} is established, we could simply refer the reader to \cite[Ch.\,7]{G}.

\medskip

If $b$ satisfies \eqref{A_0} $\Rightarrow$  fix throughout this proof $p>\frac{2}{2-\sqrt{\delta}}$, $p \geq 2$.

\medskip

If $b$ satisfies \eqref{A_1prime} $\Rightarrow$ fix $p>\frac{2}{4-\delta_+}$, $p \geq 2$.

\medskip

Fix some $R_0 \leq 1$. Recall that $u$ is a classical solution of equation \eqref{eq22} in $\mathbb R^d$,  i.e.
$$
\big(\lambda -\Delta + b \cdot \nabla\big)u=f, 
$$
where $f \in C_c^\infty$, $\lambda \geq 0$.

\begin{proposition}[Caccioppoli's inequality \textnumero 1]
\label{c_prop} 
Let $v:=(u-k)_+$, $k \in \mathbb R$.
For all $0<r<R \leq R_0$, 
\begin{align*}
\|(\nabla v^{\frac{p}{2}}) \mathbf{1}_{B_{r}}\|_2^2 \leq \frac{K_1}{(R-r)^{2}} \|v^{\frac{p}{2}} \mathbf{1}_{B_{R}}\|_2^2 + K_2\||f-\lambda u|^{\frac{p}{2}} \mathbf{1}_{u>k} \mathbf{1}_{B_{R}}\|_2^2
\end{align*}
for generic constants $K_1$, $K_2$ (in particular, independent of $k$ or $r$, $R$). 
\end{proposition}
\begin{proof}
Let us first carry out the proof in the more difficult case when $b$ satisfies condition  \eqref{A_1prime}.
We will attend to the case when $b$ satisfies \eqref{A_0} in Remark \ref{first_cond_rem}.

 We fix a family of $[0,1]$-valued smooth cut-off functions $\{\eta=\eta_{r_1,r_2}\}_{0<r_1<r_2<R}$ on $\mathbb R^d$ satisfying
$$
\eta=\left\{
\begin{array}{ll}1 & \text{ in } B_{r_1}, \\
0 & \text{ in } \mathbb R^d - \bar{B}_{r_2},
\end{array}
\right.
$$
and
\begin{equation}
\label{e1}
\frac{|\nabla \eta|^2}{\eta} \leq \frac{c}{(r_2-r_1)^2}\mathbf{1}_{B_{r_2}},
\end{equation}
\begin{equation}
\label{e2}
\sqrt{|\nabla \eta|} \leq \frac{c}{\sqrt{r_2-r_1}}\mathbf{1}_{B_{r_2}},
\end{equation}
\begin{equation}
\label{e3}
|\nabla \sqrt{|\nabla \eta|}| \leq \frac{c}{(r_2-r_1)^{\frac{3}{2}}}\mathbf{1}_{B_{r_2}}
\end{equation}
for some constant $c$. 
For instance, one can take, for $r_1 \leq |y| \leq r_2$, 
$$\eta(y):=1-\int_1^{1+\frac{|y|-r_1}{r_2-r_1}}\varphi(s)ds, \quad \text{where } \varphi(s):=Ce^{-\frac{1}{\frac{1}{4}-(s-\frac{3}{2})^2}}, \quad \supp \varphi=[1,2],$$
with constant $C$ adjusted to $\int_1^2 \varphi(s)ds=1$.

We put equation \eqref{eq22} in the form
$$
(-\Delta + b \cdot \nabla)(u-k)= f-\lambda u
$$
(keeping in mind that even if $\lambda>0$ solution $u$ of \eqref{eq22} satisfies a priori estimate $\|u\|_\infty \leq \lambda^{-1}\|f\|_\infty$, so the $\|\cdot\|_\infty$ norm of the right-hand side of the previous identity is bounded by $2\|f\|_\infty$),
multiply it by $v^{p-1}\eta$ and integrate, obtaining
\begin{align*}
\frac{4(p-1)}{p^2}\langle \nabla v^{\frac{p}{2}},(\nabla v^{\frac{p}{2}})\eta\rangle & +  \frac{2}{p}\langle \nabla v^{\frac{p}{2}},v^{\frac{p}{2}}\nabla \eta \rangle \\
&  + \frac{2}{p}\langle b \cdot \nabla v^{\frac{p}{2}},v^{\frac{p}{2}}\eta\rangle \leq \langle |f-\lambda u|,v^{p-1}\eta\rangle.
\end{align*}
Then, applying quadratic inequality (fix some $\epsilon>0$), we have
\begin{align}
\left(\frac{4(p-1)}{p}-\frac{4}{p}\epsilon\right)\langle |\nabla v^{\frac{p}{2}}|^2\eta\rangle & \leq \frac{p}{4\epsilon}\big\langle v^p\frac{|\nabla \eta|^2}{\eta} \big\rangle  - 2\langle b \cdot \nabla v^{\frac{p}{2}},v^{\frac{p}{2}}\eta\rangle + p\langle |f-\lambda u|,v^{p-1}\eta\rangle  \label{i12}  \\
& (\text{we are integrating by parts in the second term}) \notag \\
& \leq \frac{p}{4 \epsilon}\big\langle v^p\frac{|\nabla \eta|^2}{\eta} \big\rangle  + \langle bv^{\frac{p}{2}},v^{\frac{p}{2}}\nabla \eta\rangle + \langle {\rm div\,}b,v^p\eta\rangle + p\langle |f-\lambda u|,v^{p-1}\eta\rangle \notag \\
&  =:I_1 + I_2 + I_3+I_4. \notag
\end{align}

By \eqref{e1},
$$
I_1  \leq \frac{cp}{4\epsilon (r_2-r_1)^2}\| v^{\frac{p}{2}} \mathbf{1}_{B_{r_2}}\|_2^2.
$$

By \eqref{A_1prime},
\begin{align*}
I_2 & \leq \langle |b|v^{\frac{p}{2}},v^{\frac{p}{2}}|\nabla \eta|\rangle \\
& \leq \delta \|\nabla (v^{\frac{p}{2}}\sqrt{|\nabla \eta|})\|_2\|v^{\frac{p}{2}}\sqrt{|\nabla \eta|}\|_2 + c_\delta  \|v^{\frac{p}{2}}\sqrt{|\nabla \eta|}\|^2_2 \\
& \leq \delta \bigg(\|(\nabla v^{\frac{p}{2}})\sqrt{|\nabla \eta|}\|_2 + \|v^{\frac{p}{2}} \nabla \sqrt{|\nabla \eta|}\|_2 \bigg) \, \|v^{\frac{p}{2}}\sqrt{|\nabla \eta|}\|_2 + c_\delta \|v^{\frac{p}{2}}\sqrt{|\nabla \eta|}\|^2_2.
\end{align*}
Hence, using \eqref{e2}, \eqref{e3}, we obtain 
\begin{align*}
I_2  & \leq \delta c \bigg(\frac{1}{\sqrt{r_2-r_1}}\|(\nabla v^{\frac{p}{2}})\mathbf{1}_{B_{r_2}}\|_2 + \frac{1}{(r_2-r_1)^{\frac{3}{2}}}\|v^{\frac{p}{2}}\mathbf{1}_{B_{r_2}}\|_2 \bigg)\, \frac{1}{\sqrt{r_2-r_1}}\|v^{\frac{p}{2}}\mathbf{1}_{B_{r_2}}\|_2 \\
&  + \frac{c_\delta c}{r_2-r_1}\|v^{\frac{p}{2}}\mathbf{1}_{B_{r_2}}\|_2^2.
\end{align*}
Thus, since $r_2-r_1<1$,
\begin{align*}
I_2  & \leq \frac{C_1}{r_2-r_1}\|(\nabla v^{\frac{p}{2}})\mathbf{1}_{B_{r_2}}\|_2\| v^{\frac{p}{2}}\mathbf{1}_{B_{r_2}}\|_2 + C_1 \biggl(1+\frac{1}{(r_2-r_1)^2} \bigg)\|v^{\frac{p}{2}}\mathbf{1}_{B_{r_2}}\|^2_2
\end{align*}
for appropriate constant $C_1$. 

Next, by \eqref{A_1prime},
\begin{align*}
I_3 & \leq \langle {\rm div\,}b_+,v^p\eta\rangle \leq \delta_+\|\nabla (v^{\frac{p}{2}}\sqrt{\eta})\|_2^2 + c_{\delta_+}\|v^{\frac{p}{2}}\sqrt{\eta}\|_2^2 \\
& = \delta_+\|(\nabla v^{\frac{p}{2}})\sqrt{\eta}+v^{\frac{p}{2}}\frac{\nabla \eta}{\sqrt{\eta}}\|_2^2 + c_{\delta_+}\|v^{\frac{p}{2}}\sqrt{\eta}\|_2^2 \\
& \leq \delta_+\bigg( (1+\epsilon_1)\|(\nabla v^{\frac{p}{2}})\sqrt{\eta}\|_2^2+ \bigg(1+\frac{1}{\epsilon_1}\bigg)\|v^{\frac{p}{2}}\frac{\nabla \eta}{\sqrt{\eta}}\|_2^2\bigg) + c_{\delta_+}\|v^{\frac{p}{2}}\sqrt{\eta}\|_2^2 \qquad (\epsilon_1>0)\\
& \leq \delta_+ (1+\epsilon_1) \|(\nabla v^{\frac{p}{2}})\sqrt{\eta}\|_2^2 + \frac{c_1}{(r_2-r_1)^2}\| v^{\frac{p}{2}} \mathbf{1}_{B_{r_2}}\|^2_{2}, \qquad c_1:=\delta_+\bigg(1+\frac{1}{\epsilon_1}\bigg)c + c_{\delta_+}.
\end{align*}

Finally, we estimate using Young's inequality ($p'=\frac{p}{p-1}$):
\begin{align*}
\frac{1}{p}I_4 & \leq \frac{\varepsilon_2^{p'}}{p'} \langle v^p \eta\rangle + \frac{1}{p\varepsilon_2^p} \langle |f-\lambda u|^p \mathbf{1}_{v>0}\eta \rangle \qquad (\varepsilon_2>0).
\end{align*}
Applying the estimates on $I_1$-$I_4$ in \eqref{i12}, we obtain
\begin{align}
\||\nabla v^{\frac{p}{2}}|\mathbf{1}_{B_{r_1}}\|_2^2  & \leq \frac{C_1}{r_2-r_1}\|(\nabla v^{\frac{p}{2}})\mathbf{1}_{B_{r_2}}\|_2\| v^{\frac{p}{2}}\mathbf{1}_{B_{R}}\|_2 \notag \\
& + C_2 \biggl(1+\frac{1}{(r_2-r_1)^2} \bigg)\|v^{\frac{p}{2}}\mathbf{1}_{B_{R}}\|^2_2 + C_3\| |f-\lambda u|^\frac{p}{2} \mathbf{1}_{v>0} \mathbf{1}_{B_{R}}\|_2^2. \label{proto_02}
\end{align}
Divide \eqref{proto_02} by $\|v^{\frac{p}{2}}\mathbf{1}_{B_{R}}\|^2_2$:
\begin{align}
\frac{\|(\nabla v^{\frac{p}{2}})\mathbf{1}_{B_{r_1}}\|_2^2}{\|v^{\frac{p}{2}}\mathbf{1}_{B_{R}}\|^2_2}  & \leq \frac{C_1}{r_2-r_1}\frac{\|(\nabla v^{\frac{p}{2}})\mathbf{1}_{B_{r_2}}\|_2}{\|v^{\frac{p}{2}}\mathbf{1}_{B_{R}}\|_2} + C_2 \biggl(1+\frac{1}{(r_2-r_1)^2} \bigg) + C_3S^2, \label{proto_cacc2}
\end{align}
where
$$
S^2:=\frac{\||f-\lambda u|^\frac{p}{2} \mathbf{1}_{v>0}\mathbf{1}_{B_{R}}\|_2^2}{\|v^{\frac{p}{2}}\mathbf{1}_{v>0} \mathbf{1}_{B_{R}}\|^2_2}.
$$
Inequality \eqref{proto_cacc2} is the pre-Caccioppoli inequality that we will now iterate.

Put
$$
a^2_n:= \frac{\|(\nabla v^{\frac{p}{2}})\mathbf{1}_{B_{R-\frac{R-r}{2^{n-1}}}}\|_2^2}{\|v^{\frac{p}{2}}\mathbf{1}_{v>0} \mathbf{1}_{B_{R}}\|^2_2},
$$
the inequality \eqref{proto_cacc2} yields
$$
 a_n^2 \notag  \leq C(R-r)^{-1}2^{n} a_{n+1} + C^2 (R-r)^{-2}2^{2n} + C^2 S^2
$$
with constant $C$ independent of $n$.
We multiply this inequality by $(R-r)^2$ and divide by $C^2 2^{2n}$. Then, setting $y_n:=\frac{(R-r) a_n}{C 2^{n}},$ we obtain
\begin{equation}
\label{y_int}
y_n^2 \leq y_{n+1} + 1+(R-r)^2S^2
\end{equation}
for all $n=1,2,\dots$ 
A priori, all
$a_n$'s are bounded by a non-generic constant
$
\|(\nabla v^{\frac{p}{2}})\mathbf{1}_{B_{R}}\|_2/ \|v^{\frac{p}{2}}\mathbf{1}_{B_{R}}\|_2<\infty,
$
so $\sup_n y_n<\infty$.
Therefore, we can iterate
 \eqref{y_int} and thus estimate all $y_n$, $n=1,2,\dots$, via nested square roots $1+(R-r)^2S^2+\sqrt{1+(R-r)^2S^2 + \sqrt{\dots}}$, 
obtaining
$$y_n^2 \leq 3+2(R-r)^2 S^2, \quad n=1,2,\dots$$
Taking $n=1$, 
we arrive at 
$
a_1 \leq K_1 (R-r)^{-2}+K_2S^2
$
for appropriate constants $K_1$ and $K_2$, i.e.
$$
\frac{\|(\nabla v^{\frac{p}{2}})\mathbf{1}_{B_{r}}\|_2^2}{\|v^{\frac{p}{2}}\mathbf{1}_{B_{R}}\|^2_2}  \leq K_1 (R-r)^{-2} + K_2 \frac{\||f|^\frac{p}{2}\mathbf{1}_{v>0}  \mathbf{1}_{B_{R}}\|_2^2}{\|v^{\frac{p}{2}}\mathbf{1}_{B_{R}}\|^2_2},
$$
as claimed.

\begin{remark}
\label{first_cond_rem}
If $b$ satisfies condition \eqref{A_0}, then we can work with somewhat simpler cutoff functions $\eta \in C_c^\infty$,
$\eta =1$ in $B_{r_1}$, $\eta = 0$ in $\mathbb R^d \setminus B_{r_2}$, i.e.\,$|\nabla \eta| \leq c_1(r_2-r_1)^{-1}$, $|\Delta \eta| \leq c_2(r_2-r_1)^{-2}$, and we do not need to integrate by parts in order to estimate the second term in the RHS of \eqref{i12}. Instead, we only need to apply quadratic inequality:
$$
2|\langle b \cdot \nabla v^{\frac{p}{2}},v^{\frac{p}{2}}\eta\rangle| \leq \alpha \langle |\nabla v|^{\frac{p}{2}}\eta\rangle + \frac{1}{4\alpha}\langle |b|^2,v^p\eta\rangle, \quad \alpha=\frac{2}{\sqrt{\delta}}.
$$
%Regarding the terms containing $\mathsf{h}$, we simply take $\gamma=1$, which transforms condition $|\mathsf{h}|^{\frac{1+\gamma}{2}} \in \mathbf{F}_\chi$, $\chi<\infty$ from \eqref{second_cond_} into condition $\mathsf{h} \in \mathbf{F}_\chi$ in \eqref{first_cond_}, and argue as in the estimate on $I_4$ above.
\end{remark}

The proof of Proposition \ref{c_prop} is completed.
\end{proof}

\begin{lemma}[{\cite[Lemma 7.1]{G}}]
\label{dg_lemma}
If $\{z_i\}_{i=0}^\infty \subset \mathbb R_+$ is a sequence of positive real numbers  such that
$$
z_{i+1} \leq N C_0^i z^{1+\alpha}_i
$$
for some $C_0>1$, $\alpha>0$, and
$$
z_0 \leq N^{-\frac{1}{\alpha}}C_0^{-\frac{1}{\alpha^2}}.
$$
Then
$
\lim_i z_i=0.
$
\end{lemma}

\begin{lemma}[{\cite[Lemma 7.3]{G}}]
\label{lem2}
Let $\varphi(t)$ be a positive function, and assume that there exists a constant $q$ and a number $0<\tau<1$ such that for every $0<R<R_0$
$$
\varphi(\tau R) \leq \tau^\delta \varphi(R) + BR^\beta
$$
with $0<\beta<\delta$, and
$$
\varphi(t) \leq q\varphi(\tau^k R)
$$
for every $t$ in the interval $(\tau^{k+1}R,\tau^k R)$. Then, for every $0<\rho<R< R_0$, we have
$$
\varphi(\rho) \leq C \biggl(\biggl(\frac{\rho}{R}\biggr)^\beta \varphi(R) + B \rho^\beta \biggr)
$$
with constant $C$ that depends only on $q$, $\tau$, $\delta$ and $\beta$.
\end{lemma}

\begin{proposition}
\label{prop71}
For all $0<r<R \leq R_0$, 
\begin{align*}
\sup_{B_{\frac{R}{2}}} u \leq C_1 \biggl(\frac{1}{|B_R|}\langle u^p\mathbf{1}_{B_R \cap \{u>0\}} \rangle \biggr)^{\frac{1}{p}} \biggl(\frac{|B_R \cap \{u>0\}|}{|B_R|}\biggr)^\frac{\alpha}{p} + C_2 R^{\frac{2}{p}}
\end{align*}
for generic constants $C_1$, $C_2$ that also depend on $\|f-\lambda u\|_\infty~(\leq 2\|f\|_\infty)$, where $\alpha>0$ is fixed by $\alpha(\alpha+1)=\frac{2}{d}$.
\end{proposition}
\begin{proof}
Without loss of generality, $R_0=1$. Let $\frac{1}{2}<r<\rho \leq 1$. Fix $\eta \in C_c^\infty$, $\eta=1$ on $B_{r}$, $\eta=0$ on $\mathbb R^d \setminus \bar{B}_{\frac{r+\rho}{2}}$, $|\nabla \eta| \leq \frac{4}{\rho-r}$. Set $\zeta:=\eta v=\eta(u-k)_+$, $k \in \mathbb R$. 
Using H\"{o}lder's inequality and Sobolev's embedding theorem, we obtain
\begin{align*}
\|v^{\frac{p}{2}}\mathbf{1}_{B_r}\|_2^2 & \leq \|\zeta^{\frac{p}{2}}\mathbf{1}_{B_r}\|_2^2  \leq \langle \mathbf{1}_{B_{r} \cap \{u>k\}}\rangle^{\frac{2}{d}} 
 \langle \zeta^{\frac{pd}{d-2}}\mathbf{1}_{B_\frac{r+\rho}{2}}\rangle^{\frac{d-2}{d}} \\
& \leq c_1|B_{r} \cap \{u>k\}|^{\frac{2}{d}}\langle |\nabla \zeta^{\frac{p}{2}}|^2\mathbf{1}_{B_{\frac{r+\rho}{2}}}\rangle \\
& = c_1|B_{r} \cap \{u>k\}|^{\frac{2}{d}}\langle |(\nabla \eta^{\frac{p}{2}})v^{\frac{p}{2}}+\eta^{\frac{p}{2}}\nabla v^{\frac{p}{2}}|^2\mathbf{1}_{B_{\frac{r+\rho}{2}}}\rangle
\end{align*}
Hence
\begin{align*}
\|v^{\frac{p}{2}}\mathbf{1}_{B_r}\|_2^2 \leq c_2|B_{r} \cap \{u>k\}|^{\frac{2}{d}}\biggl(\frac{1}{(\rho-r)^2}\|v^{\frac{p}{2}} \mathbf{1}_{B_\frac{r+\rho}{2}} \|_2^2 + \|(\nabla v^{\frac{p}{2}})\mathbf{1}_{B_{\frac{r+\rho}{2}}}\|_2^2 \biggr).
\end{align*}
On the other hand, Proposition \ref{c_prop} yields:
\begin{align}
\label{v_cor}
\|(\nabla v^{\frac{p}{2}}) \mathbf{1}_{B_{\frac{r+\rho}{2}}}\|_2^2 \leq \frac{K_1}{(\rho-r)^{2}} \|v^{\frac{p}{2}} \mathbf{1}_{B_{\rho}}\|_2^2 + K_2\|f-\lambda u\|_\infty^p\, \big|B_\rho\cap\{u>k\}\big|,
\end{align}
so
\begin{align}
\|v^{\frac{p}{2}}\mathbf{1}_{B_r}\|_2^2 & \leq C|B_{r} \cap \{u>k\}|^{\frac{2}{d}}\biggl(\frac{1}{(\rho-r)^2}\|v^{\frac{p}{2}} \mathbf{1}_{B_\rho} \|_2^2 + \|f-\lambda u\|_\infty^p\, \big|B_\rho\cap\{u>k\}\big| \biggr) \notag \\
& \leq \frac{C|B_{\rho} \cap \{u>k\}|^{\frac{2}{d}}}{(\rho-r)^2}\|v^{\frac{p}{2}} \mathbf{1}_{B_\rho} \|_2^2 + C\|f-\lambda u\|_\infty^p |B_{\rho} \cap \{u>k\}|^{1+\frac{2}{d}}. \label{i_8}
\end{align}
Now, returning from notation $v$ to $(u-k)_+$, we note that 
if $h<k$, then $\|(u-k)^{\frac{p}{2}}\mathbf{1}_{B_\rho \cap \{u>k\}}\|_2 \leq \|(u-h)^{\frac{p}{2}}\mathbf{1}_{B_\rho \cap \{u>h\}}\|_2$ and $\|(u-h)^{\frac{p}{2}}\mathbf{1}_{B_\rho \cap \{u>h\}}\|_2^2 \geq (k-h)^p |B_\rho \cap \{u>k\}|$. Therefore, we obtain from \eqref{i_8}
\begin{align*}
\|(u-k)_+^{\frac{p}{2}}\mathbf{1}_{B_r}\|_2^2 & \leq \frac{C}{(\rho-r)^2}\|(u-h)_+^{\frac{p}{2}}\mathbf{1}_{B_\rho}\|_2^2|B_{\rho} \cap \{u>h\}|^{\frac{2}{d}} \\
& + \frac{C\|f-\lambda u\|_\infty^p}{(k-h)^p}\|(u-h)_+^{\frac{p}{2}}\mathbf{1}_{B_\rho}\|_2^2|B_{\rho} \cap \{u>h\}|^{\frac{2}{d}}.
\end{align*}
Multiplying this inequality by $|B_r \cap \{u>k\}|^\alpha~~\big(\leq \frac{1}{(k-h)^{p\alpha}} \|(u-h)_+^{\frac{p}{2}}\mathbf{1}_{B_\rho}\|_2^{2\alpha}\big)$ and using $\alpha^2 + \alpha=\frac{2}{d}$, we obtain
\begin{align*}
& \|(u-k)_+^{\frac{p}{2}}\mathbf{1}_{B_r}\|_2^2|B_r \cap \{u>k\}|^\alpha \\
&\leq C \biggl[\frac{1}{(\rho-r)^2} + \frac{\|f-\lambda u\|_\infty^p}{(k-h)^p} \biggr]\frac{1}{(k-h)^{p\alpha}} \bigl(\|(u-h)_+^{\frac{p}{2}}\mathbf{1}_{B_\rho}\|_2^2|B_\rho \cap \{u>h\}|^\alpha \bigr)^{1+\alpha}.
\end{align*}
Now, take $r:=r_{i+1}$, $\rho:=r_i$, where $r_i:=\frac{R}{2}(1+\frac{1}{2^i})$ and $k:=k_{i+1}$, $h:=k_i$, where $k_i:=\xi(1-2^{-i})$, with constant $\xi \geq R^{\frac{2}{p}}$ to be chosen later. Then, setting $$z_i=z(k_i,r_i):=\|(u-k_i)_+^{\frac{p}{2}}\mathbf{1}_{B_{r_i}}\|_2^2|B_{r_i} \cap \{u>k_i\}|^\alpha,$$ we have
$$
z_{i+1} \leq K\bigg[2^{2i} + \frac{2^{pi}R^2}{\xi^p}\bigg]\frac{1}{R^2}\frac{2^{pi\alpha}}{\xi^{p\alpha}}z_i^{1+\alpha}
$$
hence (using  $\xi \geq R^{\frac{2}{p}}$)
$$
z_{i+1} \leq 2^{p(1+\alpha)i} \frac{2K }{R^2}\frac{1}{\xi^{p\alpha}}z_i^{1+\alpha}.
$$
We apply Lemma \ref{dg_lemma}. In the notation of this lemma, $C_0=2^{p(1+\alpha)}$ and $N=\frac{2K }{R^2}\frac{1}{\xi^{p\alpha}}$. We need
$$z_0 \leq N^{-\frac{1}{\alpha}}C_0^{-\frac{1}{\alpha^2}}$$
where, recall, $z_0=\langle u^p \mathbf{1}_{B_R \cap \{u>0\}}\rangle |B_R \cap \{u>0\}|^\alpha$. The latter amounts to requiring $$\xi \geq C_1 R^{-\frac{2}{p\alpha}}z_0^{\frac{1}{p}}.$$
Take
$
\xi := R^{\frac{2}{p}} + C_1 R^{-\frac{2}{p\alpha}}z_0^{\frac{1}{p}}.
$
By Lemma \ref{dg_lemma}, $z(\xi,\frac{R}{2})=0$, i.e.\,$\sup_{\frac{R}{2}}u \leq \xi$. The claimed inequality follows.
\end{proof}

Set $${\rm osc\,}(u,R):=\sup_{y,y' \in B_R}|u(y)-u(y')|.$$

\begin{proposition}
\label{osc_prop}
Fix $k_0$ by $$2k_0=M(2R)+ m(2R):=\sup_{B_{2R}}u + \inf_{B_{2R}}u.$$ Assume that $|B_R \cap \{u>k_0\}| \leq \gamma |B_R|$ for some $\gamma<1$. If 
\begin{equation}
\label{hyp_n}
{\rm osc\,}(u,2R) \geq 2^{n+1}C R^{\frac{2}{p}},
\end{equation}
then, for $k_n:=M(2R)-2^{-n-1}{\rm osc\,}(u,2R)$,
$$
|B_{R} \cap \{u>k_n\}| \leq c n^{-\frac{d}{2(d-1)}}|B_R|.
$$
\end{proposition}
\begin{proof}
1.~For $h \in ]k_0,k[$, set $w:=(u-h)^{\frac{p}{2}}$ if $h<u<k$, set $w:=(k-h)^{\frac{p}{2}}$ if $u \geq k$, and $w:=0$ if $u \leq h.$ Note that $w=0$ in $B_R \setminus (B_R \cap \{u>k_0\})$. The measure of this set is greater than $\gamma |B_R|$, so the Sobolev embedding theorem applies and yields
\begin{align*}
(k-h)^{\frac{p}{2}}|B_R \cap \{u>k\}|^{\frac{d-1}{d}} & \leq c_1 \langle w^{\frac{d}{d-1}}\mathbf{1}_{B_R}\rangle^{\frac{d-1}{d}}  \leq c_2\langle |\nabla w| \mathbf{1}_\Delta \rangle \\
& \leq c_2 |\Delta|^{\frac{1}{2}}\langle |\nabla (u-h)^{\frac{p}{2}}|^2 \mathbf{1}_{B_R \cap \{u>h\}}\rangle^{\frac{1}{2}},
\end{align*}
where $$\Delta:=B_R \cap \{u>h\} \setminus (B_R \cap \{u>k\}).$$ 
Now, it follows from Proposition \ref{c_prop} that
\begin{align*}
\langle |\nabla (u-h)^{\frac{p}{2}}|^2 \mathbf{1}_{B_R \cap \{u>h\}}\rangle & \leq \frac{C_3}{R^2}\langle (u-h)^p \mathbf{1}_{B_{2R} \cap \{u>h\}}\rangle + C_4 |B_{2R} \cap \{u>h\}| \\
& \leq C_3 R^{d-2} (M(2R)-h)^p + C_5 R^d.
\end{align*}
For $h \leq k_n$ we have $M(2R)-h \geq M(2R)-k_n \geq CR^{\frac{2}{p}}$, where we have used \eqref{hyp_n}.
Therefore, summarizing what was written above, we have
$$
(k-h)^{\frac{p}{2}}|B_R \cap \{u>k\}|^{\frac{d-1}{d}} \leq c|\Delta|^{\frac{1}{2}}R^{\frac{d-2}{2}}(M(2R)-h)^\frac{p}{2}.
$$

2.~Select $k=k_i:=M(2R)-2^{-i-1}{\rm osc\,}(u,2R)$, $h=k_{i-1}$. Then
$$
M(2R)-h=2^{-i}{\rm osc\,}(u,2R), \quad |k-h|=2^{-i-1}{\rm osc\,}(u,2R),
$$
so
$$
|B_R \cap \{u>k_n\}|^{\frac{2(d-1)}{d}} \leq |B_R \cap \{u>k_i\}|^{\frac{2(d-1)}{d}} \leq C |\Delta_i| R^{d-2},
$$
where $\Delta_i:=B_R \cap \{u>k_i\} \setminus (B_R \cap \{u>k_{i-1}\})$.
Summing up in $i$ from $1$ to $n$, we obtain 
$$
n |B_R \cap \{u>k_n\}|^{\frac{2(d-1)}{d}} \leq C R^{d-2} |B_R \cap \{u>k_0\}| \leq C' R^{2(d-1)},
$$
and the claimed inequality follows.
\end{proof}

\subsubsection*{Proof of Theorem \ref{thm1}, completed} Fix $k_0$ by $2k_0=M(2R)+m(2R)$. Without loss of generality, $|B_R \cap \{u>k_0\}| \leq \frac{1}{2}|B_R|$ (otherwise we replace $u$ by $-u$). Set $k_n:=M(2R)-2^{-n-1}{\rm osc\,}(u,2R)>k_0$. By Proposition \ref{prop71},
\begin{align*}
\sup_{B_\frac{R}{2}}(u-k_n) & \leq  C_1 \bigl(\frac{1}{|B_R|}\langle (u-k_n)^p\mathbf{1}_{B_R \cap \{u>k_n\}} \rangle \bigr)^{\frac{1}{p}} \biggl(\frac{|B_R \cap \{u>k_n\}|}{|B_R|}\biggr)^\frac{\alpha}{p} + C_2 R^{\frac{2}{p}} \\& \leq C_1 \sup_{B_R} (u-k_n) \biggl(\frac{|B_R \cap \{u>k_n\}|}{|B_R|}\biggr)^{\frac{1+\alpha}{p}} + C_2 R^{\frac{2}{p}}
\end{align*}
We now apply Proposition \ref{osc_prop} (with, say, $C=1$). Fix $n$ by 
$$
 c n^{-\frac{d}{2(d-1)}} \leq \bigg(\frac{1}{2C_1}\bigg)^{\frac{p}{1+\alpha}}.
$$ Then, if ${\rm osc\,}(u,2R) \geq 2^{n+1} R^{\frac{2}{p}}$, we obtain from Proposition \ref{osc_prop}
$$
M\left(\frac{R}{2}\right)-k_n \leq \frac{1}{2}(M(2R)-k_n) + C_2R^{\frac{2}{p}}, 
$$
so, 
$$
M\left(\frac{R}{2}\right) \leq M(2R) - \frac{1}{2^{n+1}}{\rm osc\,}(u,2R)+\frac{1}{2}\frac{1}{2^{n+1}}{\rm osc\,}(u,2R) +  C_2R^{\frac{2}{p}},
$$
$$
M\left(\frac{R}{2}\right) - m\left(\frac{R}{2}\right) \leq M(2R) - m(2R) - \frac{1}{2}\frac{1}{2^{n+1}}{\rm osc\,}(u,2R) +  C_2R^{\frac{2}{p}}.
$$
Hence, since ${\rm osc\,}(u,2R)=M(2R)-m(2R)$,
\begin{equation}
\label{osc_1}
{\rm osc\,}\biggl(u,\frac{R}{2}\biggr) \leq \biggl(1-\frac{1}{2^{n+2}} \biggr){\rm osc\,}(u,2R)+C_2R^{\frac{2}{p}}.
\end{equation}
If ${\rm osc\,}(u,2R) \leq 2^{n+1}R^{\frac{2}{p}}$, then, clearly,
\begin{equation}
\label{osc_2}
{\rm osc\,}\left(u,\frac{R}{2}\right) \leq \biggl(1-\frac{1}{2^{n+2}} \biggr){\rm osc\,}(u,2R)+ \frac{1}{2} R^{\frac{2}{p}}.
\end{equation}
This yields the sought H\"{o}lder continuity of $u$ via Lemma \ref{lem2} with $\tau=\frac{1}{4}$, $\delta=\log_\tau(1-2^{-n-2})$ and $0<\beta<\frac{2}{p} \wedge \delta$. (Note that the second inequality in the conditions of Lemma \ref{lem2} holds if $q=1$ and $\varphi$ is non-decreasing, which is our case.)  \hfill \qed

\bigskip

\section{Proof of Theorem \ref{thm2}}

If $b$ satisfies \eqref{first_cond}, then we fix throughout the proof $p>\frac{2}{2-\sqrt{\delta}}$, $p \geq 2$. If $b$ satisfies \eqref{second_cond}, then we fix $p>\frac{2}{4-\delta_+}$, $p \geq 2$, $p' \leq 1+\gamma$,  where, recall $p':=\frac{p}{p-1}$.

\medskip

We are going to modify some parts of the proof of Theorem \ref{thm1}. But there are some important differences:

\begin{enumerate}

\item To prove Theorem \ref{thm2}, we need to obtain an $L^\infty$ bound on solution $u$ of nonhomogeneous equation \eqref{eq7}, i.e.\,estimate \eqref{unif_est_u}. At the first sight, establishing an $L^\infty$ bound seems to be easier than what we did in Theorem \ref{thm1}, i.e.\,proved H\"{o}lder continuity of solution. However, equation \eqref{eq7} is more sophisticated than the elliptic equation in Theorem \ref{thm1}: its right-hand side contains a posteriori unbounded function $|\mathsf{h}|$. (To add more details: in the proof of Theorem \ref{markov_thm} we will need to take vector field $\mathsf{h}=b_n$ and $\mathsf{h}=b_n-b_m$, but, crucially, the sought  $L^\infty$ bound on solution $u$ should not depend on $n$ or $m$. In other words, $L^\infty$ bound on solution $u$ should not depend on the boundedness or smoothness of $b$ and $\mathsf{h}$.)

\medskip

\item Since there is now a posteriori unbounded factor $|\mathsf{h}|$ in the right-hand side of \eqref{eq7}, we will need a different Caccioppoli's inequality, i.e.\,the one in Proposition \ref{c_prop2}.

\medskip

\item Although the sought $L^\infty$ bound on solution $u$ of \eqref{eq7} is a global bound on $\mathbb R^d$, we will still have to argue locally, i.e.\,work with cutoff functions. There are important reasons for this, see explanation Section \ref{about_sect}.
\end{enumerate}

\begin{proposition}[Caccioppoli's inequality \textnumero 2]
\label{c_prop2}
 Fix $R_0 \leq 1$. 
For all $0<r<R \leq R_0$ and every $k \geq 0$, the positive part $v:=(u-k)_+$ of $u-k$ satisfies
\begin{align*}
\lambda \|v^{\frac{p}{2}}\mathbf{1}_{B_{r}}\|_2^2 + \|(\nabla v^{\frac{p}{2}}) \mathbf{1}_{B_{r}}\|_2^2 & \leq \frac{K_1}{(R-r)^{2}} \|v^{\frac{p}{2}} \mathbf{1}_{B_{R}}\|_2^2 \\
&+ K_2\|\big(\mathbf{1}_{|\mathsf{h}|>1} + |\mathsf{h}|^\frac{p}{2}\mathbf{1}_{|\mathsf{h}| \leq 1}\big)|f|^{\frac{p}{2}} \mathbf{1}_{u>k} \mathbf{1}_{B_{R}}\|_2^2
\end{align*}
for generic constants $K_1$, $K_2$ (so, independent of $r$, $R$ and $k$).
\end{proposition}

\begin{remark}
Comparing the Caccioppoli inequality of Proposition \ref{c_prop2} with the one in Proposition \ref{c_prop}, one notices that we kept bounded function $\mathbf{1}_{|\mathsf{h}|>1} + |\mathsf{h}|^\frac{p}{2}\mathbf{1}_{|\mathsf{h}| \leq 1}$ in the right-hand side. We will use this function as follows. In the proof of Theorem \ref{markov_thm} we will consider, in particular, $\mathsf{h}=b_n-b_m$. We will need to show that the corresponding solution $u$ goes to zero locally uniformly on $\mathbb R^d$ as $n,m \rightarrow \infty$, which will be possible to do using \eqref{unif_est_u} precisely because we kept $\mathbf{1}_{|\mathsf{h}|>1} + |\mathsf{h}|^\frac{p}{2}\mathbf{1}_{|\mathsf{h}| \leq 1}$.
\end{remark}

\begin{proof} We extend the proof of Proposition \ref{c_prop} to the setting of Proposition \ref{c_prop2}.  Once again, first we carry out the proof in the more difficult case when $b$ and $\mathsf{h}$ satisfy condition \eqref{second_cond}, and then attend to the case when $b$ and $\mathsf{h}$ satisfy \eqref{first_cond} in Remark \ref{first_cond_rem2}.

Let $\{\eta=\eta_{r_1,r_2}\}_{0<r_1<r_2<R}$ be a family of $[0,1]$-valued smooth cut-off functions satisfying \eqref{e1}-\eqref{e3}.

From equation \eqref{eq7} we obtain, since both $\lambda$ and $k$ are non-negative,
$$
(\lambda-\Delta + b \cdot \nabla)(u-k) \leq |\mathsf{h}f|.
$$
After multiplying by $v^{p-1}\eta \geq 0$ and integrating, we obtain
\begin{align*}
\lambda \langle v^p\eta\rangle  + \frac{4(p-1)}{p^2}\langle \nabla v^{\frac{p}{2}},(\nabla v^{\frac{p}{2}})\eta\rangle & +  \frac{2}{p}\langle \nabla v^{\frac{p}{2}},v^{\frac{p}{2}}\nabla \eta \rangle \\
&  + \frac{2}{p}\langle b \cdot \nabla v^{\frac{p}{2}},v^{\frac{p}{2}}\eta\rangle \leq \langle |\mathsf{h}f|,v^{p-1}\eta\rangle.
\end{align*}
Then, applying quadratic inequality (fix some $\epsilon>0$), we have
\begin{align}
p\lambda \langle v^p\eta\rangle & + \left(\frac{4(p-1)}{p}-\frac{4}{p}\epsilon\right)\langle |\nabla v^{\frac{p}{2}}|^2\eta\rangle  \leq \frac{p}{4\epsilon}\big\langle v^p\frac{|\nabla \eta|^2}{\eta} \big\rangle  - 2\langle b \cdot \nabla v^{\frac{p}{2}},v^{\frac{p}{2}}\eta\rangle + p\langle |\mathsf{h}f|,v^{p-1}\eta\rangle  \label{i12__}  \\
& (\text{we are integrating by parts}) \notag \\
& \leq \frac{p}{4 \epsilon}\big\langle v^p\frac{|\nabla \eta|^2}{\eta} \big\rangle  + \langle bv^{\frac{p}{2}},v^{\frac{p}{2}}\nabla \eta\rangle + \langle {\rm div\,}b,v^p\eta\rangle + p\langle |\mathsf{h}f|,v^{p-1}\eta\rangle \notag =:I_1 + I_2 + I_3+I_4. \notag
\end{align}
Terms $I_1$-$I_3$ are estimated in the same way as in the proof of Proposition \ref{c_prop}. Term $I_4$ is different, so we argue as follows:
\begin{align*}
\frac{1}{p}I_4 & \leq \langle (|\mathsf{h}|\mathbf{1}_{|\mathsf{h}|>1} + |\mathsf{h}|\mathbf{1}_{|\mathsf{h}|\leq 1})|f|,v^{p-1}\eta \rangle \\
& (\text{we open the brackets and apply Young's inequality}) \\
& \leq \frac{\varepsilon_2^{p'}}{p'} \langle |\mathsf{h}|^{p'}\mathbf{1}_{|\mathsf{h}|>1}v^p \eta\rangle + \frac{1}{p\varepsilon_2^p} \langle \mathbf{1}_{|\mathsf{h}|>1}|f|^p \eta \rangle \qquad (\varepsilon_2>0)\\
& + \frac{\varepsilon_2^{p'}}{p'} \langle v^p \eta\rangle + \frac{1}{p\varepsilon_2^p} \langle |\mathsf{h}|^p\mathbf{1}_{|\mathsf{h}| \leq 1}|f|^p \eta \rangle\\
& \quad \bigg(\text{using $1+\gamma \geq p'$, we apply $|\mathsf{h}|^{\frac{1+\gamma}{2}} \in \mathbf{F}_\chi$ in the first term,} \\
& \quad \text{i.e.\,}\langle |\mathsf{h}|^{p'}\mathbf{1}_{|\mathsf{h}|>1}v^p \eta\rangle \leq \langle (|\mathsf{h}|^{\frac{1+\gamma}{2}})^2(v^\frac{p}{2} \sqrt{\eta})^2\rangle \leq \chi \big\langle \big|(\nabla v^\frac{p}{2})^2 \sqrt{\eta} + v^\frac{p}{2}\frac{1}{2}\frac{\nabla \eta}{\sqrt{\eta}}\big|^2 \big\rangle + c_\chi \langle v^p \eta\rangle, \text{\,so:}\bigg) \\
& \leq  2 \frac{\varepsilon_2^{p'}}{p'} \chi \langle |\nabla v^{\frac{p}{2}}|^2\eta \rangle + \frac{1}{2}\frac{\varepsilon_2^{p'}}{p'}\chi \langle v^p \frac{|\nabla \eta|^2}{\eta}\rangle + \frac{\varepsilon_2^{p'}}{p'} (c_\chi+1) \langle v^p  \eta \rangle + \frac{1}{p\varepsilon_2^p}\big\langle \Theta|f|^p \mathbf{1}_{v>0} \eta \big\rangle,
\end{align*}
where $\Theta:=\mathbf{1}_{|\mathsf{h}|>1} + |\mathsf{h}|^p\mathbf{1}_{|\mathsf{h}| \leq 1}$.
Selecting $\varepsilon_2$ sufficiently small and applying the estimates on $I_1$-$I_4$ in \eqref{i12__}, we obtain
\begin{align}
\lambda \|v^{\frac{p}{2}}\mathbf{1}_{B_{r_1}}\|_2^2 + \||\nabla v^{\frac{p}{2}}|\mathbf{1}_{B_{r_1}}\|_2^2  & \leq \frac{C_1}{r_2-r_1}\|(\nabla v^{\frac{p}{2}})\mathbf{1}_{B_{r_2}}\|_2\| v^{\frac{p}{2}}\mathbf{1}_{B_{R}}\|_2 \notag \\
& + C_2 \biggl(1+\frac{1}{(r_2-r_1)^2} \bigg)\|v^{\frac{p}{2}}\mathbf{1}_{B_{R}}\|^2_2 + C_3\|\Theta^{\frac{1}{2}} |f|^\frac{p}{2} \mathbf{1}_{v>0} \mathbf{1}_{B_{R}}\|_2^2, \notag
\end{align}
so, dividing by $\|v^{\frac{p}{2}}\mathbf{1}_{B_{R}}\|^2_2$, we arrive at
\begin{align}
\frac{\lambda \|v^{\frac{p}{2}}\mathbf{1}_{B_{r_1}}\|_2^2 + \|(\nabla v^{\frac{p}{2}})\mathbf{1}_{B_{r_1}}\|_2^2}{\|v^{\frac{p}{2}}\mathbf{1}_{B_{R}}\|^2_2}  & \leq \frac{C_1}{r_2-r_1}\frac{\lambda \|v^{\frac{p}{2}}\mathbf{1}_{B_{r_2}}\|_2^2 + \|(\nabla v^{\frac{p}{2}})\mathbf{1}_{B_{r_2}}\|_2}{\|v^{\frac{p}{2}}\mathbf{1}_{B_{R}}\|_2} \notag \\
&+ C_2 \biggl(1+\frac{1}{(r_2-r_1)^2} \bigg) + C_3S^2, \label{proto_cacc2__}
\end{align}
where
$
S^2:=\frac{\|\Theta^{\frac{1}{2}}|f|^\frac{p}{2} \mathbf{1}_{v>0}\mathbf{1}_{B_{R}}\|_2^2}{\|v^{\frac{p}{2}}\mathbf{1}_{v>0} \mathbf{1}_{B_{R}}\|^2_2}.
$
This is the pre-Caccioppoli inequality that we will iterate.
Put
$$
a^2_n:= \frac{\lambda \|v^{\frac{p}{2}}\mathbf{1}_{B_{R-\frac{R-r}{2^{n-1}}}}\|_2^2 +  \|(\nabla v^{\frac{p}{2}})\mathbf{1}_{B_{R-\frac{R-r}{2^{n-1}}}}\|_2^2}{\|v^{\frac{p}{2}}\mathbf{1}_{v>0} \mathbf{1}_{B_{R}}\|^2_2},
$$
so the inequality \eqref{proto_cacc2__} yields
$$
 a_n^2 \notag  \leq C(R-r)^{-1}2^{n} a_{n+1} + C^2 (R-r)^{-2}2^{2n} + C^2 S^2
$$
with constant $C$ independent of $n$.
We multiply this inequality by $(R-r)^2$ and divide by $C^2 2^{2n}$. Setting $y_n:=\frac{(R-r) a_n}{C 2^{n}},$ we obtain
\begin{equation}
\label{y_int__}
y_n^2 \leq y_{n+1} + 1+(R-r)^2S^2
\end{equation}
for all $n=1,2,\dots$ 
We note that all
$a_n$'s are bounded by a non-generic constant
$
(\lambda \|v^{\frac{p}{2}}\mathbf{1}_{B}\|_2^2 + \|(\nabla v^{\frac{p}{2}})\mathbf{1}_{B_{R}}\|_2) / \|v^{\frac{p}{2}}\mathbf{1}_{B_{R}}\|_2<\infty,
$
so $\sup_n y_n<\infty$.
Therefore, we can iterate
 \eqref{y_int__} and hence estimate all $y_n$, $n=1,2,\dots$, via nested square roots $1+(R-r)^2S^2+\sqrt{1+(R-r)^2S^2 + \sqrt{\dots}}$, 
obtaining
$$y_n^2 \leq 3+2(R-r)^2 S^2, \quad n=1,2,\dots$$
Taking $n=1$, 
we arrive at 
$
a_1 \leq K_1 (R-r)^{-2}+K_2S^2
$
for appropriate constants $K_1$ and $K_2$, i.e.
$$
\frac{\lambda \|v^{\frac{p}{2}}\mathbf{1}_{B_{r}}\|_2^2 + \|(\nabla v^{\frac{p}{2}})\mathbf{1}_{B_{r}}\|_2^2}{\|v^{\frac{p}{2}}\mathbf{1}_{B_{R}}\|^2_2}  \leq K_1 (R-r)^{-2} + K_2 \frac{\|\Theta^{\frac{1}{2}}|f|^\frac{p}{2}\mathbf{1}_{v>0}  \mathbf{1}_{B_{R}}\|_2^2}{\|v^{\frac{p}{2}}\mathbf{1}_{B_{R}}\|^2_2},
$$
as claimed.

\begin{remark}
\label{first_cond_rem2}
If $b$ and $\mathsf{h}$ satisfy condition \eqref{first_cond}, then, again, we can work with cutoff functions $\eta \in C_c^\infty$,
$\eta =1$ in $B_{r_1}$, $\eta = 0$ in $\mathbb R^d \setminus B_{r_2}$, i.e.\,$|\nabla \eta| \leq c_1(r_2-r_1)^{-1}$, $|\Delta \eta| \leq c_2(r_2-r_1)^{-2}$, and we estimate the second term in the RHS of \eqref{i12__} right away using quadratic inequality:
$$
2|\langle b \cdot \nabla v^{\frac{p}{2}},v^{\frac{p}{2}}\eta\rangle| \leq \alpha \langle |\nabla v|^{\frac{p}{2}}\eta\rangle + \frac{1}{4\alpha}\langle |b|^2,v^p\eta\rangle, \quad \alpha=\frac{2}{\sqrt{\delta}}.
$$
Regarding the terms containing $\mathsf{h}$, we simply take $\gamma=1$, which transforms condition $|\mathsf{h}|^{\frac{1+\gamma}{2}} \in \mathbf{F}_\chi$, $\chi<\infty$ from \eqref{second_cond} into condition $\mathsf{h} \in \mathbf{F}_\chi$ in \eqref{first_cond}, and argue as in the estimate on $I_4$ above.
\end{remark}

This ends the proof of Proposition \ref{c_prop2}.
\end{proof}

Recall that we have fixed $1<\theta<\frac{d}{d-2}$ in the statement of the theorem.

\begin{proposition} 
\label{nonhom_prop}
There exists generic constants $K$  and $\beta \in ]0,1[$ such that, for all $\lambda \geq 1$, the positive part $u_+$ of solution $u$ of non-homogeneous equation \eqref{eq7} satisfies
\begin{equation}
\label{dg_ineq}
\sup_{B_{\frac{1}{2}}} u_+ \leq K \biggl( \langle u_+^{p\theta}\mathbf{1}_{B_1}\rangle^{\frac{1}{p\theta}} + \lambda^{-\frac{\beta}{p}}\big\langle (\mathbf{1}_{|\mathsf{h}|>1} + |\mathsf{h}|^p\mathbf{1}_{|\mathsf{h}| \leq 1})^{\theta'}|f|^{p\theta'}\mathbf{1}_{B_1}\big\rangle^{\frac{1}{p\theta'}} \biggr).  
\end{equation}
\end{proposition}

\begin{remark}
In the proof of Proposition \ref{nonhom_prop} we iterate simultaneously over (a) balls of decreasing radius (b) super-level sets of solution $u$. To get an overview of the proof, one can first formally take $r=R=\infty$ to see that the iterations over super-level sets indeed work as intended.
\end{remark}

\begin{proof}
Proposition \ref{c_prop2} yields
\begin{align*}
\lambda \|v^{\frac{p}{2}}\mathbf{1}_{B_{r}}\|_2^2 + \|v^{\frac{p}{2}}\|^2_{W^{1,2}(B_{r})} & \leq \tilde{K}_1(R-r)^{-2}\|v\|^{p}_{L^p(B_{R})} \\
&+ K_2\|\Theta^{\frac{1}{p}}f\mathbf{1}_{u>k}\|_{L^p(B_R)}^p, \qquad v:=(u-k)_+,\;\;k \geq 0,
\end{align*}
where $\Theta:=\mathbf{1}_{|\mathsf{h}|>1} + |\mathsf{h}|^p\mathbf{1}_{|\mathsf{h}| \leq 1}$ and $\tilde{K}_1$, $K_2$ are generic constants.
By the Sobolev embedding theorem,
$$
\lambda \|v\|_{L^p(B_r)}^p + \|v\|^{p}_{L^{\frac{pd}{d-2}}(B_{r})}\leq C_1(R-r)^{-2}\|v\|^{p}_{L^p(B_{R})}+C_2\|\Theta^{\frac{1}{p}}f\mathbf{1}_{u>k}\|_{L^p(B_R)}^p.
$$
Next, we estimate the left-hand side from below using interpolation inequality:
$$
\lambda^{\beta} \|v\|_{L^q(B_r)}^p \leq \beta\lambda \|v\|_{L^p(B_r)}^p + (1-\beta)\|v\|_{L^{\frac{pd}{d-2}}(B_r)}^p, \quad 0<\beta<1, \quad \frac{1}{q}=\beta\frac{1}{p}+(1-\beta)\frac{d-2}{pd}.
$$

\begin{remark}1.~To prove the previous inequality, we put for brevity $r=\frac{pd}{d-2}$, $\mu=\lambda^{\frac{q\beta}{p}}$ and write $\mu |v|^q=(\mu^{\frac{1}{\beta q}} v)^{p\frac{\beta q}{p}} |v|^{r\frac{(1-\beta)q}{r}}$, so, denoting here the integration over $B_r$ by $\langle \cdot \rangle$, we have $$\langle \mu |v|^q\rangle \leq \langle (\mu^{\frac{1}{\beta q}} v)^p \rangle^{\frac{q\beta}{p}} \langle |v|^r\rangle^{\frac{(1-\beta)q}{r}}.$$ Hence $$\mu^{\frac{p}{q}}\|v\|_q^p \leq \langle (\mu^{\frac{1}{\beta q}} v)^p\rangle^\beta \langle |v|^r\rangle^{(1-\beta)\frac{p}{r}},$$ and it remains to apply in the right-hand side Young's inequality $ab \leq \beta a^{\frac{1}{\beta}} + (1-\beta)b^{\frac{1}{1-\beta}}$.

\medskip

2.~We could take $\beta=0$, in which case $$q=\frac{pd}{d-2}, \quad \theta_0=\frac{d}{d-2},$$ but then we lose the dependence on $\lambda$ in the second term in the right-hand side of \eqref{dg_ineq}. Strictly speaking, we do not need to keep track of the dependence on $\lambda$ in this work, but this is needed for some other applications of Theorem \ref{thm2}, e.g.\,to construct strongly continuous Feller semigroup in \cite{KiS_feller_4}. To get an overview of this proof it is worthwhile to first ignore $\lambda$ and to take $\beta=0$ and $\theta_0=\frac{d}{d-2}$.
\end{remark}

Put $\theta_0:=\frac{q}{p}$, so $1<\theta_0<\frac{d}{d-2}$. We fix $\beta \in ]0,1[$ sufficiently small so that $\theta<\theta_0$. 

\medskip

Hence, taking into account that $q=p\theta_0$,
$$
\lambda^{\beta} \|v\|_{L^{p\theta_0}(B_r)}^p \leq \tilde{C}_1(R-r)^{-2}\|v\|^{p}_{L^p(B_{R})}+\tilde{C}_2\|\Theta^{\frac{1}{p}}f\mathbf{1}_{u>k}\|_{L^p(B_R)}^p.
$$
Applying H\"{o}lder's inequality in the RHS, we obtain
\begin{equation}
\label{ineq_h_n}
\lambda^{\beta} \|v\|_{L^{p\theta_0}(B_r)}^p \leq \tilde{C}_1(R-r)^{-2}|B_{R}|^{\frac{\theta-1}{\theta}}\|v\|^{p}_{L^{p\theta}(B_{R})}+\tilde{C}_2\|\Theta^{\frac{1}{p}}f\mathbf{1}_{u>k}\|_{L^p(B_R)}^p.
\end{equation}
Set $$R_m:=\frac{1}{2}+\frac{1}{2^{m+1}}, \quad m \geq 0,$$
so $B^m \equiv B_{R_m}$ is a decreasing sequence of balls converging to the ball of radius $\frac{1}{2}$. 
By \eqref{ineq_h_n},
\begin{align}
\lambda^{ \beta} \|v\|_{L^{p\theta_0}(B^{m+1})}^p & \leq \hat{C}_1 2^{2m}
\|v\|_{L^{p\theta}(B^{m})}^{p}+\tilde{C}_2  \|\Theta^{\frac{1}{p}}f\mathbf{1}_{u>k}\|_{L^p(B^m)}^p \notag \\
& \leq \hat{C}_1 2^{2m}
\|v\|_{L^{p\theta}(B^{m})}^{p} + \tilde{C}_2 H|B^{m} \cap \{v>0\}|^{\frac{1}{\theta}},
\label{theta_d_d-2_n}
\end{align}
where
$$
H:=\langle \Theta^{\theta'}|f|^{p\theta'}\mathbf{1}_{B^0}\rangle^{\frac{1}{\theta'}} \quad (B^0=B_1, \text{\,i.e.\,ball of radius } 1)
$$

On the other hand, by H\"{o}lder's inequality,
$$
 \|v\|_{L^{p\theta}(B^{m+1})}^{p\theta} \leq  \|v\|_{L^{p\theta_0}(B^{m+1})}^{p\theta} \biggl(|B^{m} \cap \{v>0\}| \biggr)^{1-\frac{\theta}{\theta_0}}.
$$
Applying \eqref{theta_d_d-2_n} in the first multiple in the RHS, we obtain
$$
 \|v\|_{L^{p\theta}(B^{m+1})}^{p\theta} \leq \tilde{C}\lambda^{-\beta\theta}\biggl( 2^{2\theta m}\|v\|_{L^{p\theta}(B^{m})}^{p\theta}+ H^{\theta}|B^{m} \cap \{v>0\}| \biggr) \biggl(|B^{m} \cap \{v>0\}| \biggr)^{1-\frac{\theta}{\theta_0}}.
$$
Now, put $v_m:=(u-k_m)_+$ where $k_m:=\xi(1-2^{-m}) \uparrow \xi,$ where constant $\xi>0$ will be chosen later.
Then, using $2^{2\theta m} \leq 2^{p\theta m}$ and dividing by $\xi^{p\theta}$,
\begin{align*}
&\frac{1}{\xi^{p\theta}}  \|v_{m+1}\|_{L^{p\theta}(B^{m+1})}^{p\theta} \\
&\leq \tilde{C}\lambda^{-\beta\theta}\biggl( \frac{2^{p\theta m}}{\xi^{p\theta}}\|v_{m+1}\|_{L^{p\theta}(B^{m})}^{p\theta}  + \frac{1}{\xi^{p\theta}} H^{\theta}|B^{m} \cap \{u>k_{m+1}\}|\biggr) \biggl(|B^{m} \cap \{u>k_{m+1}\}| \biggr)^{1-\frac{\theta}{\theta_0}}.
\end{align*}
Using that $\lambda \geq 1$, we further obtain
\begin{align*}
&\frac{1}{\xi^{p\theta}}  \|v_{m+1}\|_{L^{p\theta}(B^{m+1})}^{p\theta} \\
&\leq \tilde{C}\biggl( \frac{2^{p\theta m}}{\xi^{p\theta}}\|v_{m+1}\|_{L^{p\theta}(B^{m})}^{p\theta}  + \frac{1}{\xi^{p\theta}} \lambda^{-\beta\theta} H^{\theta}|B^{m} \cap \{u>k_{m+1}\}|\biggr) \biggl(|B^{m} \cap \{u>k_{m+1}\}| \biggr)^{1-\frac{\theta}{\theta_0}}.
\end{align*}
From now on, we require that constant $\xi$ satisfies $\xi^p \geq \lambda^{-\beta}H$, so
\begin{align}
\label{dg_ineq7}
&\frac{1}{\xi^{p\theta}}  \|v_{m+1}\|_{L^{p\theta}(B^{m+1})}^{p\theta} \\
&\leq \tilde{C}\biggl( \frac{2^{p\theta m}}{\xi^{p\theta}}\|v_{m+1}\|_{L^{p\theta}(B^{m})}^{p\theta}  + |B^{m} \cap \{u>k_{m+1}\}|\biggr) \biggl(|B^{m} \cap \{u>k_{m+1}\}| \biggr)^{1-\frac{\theta}{\theta_0}}. \notag
\end{align}
Now, 
\begin{align*}
|B^{m} \cap \{u>k_{m+1}\}| & = \big|B^{m} \cap \big\{(\frac{u-k_m}{k_{m+1}-k_m})^{2\theta}>1\big\}\big| \\
&\leq (k_{m+1}-k_m)^{-p\theta} \langle v^{p\theta}_m \mathbf{1}_{B^{m}} \rangle  = \xi^{-p\theta}2^{p\theta(m+1)} \|v_m\|_{L^{p\theta}(B^{m})}^{p\theta},
\end{align*}
so using in \eqref{dg_ineq7} $\|v_{m+1}\|_{L^{p\theta}(B^{m})} \leq \|v_{m}\|_{L^{p\theta}(B^{m})}$ and applying the previous inequality, we obtain
$$
\frac{1}{\xi^{p\theta}} \|v_{m+1}\|_{L^{p\theta}(B^{m+1})}^{p\theta} \leq \tilde{C}  2^{p\theta m(2-\frac{\theta}{\theta_0})} \biggl(\frac{1}{\xi^{p\theta}}\|v_m\|_{L^{p\theta}(B^{m})}^{p\theta}\biggr)^{2-\frac{\theta}{\theta_0}}.
$$
Denote
$
z_m:=\frac{1}{\xi^{p\theta}}\|v_m\|_{L^{p\theta}(B^{m})}^{p\theta}.
$
Then
$$
z_{m+1} \leq \tilde{C}\gamma^m z_m^{1+\alpha}, \quad m=0,1,2,\dots, \quad \alpha:=1-\frac{\theta}{\theta_0},\;\; \gamma:=2^{p\theta (2-\frac{\theta}{\theta_0})}
$$
and $z_0 = \frac{1}{\xi^{p\theta}}\langle u_+^{p\theta}\mathbf{1}_{B^0} \rangle \leq \tilde{C}^{-\frac{1}{\alpha}}\gamma^{-\frac{1}{\alpha^2}}$ (recall: $B^0:=B_{R_0} \equiv B_1$)
provided that we fix $c$ by $$\xi^{p\theta}:=\tilde{C}^{\frac{1}{\alpha}}\gamma^{\frac{1}{\alpha^2}}\langle u_+^{p\theta}\mathbf{1}_{B^0}\rangle + \lambda^{-\beta\theta}H^\theta.$$
Hence, by Lemma \ref{dg_lemma}, $z_m \rightarrow 0$ as $m \rightarrow \infty$. It follows that
$
\sup_{B_{1/2}}u_+ \leq \xi,
$
and the claimed inequality follows. 
\end{proof}

To end the proof of Theorem \ref{thm2}, we need to estimate $\langle u_+^{p\theta}\mathbf{1}_{B_1}\rangle^{1/p\theta}$ in the RHS of \eqref{dg_ineq} in terms of $\mathsf{h}$ and $f$. We will do it by estimating $\langle u_+^{p\theta}\rho\rangle^{1/p\theta}$, where, recall,
$$
\rho(x)=(1+k|x|^{2})^{-\frac{d}{2}-1}, 
$$
and then applying inequality $\rho \geq c\mathbf{1}_{B_1}$ for appropriate constant $c=c_d$.

\begin{proposition}
\label{rho_prop}
There exist generic constants $C$, $k$ and $\lambda_0>0$ such that for all $\lambda \geq \lambda_0$,
\begin{align}
(\lambda-\lambda_0) \langle u^p \rho \rangle + \langle |\nabla u^{\frac{p}{2}}|^2 \rho \rangle \leq  C  \big\langle \big(\mathbf{1}_{|\mathsf{h}|>1} + |\mathsf{h}|^p\mathbf{1}_{|\mathsf{h}| \leq 1}\big)|f|^p\rho \big\rangle. \label{c_ineq_rho} 
\end{align}
\end{proposition}

\begin{proof}
Let $b$ satisfy condition \eqref{second_cond}.
We may assume without loss of generality that $p>\frac{2}{2-\delta_+}$ is rational with odd denominator.
We multiply equation \eqref{eq7} by $u^{p-1}\rho$ and integrate to obtain
\begin{align*}
\lambda \langle u^p\rho\rangle + \frac{4(p-1)}{p^2}\langle \nabla u^{\frac{p}{2}},(\nabla u^{\frac{p}{2}})\rho\rangle +  \frac{2}{p}\langle \nabla u^{\frac{p}{2}},u^{\frac{p}{2}}\nabla \rho \rangle  + \frac{2}{p}\langle b \cdot \nabla u^{\frac{p}{2}},u^{\frac{p}{2}}\rho\rangle = \langle |\mathsf{h}|f,u^{p-1}\rho\rangle.
\end{align*}
Now we argue as in the proof of Proposition \ref{c_prop2}, but instead of the iterations we use a straightforward estimate $|\nabla \rho| \leq (\frac{d}{2}+1)\sqrt{k} \rho$ in order to get rid of $\nabla \rho$ in the previous identity. We arrive at 
\begin{align*}
p\lambda \langle u^p\rho \rangle & + \biggl(\frac{4(p-1)}{p}-\frac{4}{p}\varepsilon \biggr) \langle |\nabla u^{\frac{p}{2}}|^2 \rho\rangle \\ 
& \leq \frac{p}{4\varepsilon} (\frac{d}{2}+1)^2 k \langle v^p\rho\rangle + (\frac{d}{2}+1)\sqrt{k}\langle |b|u^p\rho\rangle + \langle {\rm div\,}b_+,u^p\rho\rangle \qquad (\varepsilon,\varepsilon'>0)\\
& + p\biggl(  2\varepsilon' \chi \langle |\nabla u^{\frac{p}{2}}|^2\rho \rangle + 2\varepsilon' \chi (\frac{d}{2}+1)^2 k \langle u^p \rho\rangle \\
& + \varepsilon' (c_\chi+1) \langle u^p  \rho \rangle  + \frac{1}{4\varepsilon'}\langle \big(\mathbf{1}_{|\mathsf{h}|>1} + |\mathsf{h}|^{p}\mathbf{1}_{|\mathsf{h}| \leq 1}\big)|f|^p \rho \rangle \biggr).
\end{align*}
The terms $\langle |b|u^p\rho\rangle$, $\langle ({\rm div\,}b)_+,u^p\rho\rangle$ are estimated by applying quadratic inequality and using condition \eqref{second_cond}. Selecting $\varepsilon$, $\varepsilon'$, $k$ sufficiently small, we arrive at the sought inequality. 

\medskip

If $b$ satisfies \eqref{first_cond}, then the proof is similar but easier (i.e.\,we do not need to integrate by parts, only to apply quadratic inequality to $\langle b \cdot \nabla u^{\frac{p}{2}},u^{\frac{p}{2}}\rho\rangle$ and use form-boundedness of $b$).
\end{proof}

\subsubsection*{Proof of Theorem \ref{thm2}, completed} By Proposition \ref{nonhom_prop}, for all $\lambda \geq 1$,
$$
\sup_{y \in B_{\frac{1}{2}}(x)} |u(y)| \leq K \biggl( \langle |u|^{p\theta}\rho_x\rangle^{\frac{1}{p\theta}} + \lambda^{-\frac{\beta}{p}}\big\langle \big(\mathbf{1}_{|\mathsf{h}|>1} + |\mathsf{h}|^{p\theta'}\mathbf{1}_{|\mathsf{h}| \leq 1}\big)|f|^{p\theta'}\rho_x\big\rangle^{\frac{1}{p\theta'}} \biggr),
$$
where $\rho_x(y):=\rho(y-x)$, and constant $C$ is generic, so
$$
\|u\|_\infty \leq K \sup_{x \in \frac{1}{2}\mathbb Z^d}\biggl( \langle |u|^{p\theta}\rho_x\rangle^{\frac{1}{p\theta}} + \lambda^{-\frac{\beta}{p}}\big\langle \big(\mathbf{1}_{|\mathsf{h}|>1} + |\mathsf{h}|^{p\theta'}\mathbf{1}_{|\mathsf{h}| \leq 1}\big)|f|^{p\theta'}\rho_x\big\rangle^{\frac{1}{p\theta'}} \biggr).
$$
Applying Proposition \ref{rho_prop} to the first term in the RHS (with $p\theta$ instead of $p$), we obtain for all $\lambda \geq \lambda_0 \vee 1$
\begin{align*}
\|u\|_\infty \leq C  \sup_{x \in \frac{1}{2}\mathbb Z^d}\biggl( & (\lambda-\lambda_0)^{-\frac{1}{p\theta}}\big\langle \big(\mathbf{1}_{|\mathsf{h}|>1} + |\mathsf{h}|^{p\theta}\mathbf{1}_{|\mathsf{h}| \leq 1}\big) |f|^{p\theta}\rho_x\big\rangle^{\frac{1}{p\theta}} \\
&+ \lambda^{-\frac{\beta}{p}}\big\langle \big(\mathbf{1}_{|\mathsf{h}|>1} + |\mathsf{h}|^{p\theta'}\mathbf{1}_{|\mathsf{h}| \leq 1}\big)|f|^{p\theta'}\rho_x\big\rangle^{\frac{1}{p\theta'}} \biggr).
\end{align*}
This ends the proof of Theorem \ref{thm2}. \hfill \qed

\bigskip

\section{Proof of Theorem \ref{markov_thm}}

(\textit{i}) By the assumption of the theorem, the Borel measurable vector field $b:\mathbb R^d \rightarrow \mathbb R^d$ satisfies either 
\begin{equation}
\tag{$\mathbb{A}_1$}
b \in \mathbf{F}_\delta \quad \text{with } \delta<4
\end{equation}
or
\begin{equation}
\tag{$\mathbb{A}_2$}
\left\{
\begin{array}{l}
b \in \mathbf{MF}_\delta \text{ for some } \delta<\infty, \\[2mm]
({\rm div\,}b)_- \in L^1+L^\infty, \\[2mm]
({\rm div\,}b)_+^{\frac{1}{2}} \in \mathbf{F}_{\delta_+} \text{ with } \delta_+<4, \\[2mm]
|b|^{\frac{1+\alpha}{2}} \in \mathbf{F}_{\chi} \quad \text{ for some $\alpha>0$ fixed arbitrarily small, and some $\chi<\infty$}.
\end{array}
\right.
\end{equation}
We define a regularization of $b$ as in Section \ref{approx_sect}:
$$b_\varepsilon:=E_\varepsilon b, \quad \varepsilon \downarrow 0,$$
where $E_\varepsilon$ is the Friedrichs mollifier.
Then, recall, $\{b_\varepsilon\}$ are bounded and smooth, preserve all form-bounds in \eqref{A_0} or in \eqref{A_1}, and converge to $b$ in $[L^2_{\loc}]^d$ or in $[L^1_{\loc}]^d$, respectively.

\medskip

Step 1.~By the classical theory, for every $\varepsilon>0$, there exist unique strong solution $Y_\varepsilon$ to SDE
\begin{equation*}
Y_\varepsilon(t)=y-\int_0^t b_{\varepsilon}(Y_\varepsilon(s))ds + \sqrt{2}B(t), \quad y \in \mathbb R^d,
\end{equation*}
where $\{B(t)\}_{t \geq 0}$ is a Brownian motion in $\mathbb R^d$ on a fixed complete probability space $(\Omega,\mathcal F,\mathcal F_t,\mathbf P)$.

\smallskip

Fix $T>0$. 

\begin{lemma}
\label{g_cond_lem}
Let vector field $\mathsf{g} \in [C_b(\mathbb R^d)]^d$ be such that:

\smallskip

1. If $b$ satisfies condition \eqref{A_1}, then
\begin{equation}
\label{g_cond}
\langle |\mathsf{g}|^{1+\alpha}\varphi,\varphi \rangle \leq \chi \|\nabla \varphi\|_2^2 + c_\chi\|\varphi\|_2^2, \quad \varphi \in W^{1,2},
\end{equation}
where constants $\chi$, $c_\chi$ are from condition \eqref{A_1}. 

\smallskip

2.  If $b$ satisfies condition \eqref{A_0}, then \eqref{g_cond} holds with $\alpha=1$, $\chi=\delta$ and $c_{\chi}=c_\delta$.

\smallskip

Fix $\gamma>0$ by  $1+\alpha=(1+\gamma)^2$.
Then
\begin{equation}
\label{tight_ineq0}
\mathbf E\int_{t_0}^{t_1} |\mathsf{g}(Y_\varepsilon(s))|ds \leq C_2(t_1-t_0)^{\frac{\gamma}{1+\gamma}},
\end{equation}
where constant $C_2$ does not depend on $\varepsilon$, $y$ or $t_0$, $t_1$ (but it depends, by Theorem \ref{thm2}, on constants $\chi$, $c_\chi$). 
\end{lemma}

(We will be applying \eqref{tight_ineq0} with $\mathsf{g}=b_\varepsilon$.)

\begin{proof}[Proof of Lemma \ref{g_cond_lem}]
First, let $\mathsf{g} \in [C_c(\mathbb R^d)]^d$.
 By H\"{o}lder's inequality,
\begin{align}
\mathbf E\int_{t_0}^{t_1} |\mathsf{g}(Y_\varepsilon(s))|ds & =  \mathbf E\int_{t_0}^{t_1} e^{\lambda t} e^{-\lambda t}|\mathsf{g}(Y_\varepsilon(s))|ds \notag \\
& \leq e^{\lambda T}(t_1-t_0)^{\frac{\gamma}{1+\gamma}}\bigg(\mathbf{E}\int_{0}^{\infty} e^{-(1+\gamma)\lambda t}|\mathsf{g}(Y_\varepsilon(s))|^{1+\gamma}ds\bigg)^{\frac{1}{1+\gamma}} \notag \\
& = e^{\lambda T}(t_1-t_0)^{\frac{\gamma}{1+\gamma}} u_\varepsilon(x)^{\frac{1}{1+\gamma}} \label{e_tight}
\end{align}
where $u_\varepsilon$
is the classical solution to non-homogeneous elliptic equation
$$
\bigl[(1+\gamma)\lambda -\Delta + b_\varepsilon \cdot \nabla\bigr]u_\varepsilon=|\mathsf{g}|^{1+\gamma}.
$$
Note that, in view of the results of Section \ref{approx_sect}, condition \eqref{A_1}
implies the second condition \eqref{A_1prime} on $b$ of Theorem \ref{thm1} for $b_\varepsilon$. (If $b$ satisfies condition \eqref{A_0}, then $b_\varepsilon$ satisfy the same condition in Theorem \ref{thm1}.)
Further, we take in Theorem \ref{thm2}   $\mathsf{h}:=\mathsf{g}|\mathsf{g}|^{\gamma}$ and $f=1$ in a neighbourhood of the support of $\mathsf{g}$. In view of $1+\alpha=(1+\gamma)^2$ and \eqref{g_cond}, $\mathsf{h}$ satisfies condition $|\mathsf{h}|^{\frac{1+\gamma}{2}} \in \mathbf{F}_\chi$ of Theorem \ref{thm2}.
Thus, Theorem \ref{thm2} applies and yields
\begin{align}
\|u_\varepsilon\|_\infty \leq  
C \sup_{x \in \frac{1}{2}\mathbb Z^d}\biggl( \langle \big(\mathbf{1}_{|\mathsf{g}|>1} + |\mathsf{g}|^{(1+\gamma)p\theta}\mathbf{1}_{|\mathsf{g}| \leq 1}\big)\rho_x\rangle^{\frac{1}{p\theta}} + \langle \big(\mathbf{1}_{|\mathsf{g}|>1} + |\mathsf{g}|^{(1+\gamma)p\theta'}\mathbf{1}_{|\mathsf{g}| \leq 1}\big)\rho_x\rangle^{\frac{1}{p\theta'}} \biggr), \label{key_est0}
\end{align}
where the right-hand side is finite (by our choice of $\rho$) and clearly does not depend on $\varepsilon$.
It is seen now that \eqref{tight_ineq0} follows from \eqref{e_tight}. Using Fatou's lemma, we can replace the requirement that $\mathsf{g}$ has compact support by $\mathsf{g} \in [C_b(\mathbb R^d)]^d$.
\end{proof}

Inequality \eqref{tight_ineq0} yields, upon taking $\mathsf{g}:=b_\varepsilon$,
\begin{equation}
\label{tight_ineq}
\mathbf E\int_{t_0}^{t_1} |b_\varepsilon(Y_\varepsilon(s))|ds \leq C_2(t_1-t_0)^{\frac{\gamma}{1+\gamma}}
\end{equation}
(note that $|b_\varepsilon|^{1+\gamma}$ have independent of $\varepsilon$ finite form-bound $\chi$ and constant $c_\chi$, see Lemma \ref{cl4}). This gives us the next lemma.
We will write $Y^y_{\varepsilon}$ to emphasize the dependence of solution $Y_\varepsilon$ on $y$.

\begin{lemma}
\label{tight_lem}
{\rm (\textit{i})} For every $\beta>0$,
\begin{equation}
\label{tight_est}
\sup_{\varepsilon>0} \sup_{y \in \mathbb R^d} \mathbf P \bigg[ \sup_{t \in [0,1], \sigma' \in [0,\sigma]} |Y^{y}_\varepsilon(t+\sigma') - Y^y_\varepsilon(t)|>\beta \bigg] \leq \hat{C}H(\sigma), 
\end{equation}
where constant $\hat{C}$ and function $H$ are independent of $\varepsilon$, and $H(\sigma) \downarrow 0$ as $\sigma \downarrow 0$.

\smallskip

{\rm (\textit{ii})} For every $y \in \mathbb R^d$, the family of probability measures $$\mathbb P^\varepsilon_x:=(\mathbf P \circ Y^y_\varepsilon)^{-1}, \quad \varepsilon>0,$$ is tight on the canonical space of continuous trajectories  on $[0,T]$.
\end{lemma}
\begin{proof}[Proof of Lemma \ref{tight_lem}]
The argument is standard. For reader's convenience, we include it below (we repeat more or less verbatim a part of \cite{KiS_sharp}). Put for brevity $T=1$.
We have, for a stopping time $0 \leq \tau \leq 1$,
\begin{equation}
\label{X_tau}
Y^y_\varepsilon(\tau+\sigma)-Y^y_\varepsilon(\tau)=\int_\tau^{\tau+\sigma} b_n(s,Y^y_\varepsilon(s))ds+\sqrt{2}(B(\tau+\sigma)-B(\tau)), \quad 0<\sigma< 1.
\end{equation}
Next, note that \eqref{tight_ineq} yields
\begin{equation}
\label{key_est2}
\mathbf{E}\int_\tau^{\tau+\sigma} |b_n(s,Y^y_\varepsilon(s))|ds \leq C_0\sigma^{\frac{\gamma}{\gamma+1}},
\end{equation}
see Remark 1.2 in \cite{ZZ} (to show that \eqref{tight_ineq} $\Rightarrow$ \eqref{key_est2},  the authors of \cite{ZZ} use a decreasing sequence of stopping times $\tau_m$ converging to $\tau$ and taking values in $S=\{k2^{-m} \mid k \in \mathbb \{0,1,2,\dots\}\}$, and note that the proof of estimate \eqref{key_est2}  with $\tau_m$ in place of $\tau$ can be reduced to applying \eqref{tight_ineq} on intervals $[t_0,t_1]:=[c,c+\sigma]$, $c \in S$.) Thus, applying \eqref{key_est2} in \eqref{X_tau}, one obtains
\begin{align*}
\mathbf E\sup_{\sigma' \in [0,\sigma]} |Y^y_\varepsilon(\tau+\sigma')-Y^y_\varepsilon(\tau)|  \leq C_0 \sigma^{\frac{\gamma}{\gamma+1}} + C_1\sigma^{\frac{1}{2}}=:H(\sigma).
\end{align*}
Now, applying \cite[Lemma 2.7]{ZZ2}, we obtain: there exists constant $\hat{C}$ independent of $\varepsilon$ such that
\begin{equation}
\label{tight2}
\sup_\varepsilon \sup_{y \in \mathbb R^d} \mathbf E \bigg[ \sup_{t \in [0,1], \sigma' \in [0,\sigma]} |Y^y_\varepsilon(t+\sigma') - Y^y_\varepsilon(t)|^{\frac{1}{2}} \bigg] \leq \hat{C} H(\sigma).
\end{equation}
Applying Chebyshev's inequality in \eqref{tight2}, since $H(\sigma) \downarrow 0$ as $\sigma \downarrow 0$, we obtain the first assertion of the lemma. The second assertion follows from the first one, see \cite[Theorem 1.3.2]{SV}.  
\end{proof}

Fix $y \in \mathbb R^d$.
Let $\mathbb P_y$ be a weak subsequential limit point of $\{\mathbb P_y^{\varepsilon}\}$,
\begin{equation}
\label{w}
\mathbb P_y^{\varepsilon_k} \rightarrow \mathbb P_y \text{ weakly} \quad \text{ for some } \varepsilon_k \downarrow 0.
\end{equation}
Let us rewrite \eqref{tight_ineq0} as 
\begin{equation*}
\mathbb E^\varepsilon_y\int_{t_0}^{t_1} |\mathsf{g}(\omega_s)|ds \leq C_2(t_1-t_0)^{\frac{\gamma}{1+\gamma}}.
\end{equation*}
Taking $\mathsf{g}:=b_{\varepsilon_m}$ and then applying \eqref{w},
we obtain
$
\mathbb E_y\int_{t_0}^{t_1} |b_{\varepsilon_m}(\omega_s)|ds \leq C_2 (t_1-t_0)^\frac{\gamma}{1+\gamma},
$
and hence, using e.g.\,Fatou's lemma,
$
\mathbb E_y\int_{t_0}^{t_1} |b(\omega_s)|ds \leq C_2 (t_1-t_0)^\frac{\gamma}{1+\gamma}<\infty.
$

\medskip

Step 2.~Let us show that, for any fixed $y \in \mathbb R^d$, any subsequential limit point $\mathbb P_y$ of $\{\mathbb P_y^{\varepsilon}\}$ (say, \eqref{w} holds) is a solution to the martingale problem for SDE \eqref{sde1}. Set
$$
M^{\varphi,\varepsilon}_t:=\varphi(\omega_t)-\varphi(\omega_0) + \int_0^t (-\Delta\varphi + b_\varepsilon \cdot \nabla \varphi)(\omega_s) ds, \quad \varphi \in C_c^2.
$$ 
It suffices to show that
$\mathbb E_y[M^\varphi_{t_1} G]=\mathbb E_y[M_{t_0}^\varphi G]$ for every $\mathcal B_{t_0}$-measurable $G \in C_b\big(C([0,T],\mathbb R^d)\big)$.
We will do this by passing to the limit in $k$ in $$\mathbb E^{\varepsilon_k}_y[M^{\varphi,\varepsilon_k}_{t_1} G]=\mathbb E^{\varepsilon_k}_y[M_{t_0}^{\varphi,{\varepsilon_k}}G].$$ 
That is, we need to prove
\begin{equation}
\label{c}
\lim_{k}\mathbb E_y^{\varepsilon_k}\int_0^t (b_{\varepsilon_k} \cdot \nabla \varphi)(\omega_s) G(\omega)ds= \mathbb E_y\int_0^t (b \cdot \nabla \varphi)(\omega_s) G(\omega)ds,
\end{equation}
\textit{Proof of \eqref{c}.} First, let us note that repeating the proof of \eqref{tight_ineq0}, but this time selecting $\mathsf{h}:=\mathsf{g}|\mathsf{g}|^\gamma$, $\mathsf{g}:=b_{\varepsilon_{m_1}}-b_{\varepsilon_{m_2}}$, $f:=|\nabla \varphi|$, we have
\begin{align*}
& \mathbb E^\varepsilon_y\int_{t_0}^{t_1} \big|b_{\varepsilon_{m_1}}(\omega_s)-b_{\varepsilon_{m_2}}(\omega_s)\big||\nabla \varphi(\omega_s)|ds \\
&\leq  C_3 \sup_{x \in \frac{1}{2}\mathbb Z^d}\biggl( \langle \big(\mathbf{1}_{|\mathsf{g}|>1} + |\mathsf{g}|^{(1+\gamma)p\theta}\mathbf{1}_{|\mathsf{g}| \leq 1}\big)|\nabla \varphi|^{p\theta}\rho_x\rangle^{\frac{1}{p\theta}} + \langle \big(\mathbf{1}_{|\mathsf{g}|>1} + |\mathsf{g}|^{(1+\gamma)p\theta'}\mathbf{1}_{|\mathsf{g}| \leq 1}\big)|\nabla \varphi|^{p\theta'}\rho_x\rangle^{\frac{1}{p\theta'}} \biggr)^{\frac{1}{1+\gamma}},
\end{align*}
Since $\varphi$ has compact support, the RHS converges to $0$ as $m_1$, $m_2 \rightarrow \infty$.
Now, it follows from the weak convergence \eqref{w} and Fatou's lemma that  
\begin{align*}
 \mathbb E_y\int_{t_0}^{t_1} \big|b(\omega_s)-b_{\varepsilon_m}(\omega_s)\big|&|\nabla \varphi(\omega_s)|ds \\
& \leq   
C_3 \sup_{x \in \frac{1}{2}\mathbb Z^d}\biggl( \langle \big(\mathbf{1}_{|b-b_{\varepsilon_m}|>1} + |b-b_{\varepsilon_m}|^{(1+\gamma)p\theta}\mathbf{1}_{|b-b_{\varepsilon_m}| \leq 1}\big)|\nabla \varphi|^{p\theta}\rho_x\rangle^{\frac{1}{p\theta}} \\
& + \langle \big(\mathbf{1}_{|b-b_{\varepsilon_m}|>1} + |b-b_{\varepsilon_m}|^{(1+\gamma)p\theta'}\mathbf{1}_{|b-b_{\varepsilon_m}| \leq 1}\big)|\nabla \varphi|^{p\theta'}\rho_x\rangle^{\frac{1}{p\theta'}} \biggr)^{\frac{1}{1+\gamma}},
\end{align*}
where the RHS converges to $0$ as $m \rightarrow \infty$. 
We are in position to prove \eqref{c}:
\begin{align*}
&\left| \mathbb E_y^{\varepsilon_{n_k}}\int_0^t (b_{\varepsilon_{n_k}} \cdot \nabla \varphi)(\omega_s) G(\omega)ds - \mathbb E_y\int_0^t (b \cdot \nabla \varphi)(\omega_s) G(\omega)ds\right| \\
& \leq \left| \mathbb E_y^{\varepsilon_{n_k}}\int_0^t |b_{\varepsilon_{n_k}}-b_{\varepsilon_m}| |\nabla \varphi|(\omega_s) |G(\omega)|ds \right| \\
& + \left| \mathbb E_y^{\varepsilon_{n_k}}\int_0^t (b_{\varepsilon_m} \cdot \nabla \varphi)(\omega_s) G(\omega)ds  - \mathbb E_y\int_0^t (b_{\varepsilon_m} \cdot \nabla \varphi)(\omega_s) G(\omega)ds\right| \\
& + \left| \mathbb E_y\int_0^t |b_{\varepsilon_m}-b| |\nabla \varphi|(\omega_s) |G(\omega)|ds \right|,
\end{align*}
where the first and the third terms in the RHS can be made arbitrarily small using the estimates above and the boundedness of $G$ by selecting $m$, and then $n_k$, sufficiently large. The second term can be made arbitrarily small in view of \eqref{w} by selecting $n_k$ even larger. Thus, \eqref{c} follows.

\medskip

Step 3.~Let us now find a subsequence $\varepsilon_k \downarrow 0$ that works for all $y \in \mathbb R^d$ and yields a strong Markov family of probability measures $\mathbb P_y$, $y \in \mathbb R^d$, solutions to the martingale problem for SDE \eqref{sde1}. 
Denote $R^\varepsilon_\lambda f:=u_\varepsilon$, where $u_\varepsilon$ is the classical solution of 
$(\lambda-\Delta + b_\varepsilon \cdot \nabla)u_\varepsilon=f$ in $\mathbb R^d$, $f \in C_c^\infty$, $\lambda \geq \lambda_0 \vee 1$; $$R^\varepsilon_\lambda f(y)=\mathbb E_{\mathbb P^\varepsilon_y}\int_0^\infty e^{-\lambda s} f(\omega_s)ds.$$
By Theorem \ref{thm2}, $u_\varepsilon$ are uniformly in $\varepsilon$ bounded on $\mathbb R^d$. By Theorem \ref{thm1} applied to $b_\varepsilon$, solutions $u_\varepsilon$ are H\"{o}lder continuous on every compact, also uniformly in $\varepsilon>0$. By the Arzel\`{a}-Ascoli theorem and a standard diagonal argument there exists a subsequence $\varepsilon_k \downarrow 0$ such that sequence $\{R^\varepsilon_\lambda f\}$ converges locally uniformly on $\mathbb R^d$, for every $f$ in a fixed dense subset of $C_b$. Let us denote the limit by $R_\lambda f$.
The latter, and the uniform in $\varepsilon$ estimate
$
\|R^\varepsilon_\lambda f\|_\infty \leq \frac{1}{\lambda}\|f\|_\infty
$ 
allow us to extend $R_\lambda f$ to all $f \in C_b$. Thus, $R_\lambda f \in C_b$, $f \in C_b$. Now, for this subsequence $\varepsilon_k \downarrow 0$, for any $y_k \rightarrow y$, any two subsequential limits $\mathbb P^1$, $\mathbb P^2$ of $\{\mathbb P_{y_k}^{\varepsilon_k}\}$ (we use \eqref{w}) have the same finite-dimensional distributions (see \cite{B} for details, if needed) and therefore coincide: $\mathbb P_y:=\mathbb P^1 = \mathbb P^2$. Hence
$
\mathbb E_{\mathbb P_y}\int_0^\infty e^{-\lambda s} f(\omega_s)ds=R_\lambda f(y).
$
By what was proved above, $\mathbb P_y$ is a martingale solution of \eqref{sde1}. A simple argument (see \cite{B}) now gives that, for every $t>0$, $y \mapsto \mathbb E_{\mathbb P_y}f(X_t)$ is a continuous function. The latter, in turn, yields that $\{\mathbb P_y\}_{y \in \mathbb R^d}$ is a strong Markov family (the proof can be found e.g.\,in \cite{B} or \cite[Sect.\,I.3]{B2}).

\medskip

This completes the proof of assertion (\textit{i}).

\medskip

(\textit{ii}) Let $b_n$ be defined by \eqref{b_n}, so that vector fields $\{b_n\}$ do not increase the form-bounds of $b$. 
 In the end of the proof of (\textit{i}) we show that there exists a subsequence $b_{n_k}$ (for brevity, $\{b_n\}$ itself) such that, for every $f \in C_c^\infty(\mathbb R^{d})$, the classical solutions $\{u_n\}$ to elliptic equations
\begin{equation*}
\big(\lambda - \Delta+ b_n \cdot \nabla \big)u_n=f
\end{equation*}
converge locally uniformly on $\mathbb R^{Nd}$ to
\begin{equation}
\label{exp}
x \mapsto \mathbb E_{\mathbb P_x}\int_0^\infty e^{-\lambda s}f(\omega_s^1,\dots,\omega_s^N) ds, \quad x \in \mathbb R^{Nd},
\end{equation}
where $\lambda$ is assumed to be sufficiently large. This yields the local H\"{o}lder continuity of $u$.
At the same time, $u_n$ are weak solutions of \eqref{eq_el} in the sense of Definitions \ref{def_w1} and \ref{def_w3}. The possibility to pass to the limit $\varepsilon \downarrow 0$ in these definitions follows from the standard compactness argument (for details, if needed, see e.g.\,\cite{KiV}).

\medskip

(\textit{iii}) The proof goes by showing that $v_n$ constitute a Cauchy sequence in $L^\infty([0,1],L^p(\mathbb R^d))$, see \cite{Ki_Osaka}, see also \cite{KiS_theory}. At the elliptic level this was done earlier in \cite{KS}  using Trotter's theorem. The proof of the ($L^p$, $L^q$) estimate is due to \cite{S}. (Strictly speaking, these papers did not consider condition \eqref{A_3}, but it is easy to modify the  proofs there to cover the case \eqref{A_3} as well.)

\medskip

(\textit{iv}) It suffices to show that, for all $\mu \geq \mu_0$, for every $f \in C_c^\infty$, 
\begin{equation}
\label{conv_R}
R_\mu^\varepsilon f \rightarrow (\mu+\Lambda_p)^{-1}f \quad \text{ in $C_\infty$ as $\varepsilon \downarrow 0$},
\end{equation}
possibly after a modification of $(\mu+\Lambda_p)^{-1}f$ on a measure zero set. The rest follows from estimates $\|R_{\mu,\varepsilon} f\|_\infty \leq \mu^{-1}\|f\|_\infty$, $\|(\mu+\Lambda_p)^{-1} f\|_\infty \leq \mu^{-1}\|f\|_\infty$ (an immediate consequence of the fact that the corresponding semigroups are $L^\infty$ contractions) using a density argument.

Let us prove \eqref{conv_R}. Put $u_\varepsilon:=R_{\mu,\varepsilon}f$, so $u_\varepsilon$ is the classical solution to $(\mu-\Delta + b\cdot\nabla)u_\varepsilon =f$ on $\mathbb R^d$. Then, by Propositions \ref{nonhom_prop} and \ref{rho_prop} (with $|\mathsf{h}|=1$), for all $\mu (\geq 1 \vee \lambda_0) + 1$
\begin{align*}
\sup_{y \in B_{\frac{1}{2}}(x)} |u_\varepsilon(y)| \leq C  \biggl( \langle |f|^{p\theta}\rho_x\rangle^{\frac{1}{p\theta}}+ \langle |f|^{p\theta'}\rho_x\rangle^{\frac{1}{p\theta'}} \biggr)
\end{align*}
for constant $C$ independent of $\varepsilon$. It is seen now that for a fixed $f \in C_c^\infty$, for a given $\varepsilon>0$, we can find $R>0$ such that
$$
\sup_{y \in \mathbb R^d \setminus \bar{B}_R(0)} |u_\varepsilon(y)|<\varepsilon.
$$
In turn, inside the closed ball $\bar{B}_R(0)$, the family of solutions $\{u_\varepsilon\}_{\varepsilon>0}$ is equicontinuous by Theorem \ref{thm1}. So, applying Arzel\`{a}-Ascoli theorem and using the convergence result for the semigroups in $L^p$ from assertion (\textit{iii}), we obtain \eqref{conv_R}.

\medskip

(\textit{v}) The proof is an application of Proposition \ref{c_prop} and Gehring's lemma:

\begin{lemma}
\label{gehring_prop}
Assume that there exist constants $K \geq 1$, $1<q<\infty$ such that, for given $0\leq g \in L^q$, $0 \leq h \in L^q \cap L^\infty$ we have 
$$
\biggl(\frac{1}{|B_R|}\langle g^q \mathbf{1}_{B_R}\rangle \biggr)^{\frac{1}{q}} \leq \frac{K}{|B_{2R}|}\langle g \mathbf{1}_{B_{2R}}\rangle+ \biggl(\frac{1}{|B_{2R}|}\langle h^q \mathbf{1}_{B_{2R}}\rangle \biggr)^{\frac{1}{q}}
$$
for all $0<R<\frac{1}{2}$. Then $g \in L^s$ for some $s>q$ and
$$
\biggl(\frac{1}{|B_R|}\langle g^s \mathbf{1}_{B_R}\rangle \biggr)^{\frac{1}{s}} \leq C_1\biggl(\frac{1}{|B_{2R}|}\langle g^q \mathbf{1}_{B_{2R}}\rangle \biggr)^{\frac{1}{q}} + C_2\biggl(\frac{1}{|B_{2R}|}\langle h^s \mathbf{1}_{B_{2R}}\rangle \biggr)^{\frac{1}{s}}.
$$
\end{lemma}

We are in position to prove assertion (\textit{v}). Without loss of generality, $f \geq 0$, so $u_n \geq 0$.

\smallskip

Step 1. Set $(u_n)_{B_{2R}}:=\frac{1}{|B_{2R}|}\langle u_n \mathbf{1}_{B_{2R}}\rangle$. Repeating the proof of Proposition \ref{c_prop} with $p=2$ for $u_n-(u_n)_{B_{2R}}$, we obtain
\begin{equation}
\label{u_gehr}
\langle |\nabla u_n|^2 \mathbf{1}_{B_{R}}\rangle \leq \frac{K_1}{|B_{2R}|^{\frac{2}{d}}} \langle(u_n-(u_n)_{B_{2R}})^2 \mathbf{1}_{B_{2R}}\rangle + K_2\langle |f-\mu u_n|^2 \mathbf{1}_{B_{2R}}\rangle, \quad 0<R<\frac{1}{2}.
\end{equation}
By the Sobolev-Poincar\'{e} inequality,
\begin{equation}
\label{sob_p}
\biggl( \frac{1}{|B_{2R}|} \langle (u_n-(u_n)_{B_{2R}})^2 \mathbf{1}_{B_{2R}}\rangle\biggr)^{\frac{1}{2}} \leq  C |B_R|^{\frac{1}{d}} \biggl( \frac{1}{|B_{2R}|}\langle |\nabla u_n|^{\frac{2d}{d+2}} \mathbf{1}_{B_{2R}}\rangle\biggr)^{\frac{d+2}{2d}},
\end{equation}
i.e.
$$
\langle (u_n-(u_n)_{B_{2R}})^2 \mathbf{1}_{B_{2R}}\rangle \leq C^2|B_R|^{\frac{2}{d}+1} \biggl( \frac{1}{|B_{2R}|}\langle |\nabla u_n|^{\frac{2d}{d+2}} \mathbf{1}_{B_{2R}}\rangle\biggr)^{\frac{d+2}{d}}.
$$
Then the condition of the Gehring lemma is verified with $g=|\nabla u_n|^{\frac{2d}{d+2}}$, $g^q=|\nabla u_n|^2$ (so $q=\frac{d+2}{d}$) and $h=c |f-\mu u_n|^{\frac{2d}{d+2}}$. Hence there exists $s>\frac{d+2}{d}$ such that
$$
\biggl(\frac{1}{|B_R|}\langle |\nabla u_n|^{s\frac{2d}{d+2}} \mathbf{1}_{B_R}\rangle \biggr)^{\frac{1}{s}} \leq C_1\biggl(\frac{1}{|B_{2R}|}  \langle |\nabla u_n|^2 \mathbf{1}_{B_{2R}}\rangle \biggr)^{\frac{d}{d+2}} + C_2\biggl(\frac{1}{|B_{2R}|}\langle |f-\mu u_n|^{s\frac{2d}{d+2}} \mathbf{1}_{B_{2R}}\rangle \biggr)^{\frac{1}{s}},
$$
where all constants are independent of $n$. 

Now, passing in both sides of the previous inequality to the cubes (inscribed in $B_R$ in the left-hand side and circumscribed over $B_{2R}$ in the right-hand side), then considering an equally spaced grid in $\mathbb R^d$ so that the smaller cubes centered at the nodes of the grid cover $\mathbb R^d$, applying the previous estimate on each cube, and then summing up, we obtain the global estimate
$$
\|\nabla u_n\|^2_{s\frac{2d}{d+2}} \leq C_3\|\nabla u_n\|^2_2 + C_4\|f-\mu u_n\|^2_{s\frac{2d}{d+2}}.
$$

\smallskip

Step 2. Let us show that $\sup_{n}\|\nabla u_n\|^2_2<\infty$. To this end, we multiply $(\mu-\Delta + b_n \cdot \nabla)u_n=f$ by $u_n$ and integrate, obtaining
$
\mu\|u_n\|_2^2 + \|\nabla u_n\|_2^2 + \langle b_n \cdot \nabla u_n,u_n\rangle=\langle f,u_n\rangle,
$
where
$$\langle b_n \cdot \nabla u_n,u_n\rangle = -\frac{1}{2}\langle {\rm div\,}b_n,u_n^2\rangle \geq -\frac{1}{2}\langle ({\rm div\,}b_n)_+,u_n^2\rangle.$$ Hence, by our form-boundedness assumption on $({\rm div\,}b_n)_+$,
\begin{equation}
\label{e}
\left(\mu-\frac{c_{\delta_+}}{2}\right)\|u_n\|_2^2 + \bigg(1-\frac{\delta_+}{2}\bigg)\|\nabla u_n\|_2^2 \leq \langle f,u_n\rangle.
\end{equation}
So, applying the quadratic inequality in the right-hand side, we arrive at $(\mu-\frac{c_{\delta_+}}{2}-\frac{1}{2})\|u_n\|_2^2 + (1-\frac{\delta_+}{2})\|\nabla u_n\|_2^2 \leq \frac{1}{2}\|f\|_2^2$. Since $\delta_+<2$, $\sup_{n}\|\nabla u_n\|_2^2<\infty$ for $\mu \geq \mu_0:=\frac{c_{\delta_+}}{2}+\frac{1}{2}$.

\medskip

Step 3.~Next, $\|u_n\|_2 \leq C\|f\|_2$ and a priori bound $\|u_n\|_\infty \leq \|f\|_\infty$ yield $\sup_n\|u_n\|_{s\frac{2d}{d+2}}<\infty$. Hence $\sup_n\|f-\mu u_n\|^2_{s\frac{2d}{d+2}}<\infty$.

\medskip

Steps 1-3 give us a gradient bound $$\sup_{n}\|\nabla u_n\|^2_{s\frac{2d}{d+2}}<\infty$$ which we are going to use at the next step. 

\medskip

Step 4. Put $h:=u_n-u_m$. Then
\begin{equation*}
\mu \|h\|_2^2 + \|\nabla h\|_2^2 + \langle b_{n} \cdot \nabla h,h\rangle + \langle (b_{n}-b_{m})\cdot\nabla u_{m},h \rangle=0.
\end{equation*}
So, 
\begin{equation}
\label{h}
\left(\mu-\frac{c_{\delta_+}}{2}-\frac{1}{2}\right)\|h\|_2^2 + (1-\frac{\delta_+}{2})\|\nabla h\|_2^2 \leq |\langle (b_{n}-b_{m})\cdot\nabla u_{m},h \rangle|.
\end{equation}
In turn, the right-hand side
$$
|\langle (b_{n}-b_{m})\cdot\nabla u_{m},h \rangle| \leq \|b_{n}-b_{m}\|_{2-\varkappa}\|\nabla u_{m}\|_{s\frac{2d}{d+2}}2\|f\|_\infty
$$
where $0<\varkappa<1$ is defined by $$2-\varkappa:=\left(s\frac{2d}{d+2}\right)'= \frac{s\frac{2d}{d+2}}{s\frac{2d}{d+2}-1}$$ (recall that $s\frac{2d}{d+2}>2$). Since $\{b_{n}\}$ converge to $b$ in $L^{2-\varkappa}$, we obtain that the RHS of \eqref{h} converges to zero as $n$, $m \rightarrow \infty$, so $\{u_{n}\}$ is a Cauchy sequence in $L^{2}$. (This yields the independence of the limit on a particular choice of $\{b_n\}$ since we can always combine two different approximations of $b$ obtaining again a Cauchy sequence of the approximating solutions.) \hfill \qed

\begin{remark} 
\label{gehr_rem}
We carried out the proof of assertion (\textit{v}) assuming that $\delta_+<2$ instead of $\delta_+<4$ as in the other assertions. In fact, to handle $\delta_+<4$ we need to consider $(\mu-\Delta + b_n\cdot \nabla)u_n=f$ in $L^p$, $p>\frac{4}{4-\delta_+}$.
However, the step where we use the Sobolev-Poincar\'{e} inequality \eqref{sob_p} in \eqref{u_gehr} is ultimately an $L^2$ argument. Hence the need for a more restrictive condition $\delta_+<2$.
\end{remark}

\bigskip

\section{Proofs of Theorem \ref{thmK1_0}(\textit{i}),(\textit{ii}) and Theorem \ref{thmK1}(\textit{i})}

Theorem \ref{thmK1_0}(\textit{i}),(\textit{ii}) and Theorem \ref{thmK1}(\textit{i}) follows right away, in view of Lemmas \ref{p_1}, \ref{p_2}, from Theorem \ref{markov_thm}(\textit{i}), (\textit{ii}) where we consider the general SDE in $\mathbb R^{Nd}$ with $Y=(X_1,\dots,X_N)$, $B=(B_1,\dots,B_N)$, $y=(x_1,\dots,x_N)$ and drift $b:\mathbb R^{Nd} \rightarrow \mathbb R^{Nd}$ defined by \eqref{b_def}. 
\hfill \qed

\bigskip

\section{Proof of Theorem \ref{thmK1}(\textit{ii})}

This follows right away from Theorem \ref{thm1}(\textit{iii}) and Lemmas \ref{p_1}, \ref{p_2}.

\bigskip

\section{Proof of Theorem \ref{thmK1_0}(\textit{iii})-(\textit{v})}

(\textit{iii}) follows from Theorem \ref{markov_thm}(\textit{iii}) and Lemmas \ref{p_1}, \ref{p_2}.

(\textit{iv}) follows from the uniqueness result in \cite{KiM_JDE}, see also \cite{Ki_Morrey}, and Lemma \ref{p_1}.

(\textit{v}) follows from the result in \cite{KiM_strong} upon applying Lemma \ref{p_1}. The assertion before, i.e.\,that Theorem \ref{thmK1_0}(\textit{i})-(\textit{iv}) is also valid for the interaction kernels of the form \eqref{K4_cond}, follows upon applying an appropriate (straightforward) modification of Lemma \ref{p_1}.

\bigskip

\section{Proof of Theorem \ref{thmK1}(\textit{iii}), (\textit{iv})}

(\textit{iii}) Since the sum of two form-bounded vector fields is again form-bounded, we only need to improve Lemma \ref{p_2} for $K(y)=\sqrt{\kappa}\frac{d-2}{2}|y|^{-2}y$ and then simply repeat the proof of Theorem \ref{thmK1_0}(\textit{i})-(\textit{iii}).
In Lemma \ref{p_2} we have three estimates \eqref{b_est2}, \eqref{div_est} \eqref{gamma_est} for $b=(b_1,\dots,b_N):\mathbb R^{Nd} \rightarrow \mathbb R^{Nd}$, where now
\begin{equation}
b_i(x):= \sqrt{\kappa}\frac{d-2}{2}\frac{1}{N}\sum_{j=1, j \neq i}^N \frac{x_i-x_j}{|x_i-x_j|^2}, \quad x=(x_1,\dots,x_n) \in \mathbb R^{Nd}, \quad 1 \leq i \leq N.
\end{equation}
We do not need to change \eqref{b_est2} and \eqref{gamma_est} since the actual values of the form-bounds there are not important for the sake of repeating the proof of Theorem \ref{thmK1}, only their finiteness matters. The form-bound $\delta_+$ in \eqref{div_est}, however, plays a crucial role.
Let us estimate it using the many-particle Hardy's inequality \eqref{multi_hardy}:
\begin{align*}
({\rm div\,}b)_+={\rm div\,}b & =\sum_{i=1}^N\frac{1}{N}\sum_{j=1,j\neq i}^N {\rm div\,}K(x_i-x_j) \\
& = \sqrt{\kappa}\frac{(d-2)^2}{N} \sum_{1 \leq i<j \leq N}\frac{1}{|x_i-x_j|^2}.
\end{align*}
Applying \eqref{multi_hardy}, we obtain that $({\rm div\,}b)^{\frac{1}{2}}_+ \in \mathbf{F}_{\delta_+}$ with $\delta_+=\sqrt{\kappa}$. Armed with this result, i.e.\,a replacement of Lemma \ref{p_2}, we repeat the proof of Theorem \ref{thmK1} (i.e.\,we apply Theorem \ref{markov_thm} where we still have $\delta_+<4$).

(\textit{iv}) 
We apply Theorem \ref{thm_desing1} from Appendix \ref{app_desing}. There $\Omega:=\mathbb R^{Nd}$ 
and $\mu$ is the Lebesgue measure on $\mathbb R^{Nd}$. The semigroup $e^{-t\Lambda}$ and thus the heat kernel $e^{-t\Lambda}(x,y)$ is from assertion (\textit{ii}). The weights $\{\varphi_s\}_{s>0}$ are defined by 
\begin{equation*}
\varphi_s(x):=\prod_{1 \leq i<j \leq N} \eta(s^{-\frac{1}{2}}|x_i-x_j|), \quad s>0. 
\end{equation*}
%where $\eta$ is defined in Theorem \ref{thm_model}(\textit{ii}). 
It is easily seen that these weights $\varphi_s$ satisfy conditions \eqref{S2} and \eqref{S3} of Theorem \ref{thm_desing1}. 
In turn, condition \eqref{S1} with $j=\frac{d}{d-2}$ and $r>2(2-\frac{N-1}{N}\sqrt{\kappa})^{-1}$ was verified in Theorem \ref{thmK1}(\textit{ii}) under hypothesis \eqref{K3_cond}, see \eqref{pq}. 
Let us verify the ``desingularizing $L^1 \rightarrow L^1$ bound'' \eqref{S4} for $0<s \leq t$:

Step 1. Set $$
\qquad \eta_s(r):=\eta(s^{-\frac{1}{2}}r), \quad r>0
$$ and put 
$$
\varphi_{s}^\varepsilon(x)\equiv \varphi^\varepsilon(x):=\prod_{1 \leq i<j \leq N} \eta_s(|x_i-x_j|_\varepsilon), \quad |x_i-x_j|_\varepsilon:=\sqrt{|x_i-x_j|^2+\varepsilon}, \quad \varepsilon>0.
$$
Define
$$
\psi^\varepsilon(x):=\prod_{1 \leq i<j \leq N} (s^{-\frac{1}{2}}|x_i-x_j|_\varepsilon)^{-\sqrt{\kappa}\frac{d-2}{2}\frac{1}{N}}.
$$
and put 
$$b_\varepsilon:=-\frac{\nabla_x \psi^\varepsilon}{\psi^\varepsilon} \quad (\text{clearly, independent of $s$}).$$ 
This is a vector field $\mathbb R^{Nd} \rightarrow \mathbb R^{Nd}$ such that
$$
b_\varepsilon \cdot \nabla_x=\sqrt{\kappa}\frac{d-2}{2}\frac{1}{N}\sum_{i=1}^N\sum_{j=1, j \neq i}^N \frac{x_i-x_j}{|x_i-x_j|_\varepsilon^2} \cdot \nabla_{x_i}.
$$
 Without loss of generality, we discuss the (minus) first component $\mathbb R^{Nd} \rightarrow \mathbb R^d$ of $b_\varepsilon$:
\begin{align*}
\frac{\nabla_{x_1} \psi^\varepsilon}{\psi^\varepsilon}& =\frac{1}{\psi^\varepsilon}\sum_{2 \leq k \leq N} \prod_{1 \leq i<j \leq N, (i,j) \neq (1,k)} |x_i-x_j|_\varepsilon^{-\sqrt{\kappa}\frac{d-2}{2}\frac{1}{N}} \nabla_{x_1}\biggl(|x_1-x_k|^{-\sqrt{\kappa}\frac{d-2}{2}\frac{1}{N}}_\varepsilon\biggr) \\
& = \sum_{2 \leq k \leq N} \frac{\nabla_{x_1}|x_1-x_k|^{-\sqrt{\kappa}\frac{d-2}{2}\frac{1}{N}}_\varepsilon}{|x_1-x_k|^{-\sqrt{\kappa}\frac{d-2}{2}\frac{1}{N}}_\varepsilon} \\
& = -\sqrt{\kappa}\frac{d-2}{2}\frac{1}{N} \sum_{2 \leq k \leq N} \frac{x_1-x_k}{|x_1-x_k|_\varepsilon^2}.
\end{align*}
In the same way,
\begin{align*}
\frac{\nabla_{x_1} \varphi^\varepsilon}{\varphi^\varepsilon}& =\sum_{2 \leq k \leq N} \frac{\nabla_{x_1}\eta_s(|x_1-x_k|_\varepsilon)}{\eta_s(|x_1-x_k|_\varepsilon)}.
\end{align*}
We now compare these quantities (this will be needed at the next step):

(a) If $|x_1-x_k|_\varepsilon \leq \sqrt{s}$ for all $2 \leq k \leq N$, then, by the definition of $\eta$, 
 $$-\frac{\nabla_{x_1} \psi^\varepsilon}{\psi^\varepsilon}+\frac{\nabla_{x_1} \varphi^\varepsilon}{\varphi^\varepsilon}=0.$$
Therefore,
\begin{equation*}
\frac{\nabla_{x_1}\varphi^\varepsilon}{\varphi^\varepsilon}\cdot\big(-\frac{\nabla_{x_1} \psi^\varepsilon}{\psi^\varepsilon}+\frac{\nabla_{x_1}\varphi^\varepsilon}{\varphi^\varepsilon}\big)=0, \quad
{\rm div}_{x_1}\big(-\frac{\nabla_{x_1} \psi^\varepsilon}{\psi^\varepsilon}+\frac{\nabla_{x_1}\varphi^\varepsilon}{\varphi^\varepsilon}\big)=0.
\end{equation*}

(b) If there exists one $k_0$ such that $|x_1-x_{k_0}|_\varepsilon \geq 2\sqrt{s}$, but for the other $k \neq k_0$ $|x_1-x_k|_\varepsilon \leq \sqrt{s}$, then, since $x_1 \mapsto \eta_s(|x_1-x_{k_0}|_\varepsilon)$ is constant and so $\nabla_{x_1}\varphi^\varepsilon=0$, we have
$$
\frac{\nabla_{x_1} \psi^\varepsilon}{\psi^\varepsilon}-\frac{\nabla_{x_1} \varphi^\varepsilon}{\varphi^\varepsilon} = -\sqrt{\kappa}\frac{d-2}{2}\frac{1}{N} \frac{x_1-x_{k_0}}{|x_1-x_{k_0}|_\varepsilon^2}.
$$
Hence
\begin{equation*}
\frac{\nabla_{x_1}\varphi^\varepsilon}{\varphi^\varepsilon}\cdot\big(-\frac{\nabla_{x_1} \psi^\varepsilon}{\psi^\varepsilon}+\frac{\nabla_{x_1}\varphi^\varepsilon}{\varphi^\varepsilon}\big)=0, \quad
\bigg|{\rm div}_{x_1}\big(-\frac{\nabla_{x_1} \psi^\varepsilon}{\psi^\varepsilon}+\frac{\nabla_{x_1}\varphi^\varepsilon}{\varphi^\varepsilon}\big) \bigg| \leq \sqrt{\kappa}\frac{(d-2)^2}{2}\frac{1}{N} 4s^{-1}.
\end{equation*}

(c) More generally, if there exist $2 \leq M \leq N-1$ indices $k_0$ such that $|x_1-x_{k_0}|_\varepsilon \geq 2\sqrt{s}$, but for the other $k \neq k_0$ $|x_1-x_k|_\varepsilon \leq \sqrt{s}$, then we have
\begin{equation*}
\frac{\nabla_{x_1}\varphi^\varepsilon}{\varphi^\varepsilon}\cdot\big(-\frac{\nabla_{x_1} \psi^\varepsilon}{\psi^\varepsilon}+\frac{\nabla_{x_1}\varphi^\varepsilon}{\varphi^\varepsilon}\big)=0, \quad
\bigg|{\rm div}_{x_1}\big(-\frac{\nabla_{x_1} \psi^\varepsilon}{\psi^\varepsilon}+\frac{\nabla_{x_1}\varphi^\varepsilon}{\varphi^\varepsilon}\big) \bigg| \leq \sqrt{\kappa}\frac{(d-2)^2}{2}\frac{M}{N} 4s^{-1}.
\end{equation*}

Over the annuli $\sqrt{s}<|x_1-x_k|_\varepsilon < 2\sqrt{s}$ we make a change of variable to finally obtain, for all possible values of $|x_1-x_k|_\varepsilon$, $2 \leq k \leq N$,
\begin{equation*}
\bigg|\frac{\nabla_{x_1}\varphi^\varepsilon}{\varphi^\varepsilon}\cdot\big(-\frac{\nabla_{x_1} \psi_\varepsilon}{\psi_\varepsilon}+\frac{\nabla_{x_1}\varphi^\varepsilon}{\varphi^\varepsilon}\big)\bigg| \leq c_1\frac{N-1}{N}s^{-1} , \quad
\bigg|{\rm div}_{x_1}\big(-\frac{\nabla_{x_1} \psi^\varepsilon}{\psi^\varepsilon}+\frac{\nabla_{x_1}\varphi^\varepsilon}{\varphi^\varepsilon}\big) \bigg| \leq c_2\frac{N-1}{N} s^{-1}
\end{equation*}
for constants $c_1$ and $c_2$ independent of $\varepsilon$ and $s$.

The same holds for the other components of $b_\varepsilon=-\frac{\nabla_x \psi_\varepsilon}{\psi_\varepsilon}$. Thus,
\begin{equation}
\label{b_div_zero5}
\bigg|\frac{\nabla_{x}\varphi^\varepsilon}{\varphi^\varepsilon}\cdot\big(b_\varepsilon + \frac{\nabla_{x} \varphi^\varepsilon}{\varphi^\varepsilon}\big)\bigg| \leq c_1\frac{N-1}{\sqrt{N}}s^{-1} , \quad
\bigg|{\rm div}\big(b_\varepsilon + \frac{\nabla_{x} \varphi^\varepsilon}{\varphi^\varepsilon}\big) \bigg| \leq c_2\frac{N-1}{\sqrt{N}} s^{-1}.
\end{equation}

Step 2.\,Define the approximating operators $\Lambda_\varepsilon:=-\Delta_x + b_\varepsilon \cdot \nabla_x$ having domain $\mathcal W^{2,1}=(1-\Delta)^{-1}L^1$. Since $\varphi^\varepsilon$, $(\varphi^\varepsilon)^{-1}$ are bounded and continuous, 
one sees right away that $\varphi^\varepsilon e^{-t \Lambda_\varepsilon}(\varphi^\varepsilon)^{-1}$ is a strongly continuous semigroup in $L^1$ whose generator coincides with $-\varphi_\varepsilon  \Lambda_\varepsilon(\varphi_\varepsilon)^{-1}$ having domain $\mathcal W^{2,1}$. This generator can be computed explicitly:
\begin{equation}
\label{gen_calc}
\varphi^\varepsilon  \Lambda^\varepsilon(\varphi^\varepsilon)^{-1} = -\Delta +\nabla\cdot (b_\varepsilon+2\frac{\nabla\varphi^\varepsilon}{\varphi^\varepsilon})+W_\varepsilon, 
\end{equation}
$$
W_\varepsilon:=-\frac{\nabla\varphi^\varepsilon}{\varphi^\varepsilon}\cdot\big(b_\varepsilon+\frac{\nabla\varphi^\varepsilon}{\varphi^\varepsilon}\big)-{\rm div}\big(b_\varepsilon+\frac{\nabla\varphi^\varepsilon}{\varphi^\varepsilon}\big).
$$
By \eqref{b_div_zero5}, potential $W_\varepsilon$ is (uniformly in $\varepsilon$) bounded: $|W_\varepsilon| \leq \frac{N-1}{\sqrt{N}}\frac{c}{s}$ for a constant $c$ independent of $\varepsilon$. 
Employing formula \eqref{gen_calc} and using the general fact that $e^{t(\Delta - \nabla \cdot \mathsf{f})}$ is an $L^1$ contraction, we obtain
\begin{equation}
\label{d_astast}
\|\varphi^\varepsilon e^{-t \Lambda^\varepsilon}(\varphi^\varepsilon)^{-1}h\|_1\leq e^{c\frac{N-1}{\sqrt{N}}\frac{t}{s}}\|h\|_1, \;\; h\in L^1.
\end{equation}
It remains to pass to the limit $\varepsilon \downarrow 0$ in \eqref{d_astast}. This is done at the next step.

\smallskip

Step 3. Define $b=-\frac{\nabla_x \psi}{\psi}$, where $\psi(x)=\prod_{1 \leq i<j \leq N} |x_i-x_j|^{-\sqrt{\kappa}\frac{d-2}{2}\frac{1}{N}}$ is a Lyapunov function of the formal adjoint of $\Lambda$ 
(i.e.\,\ref{heat_rem} holds). Then
$$
b\cdot \nabla_x=\sqrt{\kappa}\frac{d-2}{2}\frac{1}{N}\sum_{i=1}^N\sum_{j=1, j \neq i}^N \frac{x_i-x_j}{|x_i-x_j|^2} \cdot \nabla_{x_i}
$$
It is seen using e.g.\,the Monotone convergence theorem that $b_\varepsilon \rightarrow b$ in $[L^2_{\loc}]^{Nd}$. Moreover, the vector fields $b_\varepsilon$ do not increase the form-bound $\delta=\kappa \big(\frac{N-1}{N}\big)^2~(<4)$ of $b$. Therefore, by Theorem \ref{markov_thm}(\textit{iii}),
\begin{equation}
\label{conv_s}
e^{-t\Lambda_\varepsilon} \rightarrow e^{-t\Lambda} \quad \text{ in } L^r(\mathbb R^{Nd}), 
\end{equation}
where $r>\frac{2}{2-\frac{N-1}{N}\sqrt{\kappa}}$.

Now, from \eqref{d_astast} we have  $$\|\varphi_\varepsilon e^{-t \Lambda^\varepsilon}g\|_1 \leq e^{c\frac{N-1}{\sqrt{N}}\frac{t}{s}}\|\varphi^\varepsilon g\|_1, \quad g \in \varphi L^1 \cap L^\infty.$$ In view of \eqref{conv_s} and since $\varphi^\varepsilon \rightarrow \varphi$ a.e., we can use Fatou's lemma to obtain $\|\varphi e^{-t \Lambda}g\|_1 \leq e^{c\frac{t}{s}}\|\varphi g\|_1$, which yields condition ($S_4$) of Theorem \ref{thm_desing1} (recall that by our assumption $s \geq t$). 

\medskip

Thus, Theorem \ref{thm_desing1} applies and gives assertion (\textit{iv}) of Theorem \ref{thmK1}. \hfill \qed

\bigskip

\section{Proof of Theorem \ref{thm_grad}}

We will need the following result on the regularization of the vector field $b$ in Theorem \ref{thm_grad}.

\begin{lemma} 
\label{B_lem}
Assume that $b \in [W^{1,1}_{\loc}(\mathbb R^d)]^d$ has symmetric Jacobian $Db=(\nabla_k b_i)_{k,i=1}^d$ and 
the negative part $B_-$ of matrix $$B(b):=Db-\frac{{\rm div\,}b}{q}I, \quad \text{ for some $q>(d-2) \vee 2$,}$$ 
has normalized eigenvectors $e_j$ and eigenvalues $\lambda_j \geq 0$ satisfying $\sqrt{\lambda_j}e_j \in \mathbf{F}_{\nu_j}$.
Set
$
\nu:=\sum_{j=1}^d \nu_j.
$ 
Set $b_\varepsilon:=E_\varepsilon b$.
The following are true:

{\rm 1.}~ $$B(b_\varepsilon)+E_\varepsilon B_- \geq 0,$$

{\rm 2.}~
\begin{equation}
\label{B_bd}
\langle B_-h,h\rangle \leq \nu \langle |\nabla|h||^2\rangle+c_\nu\langle |h|^2\rangle,
\end{equation}
and
$$
\langle (E_\varepsilon B_-) h,h\rangle \leq \nu \langle |\nabla|h||^2\rangle+c_\nu\langle |h|^2\rangle, \quad\varepsilon>0,
$$
\noindent for all $h \in [C_c^\infty(\mathbb R^d)]^d$,
with $c_\nu:=\sum_{j=1}^d c_{\nu_j}$.
\end{lemma}

\begin{proof} 1. We have, by definition, $B(b)=B_+-B_-$, and
$
B(b_\varepsilon)=E_\varepsilon B_+ - E_\varepsilon B_-.
$
Clearly, $E_\varepsilon B_+ \geq 0$, which yields the required.

2. We have $
B_-=\sum_{j=1}^d \lambda_j e_j e_j^{\top}.
$
Put for brevity $\lambda=\lambda_j$ and $e=e_j$. Denote the components of $e$ by $e^k$, $k=1,\dots,d$. Then
$
\langle  \lambda (ee^{\top}) h,h\rangle=\sum_{k,i=1}^d \langle h_k \sqrt{\lambda}e^k \sqrt{\lambda}e^i h_i \rangle = \langle \lambda (h \cdot e)^2\rangle  \leq \langle \lambda |h|^2|e|^2\rangle.
$
Therefore,
\begin{align*}
\langle B_- h,h\rangle & \leq \sum_{j=1}^d  \langle \lambda_j|h|^2|e_j|^2\rangle \\
& (\text{we use $\sqrt{\lambda_j}e_j \in \mathbf{F}_{\nu_j}$}) \\
& \leq \sum_{j=1}^d \nu_j \langle |\nabla |h||^2\rangle + \sum_{j=1}^d c_{\nu_j} \langle |h|^2\rangle,
\end{align*}
which gives us the first inequality in assertion 2.

Let us prove the second inequality in assertion 2. Writing again $\lambda=\lambda_j$ and $e=e_j$ and denoting the $k$-th component of $e$ by $e^k$, we have
\begin{align*}
\langle  E_\varepsilon (\lambda ee^{\top})h,h\rangle & =\sum_{k,i=1}^d \langle  E_\varepsilon (\sqrt{\lambda}e^k \sqrt{\lambda}e^i) h_k h_i \rangle =\sum_{k,i=1}^d \langle  \sqrt{\lambda}e^k \sqrt{\lambda}e^i E_\varepsilon (h_k h_i) \rangle \\
& 
\leq \sum_{k,i=1}^d \langle \sqrt{E_\varepsilon|h_k|^2} \sqrt{\lambda}|e^k| \sqrt{\lambda}|e^i| \sqrt{E_\varepsilon|h_i|^2} \rangle \\
& \leq \langle \lambda |e|^2,|h_\varepsilon|^2 \rangle, 
\end{align*}
where $h_\varepsilon$ denotes the vector field  with $k$-th component $\sqrt{E_\varepsilon|h_k|^2}$.
Hence, using the previous estimate, we obtain
\begin{align*}
\langle (E_\varepsilon B_-) h,h\rangle & = \sum_{j=1}^d \langle E_\varepsilon (\lambda_j e_je_j^{\top})h,h\rangle  \leq \sum_{j=1}^d  \langle \lambda_j |e_j|^2,|h_\varepsilon|^2 \rangle \\
& (\text{use } \sqrt{\lambda_j}e_j \in \mathbf{F}_{\nu_j}) \\
& \leq \nu \langle |\nabla |h_\varepsilon||^2\rangle + c_\nu \langle |h_\varepsilon|^2\rangle \\
& (\text{note that $|h_\varepsilon|=\sqrt{E_\varepsilon |h|^2}$ and apply \eqref{phi_eps}}) \\
& \leq \nu\langle |\nabla |h||^2\rangle + c_\nu \langle |h|^2\rangle,
\end{align*}
as needed.
\end{proof}

\subsubsection*{Proof of Theorem \ref{thm_grad} in the case drift $b$ satisfies condition \eqref{C_2}}

We start with the 
proof of assertion {\rm(\textit{ii})}.
Put 
$$w := \nabla u, \quad w_i :=\nabla_i u.$$ 
Multiplying equation $(\mu-\Delta +b \cdot \nabla )u=f$ by the test function $$\phi := - \sum_{i=1}^d \nabla_i (w_i |w|^{q-2})=- \nabla \cdot (w |w|^{q-2})$$ and integrating by parts twice in $\langle -\Delta u,\phi\rangle$, i.e.
\begin{align*}
\langle -\Delta u, - \sum_{i=1}^d \nabla_i (w_i |w|^{q-2})\rangle & = \sum_{i=1}^d \langle \nabla_i w, \nabla (w_i|w|^{q-2})\rangle  \equiv \sum_{i=1}^d \langle \nabla w_i, \nabla (w_i|w|^{q-2})\rangle \\
& = \sum_{i=1}^d \langle |\nabla w_i|^2 |w|^{q-2}\rangle + (q-2)\sum_{i=1}^d \langle \nabla w_i,w_i|w|^{q-3}\nabla |w|\rangle \\
& = \sum_{i=1}^d \langle |\nabla w_i|^2 |w|^{q-2}\rangle + (q-2) \langle \frac{1}{2}\nabla |w|^2,|w|^{q-3}\nabla |w|\rangle,
\end{align*}
 we obtain
\begin{equation}
\label{prin_ineq}
\mu \langle |w|^q\rangle + I_q + (q-2)J_q + \langle b \cdot w, \phi \rangle = \langle f, \phi\rangle,
\end{equation}
where 
$$
I_q := \sum_{i=1}^d \big\langle |\nabla w_i |^2, |w|^{q-2} \big\rangle, \quad J_q := \big\langle |\nabla |w| |^2, |w|^{q-2} \big\rangle.
$$

\noindent Step 1.~Regarding term $ \langle b \cdot w, \phi \rangle$ in \eqref{prin_ineq}, we have
\begin{align*}
\langle b \cdot w, \phi\rangle & = \langle \tilde{B} w,w|w|^{q-2}\rangle + \langle b \cdot \nabla |w|,|w|^{q-1}\rangle \qquad \tilde{B}:=(\nabla_k b_i)_{k,i=1}^d\\
& = \langle \tilde{B} w,w|w|^{q-2}\rangle - \frac{1}{q}\langle {\rm div\,}b,|w|^{q}\rangle \\
& \geq -\langle B_- w,w|w|^{q-2}\rangle.
\end{align*}
Hence, applying \eqref{B_bd}, we arrive at
\begin{align*}
\langle b \cdot w, \phi\rangle & \geq -\nu \langle \big|\nabla |w|^{\frac{q}{2}}\big|^2 \rangle - c_\nu\langle |w|^q \rangle\\
 & = -\nu \frac{q^2}{4}J_q - c_\nu\langle |w|^q \rangle,
\end{align*}
so \eqref{prin_ineq} yields
\begin{equation}
\label{main_ineq}
(\mu- c_\nu)\langle |w|^q \rangle + I_q+\biggl(q-2 - \frac{q^2}{4}\nu \biggr)J_q \leq  \langle f, \phi\rangle.
\end{equation}

\medskip

\noindent Step 2.~Let us estimate $\langle f, \phi\rangle$ in the previous inequality.  To this end, we evaluate $\phi$:
\begin{align}
\langle f, \phi\rangle=-\langle f, |w|^{q-2} \Delta u\rangle - (q-2)\langle f,|w|^{q-3} w \cdot \nabla |w| \rangle.
\label{f0_term}
\end{align}

(a) We estimate
\begin{equation}
\label{f_term}
|\langle f, |w|^{q-2} \Delta u \rangle| \leq \varepsilon_0 \langle |w|^{q-2} |\Delta u|^2\rangle + \frac{1}{4\varepsilon_0}\langle f^2,|w|^{q-2}\rangle, 
\end{equation}
where $\varepsilon_0>0$ will be chosen sufficiently small.

Let us deal with the first term in the RHS of \eqref{f_term}. Representing $|\Delta u|^2=|\nabla \cdot w|^2$ and integrating by parts twice, we obtain
\begin{align*}
\langle |w|^{q-2} |\Delta u|^2\rangle & =-\langle  \nabla |w|^{q-2}\cdot w, \Delta u \rangle   +\sum_{i=1}^{d}\langle w_i \nabla |w|^{q-2}, \nabla w_i \rangle +I_{q} \\
& \leq (q-2)\bigg [\frac{1}{4\varkappa} \langle   |w|^{q-2}|\Delta u|^{2} \rangle  +\varkappa J_{q} \bigg]
+ (q-2)\bigg(\frac{1}{2}I_{q}+ \frac{1}{2}J_{q}\bigg) + I_q.
\end{align*}
So, for any fixed $\varkappa > \frac{q-2}{4}$,
\begin{equation}
\label{E0}
\bigg(1-\frac{q-2}{4\varkappa}\bigg) \langle |w|^{q-2} |\Delta u |^{2} \rangle \leq I_{q}+(q-2) \bigg (\varkappa J_{q}+\frac{1}{2}I_{q}+ \frac{1}{2}J_{q} \bigg ).
\end{equation}

Let us handle the second term in the RHS of \eqref{f_term}:
\begin{align*}
\langle f^2,|w|^{q-2}\rangle & \leq \|f\|^2_{\frac{qd}{d+q-2}}\|w\|_{\frac{qd}{d-2}}^{q-2}\\
& \leq c_S\|f\|^2_{\frac{qd}{d+q-2}}\|\nabla |w|^{\frac{q}{2}}\|_2^{2\frac{(q-2)}{q}} = C\|f\|^2_{\frac{qd}{d+q-2}}J_q^{\frac{q-2}{q}}, \quad C=\frac{c_S q^2}{4}\\
&\leq \frac{q-2}{q}C\varepsilon^{\frac{q}{q-2}} J_q+\frac{2}{q}C\varepsilon^{-\frac{q}{2}}\|f\|_{\frac{qd}{d+q-2}}^q.
\end{align*}

\smallskip

(b) We estimate
\begin{align*}
 (q-2) |\langle -f, |w|^{q-3} w \cdot \nabla |w| \rangle| & \leq (q-2) J_q^\frac{1}{2} \langle f^2, |w|^{q-2} \rangle^\frac{1}{2} \\
& \leq (q-2)\bigl(\varepsilon_1 J_q + 4\varepsilon_1^{-1}\langle f^2, |w|^{q-2} \rangle \bigr),
\end{align*}
where we estimate the very last term in the same way as above.

Substituting the above estimates in \eqref{f0_term}, we obtain
\begin{equation}
\label{f3_term}
|\langle f, \phi\rangle| \leq c\varepsilon_0 (I_q+J_q) + \frac{c_1(\varepsilon,\varepsilon_1)}{\varepsilon_0}J_q + \frac{c_2(\varepsilon,\varepsilon_1)}{\varepsilon_0}\|f\|_{\frac{qd}{d+q-2}}^q,
\end{equation}
where $c_1(\varepsilon,\varepsilon_1)>0$ can be made as small as needed by first selecting $\varepsilon_1$ sufficiently small, and then selecting $\varepsilon$ even smaller.

\medskip

\noindent Step 3.~Now, we return to \eqref{main_ineq}. By \eqref{f3_term},
$$
\bigl(\mu - c_\nu\bigr)\langle |w|^q\rangle + (1-c\varepsilon_0)I_q + \bigl(q-2 - \frac{q^2}{4}\nu-c \varepsilon_0 - \frac{c_1(\varepsilon,\varepsilon_1)}{\varepsilon_0}\bigr)J_q \leq  \frac{c_2(\varepsilon,\varepsilon_1)}{\varepsilon_0}\|f\|_{\frac{qd}{d+q-2}}^q.
$$
By the pointwise inequality  $$|\nabla |w||^2 = \biggl|\frac{\sum_{i=1}^d w_i \nabla w_i}{|w|} \biggr|^2 \leq \biggl(\frac{\sum_{i=1}^d |w_i| |\nabla w_i|}{|w|} \biggr)^2 \leq \sum_{i=1}^d|\nabla w_i|^2,$$ we have $$J_q \leq I_q.$$
In particular, provided $\varepsilon_0$ is sufficiently small so that $1-c \varepsilon_0 \geq 0$, we have
$(1-c\varepsilon_0)I_q \geq (1-c\varepsilon_0)J_q$.
Therefore,
\begin{equation}
\label{ineq_9}
\bigl(\mu - c_\nu\bigr)\langle |w|^q\rangle + \bigl(q-1 - \frac{q^2}{4}\nu-c(\varepsilon_0,\varepsilon,\varepsilon_1) \bigr)J_q \leq  C(\varepsilon_0,\varepsilon,\varepsilon_1)\|f\|_{\frac{qd}{d+q-2}}^q,
\end{equation}
where constant $c(\varepsilon_0,\varepsilon,\varepsilon_1)$ can be made as small as needed by first selecting $\varepsilon_1$ sufficiently small, and then selecting $\varepsilon$ even smaller ($\varepsilon_0$ is already fixed). Take $\mu_0:=c_\nu$. Recalling that $J_q=\frac{4}{q^2}\|\nabla |\nabla u|^{\frac{q}{2}}\|_2^2$, we obtain the required gradient estimate from \eqref{ineq_9}.

\medskip

\noindent \textit{Proof of assertion {\rm(\textit{i})}. }  Steps 1 and 3 do not change. Step 2 now consists of estimating $\langle |\mathsf{g}|f,\phi\rangle$, which we represent as
\begin{align*}
\langle  |\mathsf{g}|f, \phi\rangle=-\langle  |\mathsf{g}|f, |w|^{q-2} \Delta u\rangle - (q-2)\langle |\mathsf{g}|f,|w|^{q-3} w \cdot \nabla |w| \rangle.
\end{align*}

($a'$) We have
$$
|\langle |\mathsf{g}|f, |w|^{q-2} \Delta u \rangle| \leq \varepsilon_0 \langle |w|^{q-2} |\Delta u|^2\rangle + \frac{1}{4\varepsilon_0}\langle |\mathsf{g}|^2f^2,|w|^{q-2}\rangle, \quad \varepsilon_0>0,
$$
where $\langle |w|^{q-2} |\Delta u|^2\rangle$ is estimates in the same way as in (a) above, and
\begin{align*}
\langle |\mathsf{g}|^2f^2,|w|^{q-2}\rangle & = \langle  |\mathsf{g}|^{2-\frac{4}{q}}|w|^{q-2},  |\mathsf{g}|^{\frac{4}{q}}f^{2}\rangle \\
&\leq \frac{q-2}{q}\varepsilon^{\frac{q}{q-2}}\langle |\mathsf{g}|^{2} |w|^{q} \rangle +\frac{2}{q}\varepsilon^{-\frac{q}{2}} \langle \rho |\mathsf{g}|^2 f^q\rangle  \\
& \text{(we are using $\mathsf{g} \in \mathbf{F}_{\delta_1}$)}\\
&\leq \frac{q-2}{q}\varepsilon^{\frac{q}{q-2}} \bigg [\delta_1 \frac{q^{2}}{4}J_{q}+ c_{\delta_1} \langle |w|^{q}\rangle\bigg ] +\frac{2}{q}\varepsilon^{-\frac{q}{2}} \big\langle  |\mathsf{g}|^2 f^q \big\rangle.
\end{align*}

($b'$) We estimate
\begin{align*}
 (q-2) |\langle |\mathsf{g}|f, |w|^{q-3} w \cdot \nabla |w| \rangle| & \leq (q-2) J_q^\frac{1}{2} \langle |\mathsf{g}^2|f^2, |w|^{q-2} \rangle^\frac{1}{2} \\
& \leq (q-2)\bigl(\varepsilon_1 J_q + 4\varepsilon_1^{-1}\langle |\mathsf{g}|^2f^2, |w|^{q-2} \rangle \bigr),
\end{align*}
where we bound $\langle |\mathsf{g}|^2f^2, |w|^{q-2} \rangle$ as in ($a'$).

Now, arguing as above, we arrive at
\begin{equation*}
\bigl(\mu - c_\nu-c_0(\varepsilon_0,\varepsilon_1,\varepsilon)\bigr)\langle |w|^q\rangle + \bigl(q-1 - \frac{q^2}{4}\nu-c(\varepsilon_0,\varepsilon,\varepsilon_1) \bigr)J_q \leq  C(\varepsilon_0,\varepsilon,\varepsilon_1)\langle  |\mathsf{g}|^2 f^q \big\rangle,
\end{equation*}
where constant $c(\varepsilon_0,\varepsilon,\varepsilon_1)$ can be made as small as needed by selecting $\varepsilon_1$ sufficiently small and then selecting $\varepsilon$ even smaller. So, taking $\mu_0:=c_\mu + c_0(\varepsilon_0,\varepsilon_1,\varepsilon)$, we obtain the required gradient estimate.

\subsubsection*{Proof of Theorem \ref{thm_grad} in the case drift $b$ satisfies condition \eqref{C_1}.} One needs to estimate term $ \langle b \cdot w, \phi \rangle$ in \eqref{prin_ineq} differently. Indeed, $b$ is no longer differentiable and hence one cannot integrate by parts. Instead, arguing as in \cite{KS}, we evaluate the test function $\phi$ as
$$
\langle b \cdot w, \phi\rangle=-\langle b \cdot w, |w|^{q-2} \Delta u\rangle - (q-2)\langle b \cdot w,|w|^{q-3} w \cdot \nabla |w| \rangle,
$$
and then re-uses the elliptic equation to express $\Delta u$ in terms of $\mu u$, $b \cdot w$ and $f$ (or $|\mathsf{g}|f$). Then we repeat the argument from \cite{KS} up to the estimates on $|\langle f,\phi\rangle|$ (assertion (\textit{ii})) and $|\langle |\mathsf{g}|f,\phi\rangle|$ (assertion (\textit{i})), which we take from Step 2 above.

\bigskip

\section{Proof of Theorem \ref{unique_thm}}

Let $b_n$ be constructed as in Lemma \ref{fb_approx_lem}, i.e.
$$
b_n=E_{\varepsilon_n}b, \quad \varepsilon_n \downarrow 0,
$$
so that $b_n$ are bounded, smooth, converge to $b$ locally in $L^2$ and, crucially,   do not increase neither form-bound $\delta$ of $b$ nor constant $c_\delta$.

A comment regarding the case when $b$ satisfies condition \eqref{C_2}. 
Below we use gradient bounds from Theorem \ref{thm_grad} for vector fields $b_n$. The proof of these gradient bounds depends on a somewhat less restrictive condition than \eqref{C_2}, i.e.\,$b \in \mathbf{F}_\delta$, $\delta<\infty$, and
\begin{equation}
\label{e_in}
\langle (E_{\varepsilon_n} B_-) h,h\rangle \leq \nu \langle |\nabla|h||^2\rangle+c_\nu\langle |h|^2\rangle, 
\end{equation}
where $B_-$ is the negative part of matrix $(\nabla_k b^i)_{k,i=1}^d -\frac{{\rm div\,}b}{q}I$. (Indeed, if $B_+$ denotes the positive part of the last matrix, we have
$$
(\nabla_k b_n^i)_{k,i=1}^d -\frac{{\rm div\,}b_n}{q}I=E_{\varepsilon_n}B_+-E_{\varepsilon_n}B_-, \quad E_{\varepsilon_n}B_\pm \geq 0,
$$
and can repeat the proof of Theorem \ref{thm_grad} for $b_n$ and $E_{\varepsilon_n} B_-$.)
By Lemma \ref{B_lem}, inequality \eqref{e_in} does hold  with constants $\nu=\sum_{j=1}^d \nu_j$ and $c_\nu=\sum_{j=1}^d c_{\nu_j}$ that are, obviously, independent of $\{\varepsilon_n\}$, and so the constants in the gradient bounds in Theorem \ref{thm_grad} for $b_n$ do not depend on $n$.

\medskip

\noindent \textit{Proof of assertion {\rm ($i$)}}. Let $\{\mathbb P_x\}_{x \in \mathbb R^d}$ be the strong Markov family of martingale solutions to \eqref{sde1} constructed in Theorem \ref{markov_thm}. Fix some $y$.
Our goal is prove the following estimate:
there exists generic $q>(d-2) \vee 2$ and $C$ such that, for all $\mathsf{g} \in \mathbf{F}_{\delta_1}$, $\delta_1<\infty$, and all $\lambda$ greater than some generic $\lambda_0$, 
\begin{equation}
\label{g_bd}
\mathbb E_{\mathbb P_y} \int_0^\infty e^{-\lambda s}|\mathsf{g}f|(\omega_s)ds \leq C\|\mathsf{g}|f|^{\frac{q}{2}}\|_2^\frac{2}{q}
\end{equation}
for all $f \in C_c$. Let $\mathsf{g}_m$ the bounded smooth regularization of $\mathsf{g}$ constructed according to Lemma \ref{fb_approx_lem}. 
Using the gradient estimate of Theorem \ref{thm_grad}(\textit{i}), after applying the Sobolev embedding theorem twice, we obtain
$$
\mathbb E_{\mathbb P_y^{n}} \int_0^\infty e^{-\lambda s}|\mathsf{g}_mf|(\omega_s)ds \leq C\|\mathsf{g}_m|f|^{\frac{q}{2}}\|_2^\frac{2}{q}, \quad n,m=1,2,\dots,
$$
where $\mathbb P_x^{n}$ is the martingale solution of the regularized SDE
$$
Y(t)=y-\int_0^t b_n(Y(s))ds + \sqrt{2}B(t), \quad t \geq 0
$$
and, by the construction of $\mathbb P_x$ in the proof of Theorem \ref{markov_thm}, $\mathbb P_x^n \rightarrow \mathbb P_x$ weakly (we pass to a subsequence of $\{b_n\}$ if necessary). Thus, we have
$$
\mathbb E_{\mathbb P_y} \int_0^\infty e^{-\lambda s}|\mathsf{g}_mf|(\omega_s)ds \leq C\|\mathsf{g}_m|f|^{\frac{q}{2}}\|_2^\frac{2}{q}, \quad m=1,2,\dots
$$
Fatou's lemma applied in $m$ now yields \eqref{g_bd} and thus ends the proof of (\textit{i}).

\medskip

\noindent \textit{Proof of assertion {\rm ($i'$)}}. Let $\{\mathbb P_y^1\}_{y \in \mathbb R^d}$, $\{\mathbb P_y^2\}_{y \in \mathbb R^d}$  be two Markov families of martingale solutions to SDE
$$
Y(t)=y-\int_0^t b(Y(s))ds + \sqrt{2}B(t), \quad t \geq 0.
$$
Fix some $y$.
By our assumption, there exists $q>(d-2) \vee 2$ such that, for all $\mathsf{g} \in \mathbf{F}_{\delta_1}$, $\delta_1<\infty$, and all $\lambda$ greater than some generic $\lambda_0$, 
\begin{equation}
\label{krylov_est0}
\mathbb E_{\mathbb P_{y}^i} \int_0^\infty e^{-\lambda s}|\mathsf{g}f|(\omega_s)ds \leq C\|\mathsf{g}|f|^{\frac{q}{2}}\|_2^\frac{2}{q}
\end{equation}
for all $f \in C_c$.
Let $v_n$ be the classical solution to equation
$$
(\lambda -\Delta + b_n\cdot \nabla)v_n=-F,
$$
where $F \in C_c(\mathbb R^d)$. We will need the weight $
\rho(x)=(1+k|x|^{2})^{-\beta}$, $k>0,
$
where constant $\beta$ is fixed greater than $\frac{d}{2}$ so that $\rho \in L^1(\mathbb R^d)$. 
By It\^{o}'s formula applied to $e^{-\lambda t}\rho v_n$, we have
$$
\mathbb E_{\mathbb P_{y}^i} [e^{-\lambda t}(\rho v_n)(\omega_t)]  =\rho (y) v_n(y) + \mathbb E_{\mathbb P_{y}^i} \int_0^t \rho e^{-\lambda s}(-\lambda + \Delta - b\cdot \nabla)v_n(\omega_s)ds - S_n,
$$
where $S_n$ is the remainder term given by
$$
S_n:=\mathbb E_{\mathbb P_{y}^i} \int_0^ t e^{-\lambda s}[-(\Delta \rho)v_n - 2 \nabla \rho \cdot \nabla v_n + b \cdot (\nabla \rho)v_n](\omega_s)ds.
$$
So,
\begin{align}
\mathbb E_{\mathbb P_{y}^i} [e^{-\lambda t}(\rho v_n)(\omega_t)] & =\rho (y) v_n(y) + \mathbb E_{\mathbb P_{y}^i} \int_0^t \rho e^{-\lambda s} F(\omega_s)ds \notag \\
& - \mathbb E_{\mathbb P_{y}^i} \int_0^t [e^{-\lambda s}\rho(b-b_n) \cdot \nabla v_n](\omega_s)ds - S_n.
\label{unique}
\end{align}

\begin{proposition}
\label{cl1}
For every $k>0$, $$\mathbb E_{\mathbb P_{y}^i} \int_0^t e^{-\lambda s}[\rho(b-b_n) \cdot \nabla v_n](\omega_s)ds \rightarrow 0$$ as $n \uparrow \infty$ uniformly in $t>0$.
\end{proposition}

\begin{proof}
We have
\begin{align*}
|\mathbb E_{\mathbb P_{y}^i} \int_0^t [e^{-\lambda s}\rho(b-b_n) \cdot \nabla v_n](\omega_s)ds| & \leq |\mathbb E_{\mathbb P_{y}^i} \int_0^\infty [e^{-\lambda s}\rho(b-b_n) \cdot \nabla v_n](\omega_s)ds| \\
& (\text{we apply \eqref{krylov_est0} with $\mathsf{g}:=\rho(b-b_n) \in \mathbf{F}_{2\delta}$}) \\
& \leq K\|\rho(b-b_n) |\nabla v_n|^{\frac{q}{2}}\|_2^{\frac{2}{q}}.
\end{align*}
In turn, for a $0<\theta<1$, we have
\begin{equation}
\label{rho_b}
\|\rho(b-b_n) |\nabla v_n|^{\frac{q}{2}}\|_2 \leq \|\rho(b-b_n) \|_2^\theta\|\rho(b-b_n) |\nabla v_n|^{\frac{q}{2(1-\theta)}}\|^{1-\theta}_2.
\end{equation}

Regarding the second multiple in the RHS of \eqref{rho_b}: we assume that $\theta$ is chosen to be sufficiently close to $0$ so that $\frac{q}{1-\theta} > (d-2) \vee 2$. Then, by $b-b_n \in \mathbf{F}_{2\delta}$, 
\begin{align*}
\|\rho(b-b_n) |\nabla v_n|^{\frac{q}{2(1-\theta)}}\|_2^2 & \leq \|(b-b_n) |\nabla v_n|^{\frac{q}{2(1-\theta)}}\|_2^2  \\
&\leq 2\delta\|\nabla|\nabla v_n|^{\frac{q}{2(1-\theta)}}\|_2^2 + 2c_\delta \||\nabla v_n|^{\frac{q}{2(1-\theta)}}\|_2^2.
\end{align*}
Hence, by the gradient estimate of Theorem \ref{thm_grad}(\textit{i}), $\sup_n\|\rho(b-b_n) |\nabla v_n|^{\frac{q}{2(1-\theta)}}\|_2^2<\infty.$

The first multiple in the RHS of \eqref{rho_b}:
\begin{align*}
\|\rho(b-b_n) \|^2_2 & \leq \langle \mathbf{1}_{B_R(0)}|b-b_n|^2\rangle + \langle (1-\mathbf{1}_{B_R(0)})\rho,\rho |b-b_n|^2\rangle  \\
& \leq \langle \mathbf{1}_{B_R(0)}|b-b_n|^2\rangle + (1+k R^2)^{-\beta} \langle \rho |b-b_n|^2 \rangle.
\end{align*}
Since $b_n \rightarrow b$ in $L^2_{\loc}$, the first integral can be made as small as needed (uniformly in $R$)  by selecting $n$ sufficiently large. In the second integral $\sup_n \langle \rho |b-b_n|^2 \rangle<\infty$, since, by $b-b_n \in \mathbf{F}_{2\delta}$,
$$
\langle \rho |b-b_n|^2 \rangle \leq 2\delta \langle (\nabla \sqrt{\rho})^2\rangle + 2c_\delta \langle \rho \rangle,
$$
so it remains to apply $|\nabla \rho| \leq \beta\sqrt{k} \rho$. At the same time,
$(1+k r^2)^{-\beta}$ can be made as small as needed by selecting $r$ sufficiently large. This completes the proof.
\end{proof}

\begin{proposition}
\label{cl2}
$S_n \rightarrow 0$ as $k \downarrow 0$ uniformly in $n$ and $t$.
\end{proposition}
\begin{proof}
Using $|\nabla \rho| \leq \beta\sqrt{k} \rho$, $|\Delta \rho| \leq \beta^2 k$, we have
$$
|S_n| \leq \sqrt{k} C\mathbb E_{\mathbb P_{y}^i} \int_0^ t [\rho |v_n| + 2\rho|\nabla v_n| + \rho|b||v_n| ](\omega_s) ds.
$$
Now we can argue as in the proof of the previous proposition, using additionally $\|v_n\|_\infty \leq \lambda^{-1}\|F\|_\infty$, to show that $\sup_n\mathbb E_{\mathbb P_{y}^i} \int_0^ t [\rho |v_n| + 2\rho|\nabla v_n| + \rho|b||v_n| ](\omega_s) ds<\infty$. In fact, in this case the proof is easier since none of the terms contains simultaneously $b$ and $\nabla v_n$. Selecting $k$ sufficiently small, we can make $S_n$ as small as needed.
\end{proof}

We now complete the proof of assertion {\rm ($i'$)}.
Let us note that, for every $k>0$,
$$
\mathbb E_{\mathbb P_{y}^i} [e^{-\lambda t} \rho v_n(\omega_t)] \rightarrow 0 \quad \text{ as } t \rightarrow \infty \text{ uniformly in $n$}.
$$
Indeed, $\|v_n\|_\infty \leq \lambda^{-1}\|F\|_\infty$, so
$
|\mathbb E_{\mathbb P_{y}^i} [e^{-\lambda t}\rho v_n(\omega_t)] | \leq \lambda^{-1}e^{-\lambda t},
$
which yields the required. Combining this result with Propositions \ref{cl1} and \ref{cl2}, and taking into account that, by Theorem \ref{thm1}(\textit{iv}), $\{v_n\}$ converge uniformly as $n\rightarrow \infty$ to a continuous function $v$, we obtain from \eqref{unique} upon taking $n \rightarrow \infty$ and then taking $k \downarrow 0$:
$$
0  = v(y) + \mathbb E_{\mathbb P_{y}^i} \int_0^\infty e^{-\lambda s} F(\omega_s)ds, \quad i=1,2.
$$
Taking into account the continuity of $F$ and $\omega$, and invoking the uniqueness of Laplace transform,
we obtain that $\mathbb E_{\mathbb P_{y}^1} F(\omega_t)=\mathbb E_{\mathbb P_{y}^2} F(\omega_t)$ for all $F \in C_c$, $t>0.$ We deduce from here that the one-dimensional distributions of $\mathbb P_y^1$ and $\mathbb P_y^2$ coincide. 
Since we are dealing with Markov families of probability measures, we conclude that $\mathbb P_y^1=\mathbb P_y^2$ for every $y \in \mathbb R^d$.

\bigskip

\noindent \textit{Proof of assertion {\rm ($ii$)}}. The proof follows closely the proof of (\textit{i}), but uses the gradient estimate of Theorem \ref{thm_grad}(\textit{i}) for $q>(d-2) \vee 2$ chosen closely to $(d-2) \vee 2$. In fact, this proof is easier since we no longer need to take care of extra form-bounded vector fields $\mathsf{g}$ as in (\textit{i}).

\medskip

\noindent \textit{Proof of assertion {\rm ($ii''$)}}. 
We modify the previous proof of ($i'$). By our assumption, 
\begin{equation}
\label{kr_0}
\mathbb E_{\mathbb P_{y}^i} \int_0^\infty e^{-\lambda s}|f|(\omega_s)ds \leq C\|f\|_{\frac{d}{2-\varepsilon } \wedge \frac{2}{1-\varepsilon}}, \quad \forall f \in C_c, \quad \lambda>\lambda_0.
\end{equation}
The analogue of Proposition \ref{cl1} is proved as follows. Clearly, hypothesis 
$$
(1+|x|^{-2})^{-\beta}|b|^{\frac{d}{2-\varepsilon_1} \vee \frac{2}{1-\varepsilon_1}} \in L^1, \quad \varepsilon_1 \in ]\varepsilon,1[
$$
implies that, for any $k>0$, $\rho |b|^{\frac{d}{2-\varepsilon_1} \vee \frac{2}{1-\varepsilon_1}} \in L^1$.
We have
\begin{align}
|\mathbb E_{\mathbb P_{y}^i} \int_0^t [e^{-\lambda s}\rho(b-b_n) \cdot \nabla v_n](\omega_s)ds| & \leq |\mathbb E_{\mathbb P_{y}^i} \int_0^\infty [e^{-\lambda s}\rho(b-b_n) \cdot \nabla v_n](\omega_s)ds| \notag \\
& (\text{we apply \eqref{kr_0} using Fatou's lemma}) \notag \\
& \leq K\|\rho(b-b_n) \cdot \nabla v_n\|_{r} \quad r:=\frac{d}{2-\varepsilon } \wedge \frac{2}{1-\varepsilon} \notag \\ 
& \leq K\|\rho(b-b_n)\|_{s'} \|\nabla v_n\|_{s}, \quad \frac{1}{s'}+\frac{1}{s}=\frac{1}{r}, \label{u}
\end{align}
where $s'= \frac{d}{2-\varepsilon_1} \vee \frac{2}{1-\varepsilon_1}$ and $s=\frac{q_*d}{d-2}$, where $q_*$ was defined in assertion {\rm ($ii''$)} of Theorem \ref{unique_thm} that we are proving. 
Theorem \ref{thm_grad}(\textit{ii}), which applies by our assumptions on $\delta$, $\nu$ and $q_*$ in the end of assertion {\rm ($ii''$)}, and the Sobolev embedding theorem, yield 
$$
\sup_n \|\nabla v_n\|_{\frac{q_*d}{d-2}} <\infty.
$$
Therefore,  the second multiple in the RHS of \eqref{u} is uniformly (in $n$) bounded.

In turn, for every fixed $k$, the first multiple in the RHS of \eqref{u} tends to zero as $n \rightarrow \infty$. Indeed, since $0<\rho \leq 1$, we have $$\sup_n\|\rho^{s'} b_n^{s'}\|_1 \leq \sup_n\|\rho b_n^{s'}\|_1<\infty,$$ where the finiteness is seen, after integrating by parts, from $E_{\varepsilon_n} \rho \leq C \rho$ with constant $C$ independent of $n$ (here we simply use the fact that the Friedrichs mollifier is a convolution with a function having compact support) and our hypothesis $\|\rho |b|^{s'}\|_1<\infty$. Now, we represent
\begin{align*}
\|\rho(b-b_n)\|_{s'} &= \|\mathbf{1}_{B_R(0)}(b-b_n)\|_{s'} + \|(1-\mathbf{1}_{B_R(0)})\rho(b-b_n)\|_{s'} \\
& \leq \|\mathbf{1}_{B_R(0)}(b-b_n)\|_{s'}+(1+k R^2)^{-\beta (s'-1)} (\langle \rho b^{s'}\rangle  + \langle \rho b_n^{s'}\rangle ).
\end{align*}
The second term can be made as small as needed by selecting $R$ sufficiently large (uniformly in $n$). Then, for $R$ thus fixed, the first term can be made as small as needed by selecting $n$ sufficiently large, since $b_n \rightarrow b$ in $L^{s'}_{\loc}$ by the properties of Friedrichs mollifier.

Arguing as above, we prove $\sup_n\mathbb E_{\mathbb P_{y}^i} \int_0^ t [\rho |v_n| + 2\rho|\nabla v_n| + \rho|b_n||v_n| ](\omega_s) ds<\infty$, and hence have the analogue of Proposition \ref{cl2}. 

The rest of the proof of ($ii'$) repeats the proof of ($i'$).

\bigskip

\appendix

\bigskip

\section{A desingularization theorem from \cite{KiSSz}}

\label{app_desing}

Let $X$ be a locally compact topological space, and $\mu$ a $\sigma$-finite Borel measure on $X$. In what follows, $L^r=L^r(X,\mu)$ ($1 \leq r \leq \infty$). Let $j>1$, put $j':=\frac{j}{j-1}$.

Let $\Lambda$ be the generator of a strongly continuous semigroup $e^{-t\Lambda}$ on $L^r$ for some $r>1$,
such that for some constants $c$, $j>1$, for all $t>0$,
\begin{equation}
\label{S1}
\tag{$S_1$}
\|e^{-t\Lambda}\|_{r \rightarrow \infty} \leq ct^{-\frac{j'}{r}}.
\end{equation}
We consider a family of weights $\varphi=\{\varphi_s\}_{s>0}$ in $X$ such that 
\begin{equation}
\label{S2}
\tag{$S_2$}
0 \leq \varphi_s, \frac{1}{\varphi_s} \in L^1_{\loc}(X,\mu) \quad \text{for all $s>0$},
\end{equation}
\begin{equation}
\label{S3}
\tag{$S_3$}
\inf_{s>0,x \in X} \varphi_s(x)  \geq c_0>0.
\end{equation}

\begin{theorem}
\label{thm_desing1}
Assume that conditions {\rm($S_1$)}\,-\,{\rm($S_3$)} hold and
there exists constant $c_1$, independent of $s$, such that, for all $0<t \leq s$,
\begin{equation}
\label{S4}
\tag{$S_4$}
\|\varphi_s e^{-t\Lambda}\varphi_s^{-1}f\|_{1} \leq c_1\|f\|_{1}, \quad f \in L^1 \cap L^\infty.
\end{equation}
Then, for each $t>0$, $e^{-t\Lambda}$ is an integral operator, and there is a constant $C=C(j,c_1,c_0)$ such that, up to change of $e^{-t\Lambda}(x,y)$ on a measure zero set, the weighted Nash initial estimate
\begin{equation}
\label{nie_a}
|e^{-t\Lambda}(x,y)|\leq Ct^{-j^\prime}  \varphi_t(y)
\end{equation}
is valid for $\mu$ a.e. $x,y \in X$.
\end{theorem}

For the sake of keeping the paper self-contained, we reproduce here the proof of Theorem \ref{thm_desing1} from \cite{KiSSz}.

\begin{proof}[Proof of Theorem \ref{thm_desing1}]

1.~We will use a weighted variant of the Coulhon-Raynaud extrapolation theorem. Put $$0 \leq \psi \in L^1+L^\infty,  \quad \|f\|_{p,\sqrt{\psi}}:=\langle |f|^p \psi \rangle^{1/p}.$$ \textit{
Let $U^{t,\theta}$ be a two-parameter family of operators
\[U^{t,\theta}f = U^{t,\tau}U^{\tau,\theta}f,  \qquad f \in L^1 \cap L^\infty, \quad 0 \leq \theta < \tau < t \leq \infty.
\]
If for some $1 \leq p < q < r \leq \infty$, $\nu>0$
\begin{align*}
\| U^{t,\theta} f \|_p & \leq M_1 \| f \|_{p,\sqrt{\psi}},\\
 \| U^{t,\theta} f \|_r & \leq M_2 (t-\theta)^{-\nu} \|  f \|_q
\end{align*}
for all $(t,\theta)$ and $f \in L^1 \cap L^\infty,$ then
\begin{equation}
\label{cr}
\| U^{t,\theta} f \|_r \leq M (t-\theta)^{-\nu/(1-\beta)} \| f \|_{p,\sqrt{\psi}} ,
\end{equation}
where $\beta = \frac{r}{q}\frac{q-p}{r-p}$ and $M = 2^{\nu/(1-\beta)^2} M_1 M_2^{1/(1-\beta)}.$}
Here is the proof of \eqref{cr} for reader's convenience. Put $t_\theta:=\frac{t+\theta}{2}$. We have
\begin{align*}
\| U^{t, \theta} f \|_r & \leq M_2 (t-t_\theta)^{-\nu} \| U^{t_\theta,\theta} f \|_q \\
& \leq M_2 (t-t_\theta)^{-\nu} \| U^{t_\theta,\theta} f \|_r^\beta \;\| U^{t_\theta,\theta} f \|_p^{1-\beta} \\
& \leq M_2 M_1^{1-\beta} (t-t_\theta)^{-\nu} \| U^{t_\theta,\theta} f \|_r^\beta \;\| f \|_{p,\sqrt{\psi}}^{1-\beta},
\end{align*}
and hence
\[
(t-\theta)^{\nu/(1-\beta)} \| U^{t,\theta} f \|_r/\| f \|_{p,\sqrt{\psi}} \leq M_2 M_1^{1-\beta} 2^{\nu/(1-\beta)} \big [(t - \theta)^{\nu/(1-\beta)} \| U^{t_\theta,\theta} f \|_r\;/\| f \|_{p,\sqrt{\psi}} \big ]^\beta.
\]
Setting $R_{2 T}: = \sup_{t-\theta \in ]0,T]} \big [ (t-\theta)^{\nu/(1-\beta)} \| U^{t,\theta} f \|_r/\| f \|_{p,\sqrt{\psi}} \big ],$ we obtain from the last inequality that $R_{2 T} \leq M^{1-\beta} (R_T)^\beta.$ But $R_T \leq R_{2T}$, and so $R_{2T} \leq M.$
This gives us \eqref{cr}.

\medskip

2.~We are in position to complete the proof of Theorem \ref{thm_desing1}.
By \eqref{S4} and \eqref{S3},
\begin{align*}
\|e^{-t\Lambda}h\|_{1} &\leq c_0^{-1} \|\varphi_s e^{-t\Lambda} \varphi_s^{-1} \varphi_s h \|_1  \\
& \leq c_0^{-1}c_1 \|h\|_{1,\sqrt{\varphi_s}}, \qquad  h \in L^\infty_{\rm com}.
\end{align*}
The latter, \eqref{S1} and the Coulhon-Raynaud extrapolation theorem with $\psi:=\varphi_s$ yield
$$
\|e^{-t\Lambda}f\|_{\infty} \leq Mt^{-j'}\|\varphi_s f\|_1, \quad 0<t \leq s, \quad f \in  L^\infty_{c}.
$$
Note that ($S_1$) verifies the assumptions of the Dunford-Pettis theorem, which yields that $e^{-t\Lambda}$ is an integral operator.
Therefore, taking $s=t$ in the previous estimate, we obtain \eqref{nie_a}.
\end{proof}

\end{document}